\documentclass[a4paper,11pt,fleqn]{article}%,draft

\usepackage{amsmath,amssymb,amsthm} %,backref
\usepackage[mathscr]{eucal}
\usepackage{url}
\usepackage{hyperref}
\hypersetup{colorlinks, linkcolor=blue, citecolor=blue, urlcolor=blue}

\allowdisplaybreaks
\newcommand{\todo}[1][\null]{\ensuremath{\clubsuit}}

\newcommand{\noprint}[1]{}
\newcommand{\checked}[1][\null]{\ensuremath{\boldsymbol{\surd}}}

\newcommand{\p}{\partial}
\newcommand{\ord}{\mathop{\rm ord}\nolimits}
\newcommand{\lsemioplus}{\mathbin{\mbox{$\lefteqn{\hspace{.77ex}\rule{.4pt}{1.2ex}}{\in}$}}}
\newcommand{\ri}{\mathfrak r}

\newtheorem{theorem}{Theorem}
\newtheorem{lemma}[theorem]{Lemma}
\newtheorem{corollary}[theorem]{Corollary}
\newtheorem{proposition}[theorem]{Proposition}
{\theoremstyle{definition}
\newtheorem{remark}[theorem]{Remark}
}

\begin{document}

\par\noindent {\LARGE\bf
%Extended symmetry analysis\\ of isothermal no-slip drift flux model. II
Generalized symmetries, conservation laws\\ and Hamiltonian structures\\ of an isothermal no-slip drift flux model
\par}

\vspace{4mm}\par\noindent {\large Stanislav Opanasenko$^{\dag\ddag}$, Alexander Bihlo$^\dag$, Roman O.\ Popovych$^{\ddag\S\natural}$ and Artur Sergyeyev$^\natural$
}

\vspace{3mm}\par\noindent {\it
$^{\dag}$~Department of Mathematics and Statistics, Memorial University of Newfoundland,\\
$\phantom{^{\dag}}$~St.\ John's (NL) A1C 5S7, Canada\par
}
\vspace{2mm}\par\noindent {\it
$^\ddag$~Institute of Mathematics of NAS of Ukraine, 3 Tereshchenkivska Str., 01024 Kyiv, Ukraine\par
}
\vspace{2mm}\par\noindent {\it
$^{\S}$~Fakult\"at f\"ur Mathematik, Universit\"at Wien, Oskar-Morgenstern-Platz 1, 1090 Wien, Austria%
}
\vspace{2mm}\par\noindent {\it
$^{\natural}$~Mathematical Institute, Silesian University in Opava, Na Rybn\'\i{}\v{c}ku 1, 746 01 Opava,\\
$\phantom{^{\natural}}$~Czech Republic
}

\vspace{2mm}\par\noindent
\textup{E-mail:} sopanasenko@mun.ca, abihlo@mun.ca, rop@imath.kiev.ua, artur.sergyeyev@math.slu.cz\!
\par

\vspace{6mm}\par\noindent\hspace*{9mm}\parbox{142mm}{\small
We study the hydrodynamic-type system of differential equations modeling isothermal no-slip drift flux.
Using the facts that the system is partially coupled and its subsystem reduces to the (1+1)-dimensional Klein--Gordon equation,
we exhaustively describe generalized symmetries, cosymmetries and local conservation laws of this system.
A generating set of local conservation laws under the action of generalized symmetries
is proved to consist of two zeroth-order conservation laws.
The subspace of translation-invariant conservation laws is singled out from the entire space of local conservation laws.
We also find broad families of local recursion operators and a nonlocal recursion operator,
and construct an infinite family of Hamiltonian structures involving an arbitrary function of a single argument.
For each of the constructed Hamiltonian operators, we obtain the associated algebra of Hamiltonian symmetries.
}\par%\vspace{1mm}

\noprint{


76-XX   Fluid mechanics {For general continuum mechanics, see 74Axx, or other parts of 74-XX}
 76Mxx  Basic methods in fluid mechanics [See also 65-XX]
  76M60 Symmetry analysis, Lie group and algebra methods

37-XX	Dynamical systems and ergodic theory
 37Kxx  Infinite-dimensional Hamiltonian systems [See also 35Axx, 35Qxx]
  37K05 Hamiltonian structures, symmetries, variational principles, conservation laws

35-XX   Partial differential equations
 35Bxx  Qualitative properties of solutions
  35B06 Symmetries, invariants, etc.

}

\section{Introduction}\label{sec:IDFMIntroduction}

The drift flux model introduced in~\cite{ZuberFindlay1965} is a simplified model of a well-known two-phase flow phenomenon~\cite{IshiiHibiki2006,YadigarogluHewitt2018}.
The former was thoroughly studied in~\cite{EvjeFjelde2002,EvjeFlatten2005,EvjeFlatten2007,EvjeKarlsen2008},
where several submodels easier to tackle but still real-world applicable were suggested.
Amongst them is the isothermal no-slip drift flux model given by the system
\begin{gather*}
\rho^1_t+u\rho^1_x+u_x\rho^1=0,\\
\rho^2_t+u\rho^2_x+u_x\rho^2=0,\\
(\rho^1+\rho^2)(u_t+uu_x)+a^2(\rho^1_x+\rho^2_x)=0,
\end{gather*}
which we denote by~$\mathcal S$.
This model describes the mixing motion of liquids (or gases) rather than their individual phases.
Here $u=u(t,x)$ is the common velocity,
$\rho^1=\rho^1(t,x)$ and $\rho^2=\rho^2(t,x)$ are the densities of the liquids,
and the constant parameter~$a$ can be set to~1 by scaling $(x,u)$ with $a$.
Any constraint meaning that $\rho^1$ and $\rho^2$ are proportional, e.g., $\rho^2=\rho^1$ or $\rho^2=0$,
reduces~$\mathcal S$ to the system~$\tilde{\mathcal S_0}$
describing one-dimensional isentropic gas flows with constant sound speed,
cf. the system (3)--(4) with $\nu=0$ in~\cite[Section~2.2.7]{RozhdestvenskiiJanenko1983}.
The system~$\mathcal S$ is a diagonalizable hydrodynamic-type system
since it can be equivalently rewritten as
\begin{subequations}\label{eq:IDFMDiagonalizedSystem}
\begin{gather}
\ri^1_t+(\ri^1+\ri^2+1)\ri^1_x=0,\label{eq:IDFMDEq1}\\
\ri^2_t+(\ri^1+\ri^2-1)\ri^2_x=0,\label{eq:IDFMDEq2}\\
\ri^3_t+(\ri^1+\ri^2)\ri^3_x=0\label{eq:IDFMDEq3}
\end{gather}
\end{subequations}
by changing the dependent variables $(u,\rho^1,\rho^2)$ to the Riemann invariants $(\ri^1,\ri^2,\ri^3)$ via
\[
\ri^1=\frac{u+\ln(\rho^1+\rho^2)}2,\quad \ri^2=\frac{u-\ln(\rho^1+\rho^2)}2,\quad \ri^3=\frac{\rho^2}{\rho^1}.
\]

The corresponding characteristic velocities
\begin{gather}\label{eq:IDFMCharacteristicSpeeds}
V^1=\ri^1+\ri^2+1,\quad V^2=\ri^1+\ri^2-1,\quad V^3=\ri^1+\ri^2
\end{gather}
are distinct, meaning that the system~$\mathcal S$ is strictly hyperbolic.
Besides, the characteristic velocities satisfy the system
\[
\p_{\ri^i}\frac{V^k_{\ri^j}}{V^j-V^k}=\p_{\ri^j}\frac{V^k_{\ri^i}}{V^i-V^k}\quad\text{ for all}\quad i,j,k\in\{1,2,3\}\quad\text{with}\quad i,j\ne k.
\]
Thus, the system~$\mathcal S$ is semi-Hamiltonian and, since $V^3_{\ri^3}=0$, it is not genuinely nonlinear with respect to~$\ri^3$;
see \cite{Tsarev1991} for related definitions.
The system~$\mathcal S$ is also partially coupled.
The essential subsystem~$\mathcal S_0$ consisting of the equations~\eqref{eq:IDFMDEq1}--\eqref{eq:IDFMDEq2}
coincides with the diagonalized form of the system~$\tilde{\mathcal S_0}$
\mbox{\cite[Section~2.2.7, Eq.~(16)]{RozhdestvenskiiJanenko1983}.}

Hydrodynamic-type systems are extensively studied in the literature in view of their various physical applications in fluid
mechanics, acoustics and gas and shock dynamics~\cite{RozhdestvenskiiJanenko1983,Whitham1999} and
rich differential geometry~\cite{DubrovinNovikov1983,DubrovinNovikov1989,Tsarev1985,Tsarev1991}.
See~\cite{BlaszakSergyeyev2009,BurdeSergyeyev2013,Ferapontov1991,Ferapontov1995,GrundlandHariton2008,GrundlandHuard2007,Pavlov2003,Sergyeyev2017,Sergyeyev2018}
and references therein for an assortment of examples.

In view of the above properties, the system~$\mathcal S$ can be integrated in an implicit form.
In~\cite{OpanasenkoBihloPopovychSergyeyev2020} for this system
we expressed the general solution in terms of the general solution
of the (1+1)-dimensional Klein--Gordon equation
using the generalized hodograph transformation~\cite{Tsarev1985}
and described the entire set of local solutions via the linearization of the subsystem~$\mathcal S_0$
to the same equation.
Since the practical use of the derived representations for solutions of~$\mathcal S$ is limited
because of their implicit form and complicated structure,
in~\cite{OpanasenkoBihloPopovychSergyeyev2020}
we also began the extended classical symmetry analysis of the system~$\mathcal S$.
In particular, for this system we constructed
the maximal Lie invariance algebra~$\mathfrak g$,
the algebra of generalized symmetries of order not greater than one,
the complete point symmetry group and group-invariant solutions.
Thus, the algebra~$\mathfrak g$ is spanned by the vector fields
\begin{gather*}
\hat{\mathcal D}=t\p_t+x\p_x,\ \
\hat{\mathcal G}_1=t\p_x+\p_{\ri^1},\ \
\hat{\mathcal G}_2=\p_{\ri^1}-\p_{\ri^2},\ \
\hat{\mathcal P}^t=\p_t,\ \
\hat{\mathcal P}^x=\p_x,\ \
\hat{\mathcal W}(\Omega)=\Omega(\ri^3)\p_{\ri^3},
\end{gather*}
where $\Omega$ runs through the set of smooth functions of $\ri^3$.
The maximal Lie invariance algebra~$\mathfrak g_0$ of the essential subsystem~$\mathcal S_0$
is wider than the projection of the algebra~$\mathfrak g$ to the space with the coordinates $(t,x,\ri^1,\ri^2)$
and is spanned by the vector fields
\begin{gather*}
%\begin{split}\label{eq:IDFMSymmetryOperatorsSystem2}
\breve{\mathcal D}=t\p_t+x\p_x,\quad \breve{\mathcal G_1}=t\p_x+\p_{\ri^1},\quad
\breve{\mathcal G_2}=\p_{\ri^1}-\p_{\ri^2},
\quad \breve{\mathcal P}(\tau^0,\xi^0)=\tau(\ri^1,\ri^2)\p_t+\xi(\ri^1,\ri^2)\p_x,\\
\breve{\mathcal J}=
\left(\frac12x-t(\ri^1+\ri^2)\right)\p_t+t\left(\ri^1-\ri^2-\frac12(\ri^1+\ri^2)^2+\frac12\right)\p_x+\ri^1\p_{\ri^1}-\ri^2\p_{\ri^2},
%\end{split}
\end{gather*}
where $(\tau,\xi)$ is a tuple of smooth functions of~$(\ri^1,\ri^2)$, running through the solution set of the system $\xi_{\ri^1}=V^2\tau_{\ri^1}$, $\xi_{\ri^2}=V^1\tau_{\ri^2}$.
In~\cite{OpanasenkoBihloPopovychSergyeyev2020}, for the system~$\mathcal S$ we also found
the zeroth-order local conservation laws using the direct method
and, following~\cite{Doyle1994}, constructed the entire space of first-order conservation laws with $(t,x)$-translation-invariant densities of
and a subspace of $(t,x)$-translation-invariant conservation laws of arbitrarily high order.
Building on the description of the algebra of generalized symmetries of order not greater than one,
we obtained an infinite-dimensional subspace of generalized symmetries of arbitrarily high order for~$\mathcal S$.
(In the present paper we show that this subspace is an ideal in the entire algebra of generalized symmetries of the system~$\mathcal S$.)
\looseness=-1

\looseness=-1
At the same time, the system~$\mathcal S$ possesses two properties
that allow us to exhaustively describe the entire spaces of generalized symmetries, cosymmetries and local conservation laws
(see \cite{KrasilshchikVerbovetskyVitolo2017} for definitions).
Firstly, the system is partially coupled with the essential subsystem~$\mathcal S_0$
being linearizable through the rank-two hodograph transformation to the (1+1)-dimensional Klein--Gordon equation,
which was thoroughly studied in~\cite{OpanasenkoPopovych2018}
from the point of view of generalized and variational symmetries and local conservation laws.
Secondly, in addition to being not genuinely nonlinear with respect to~$\ri^3$,
the system~$\mathcal S$ is decoupled with respect to~$\ri^3$,
and the third equation of~$\mathcal S$ is linear in~$\ri^3$. %if $\ri^1$ and~$\ri^2$ are fixed.
Thus, speaking of the degeneracy of the system~$\mathcal S$,
we mean both its linear degeneracy and decoupling with respect to~$\ri^3$.
Due to the dual nature of this degeneracy, the system~$\mathcal S$
admits not only an infinite number of linearly independent conservation laws of arbitrarily high order,
that are related to the degeneracy, cf.~\cite{Doyle1994,Sheftel1994b},
but also similar generalized symmetries.

\looseness=-1
Substantially generalizing results of~\cite{OpanasenkoBihloPopovychSergyeyev2020},
in the present paper we comprehensively study generalized symmetries, cosymmetries and local conservation laws
of the system~$\mathcal S$.
This includes both a description of the corresponding spaces and their interrelations, which are described
in terms of recursion operators and Noether and Hamiltonian operators.
Our \textit{modus operandi} to study the system~$\mathcal S$ is
to select appropriate symmetry-like objects of the Klein--Gordon equation
(generalized symmetries, cosymmetries and conservation laws), %and nonlocal symmetries
to find their counterparts for the system~$\mathcal S$
and to complement these counterparts with the objects of the same kind that are related to the degeneracy of the system.
Then we prove that the constructed objects span the entire spaces of objects of the corresponding kinds for the system~$\mathcal S$.
As a result, we obtain one more example, in addition to a few ones existing in the literature, cf.\ \cite{OpanasenkoPopovych2018},
where generalized symmetries and local conservation laws are exhaustively described for a model
arising in real-world applications and possessing symmetry-like objects of arbitrarily high order.

The structure of this paper is as follows.
In Section~\ref{sec:IDFMSolutionsViaEssentialSubsystem} we reduce the system~$\mathcal S$ to the (1+1)-dimensional Klein--Gordon equation
and show that any regular solution of the former is expressed in terms of solutions of the latter.
In Section~\ref{sec:IDFMPreliminaries} we lay out notations and auxiliary results to be used throughout
the remainder of the paper.
It is proved in Section~\ref{sec:IDFMGenSyms} that the algebra of reduced generalized symmetries of the system~$\mathcal S$
is a (non-direct) sum of an ideal related to the degeneracy of~$\mathcal S$
and consisting of generalized vector fields with zero $\ri^1$- and $\ri^2$-components
and of a subalgebra stemming from generalized symmetries of the Klein--Gordon equation.
At the same time, not all generalized symmetries of the Klein--Gordon equation have counterparts among those of the system~$\mathcal S$,
and we solve the problem on selecting appropriate elements of the algebra of generalized symmetries of the Klein--Gordon equation.
This differs from cosymmetries and conservation laws of~$\mathcal S$,
for which there are injections from the corresponding spaces for the Klein--Gordon equation to those for the system~$\mathcal S$,
see Sections~\ref{sec:IDFMCosyms} and~\ref{sec:IDFM:CLs}, respectively.
The space of conservation laws of~$\mathcal S$ is proved to be generated, under the action of generalized symmetries of~$\mathcal S$,
by two zeroth-order conservation laws.
We also find the space of conservation-law characteristics of~$\mathcal S$.
The knowledge of them helps us to single out the conservation laws of orders zero and one as well as the $(t,x)$-translation-invariant ones.
In Section~\ref{sec:IDFMHamiltonianStructure} we construct a family of compatible hydrodynamic-type Hamiltonian operators
for the system~$\mathcal S$, parameterized by an arbitrary function of the degenerate Riemann invariant~$\ri^3$.
For each of these operators, we find the space of its distinguished (Casimir) functionals
and the associated algebra of Hamiltonian symmetries.
The partition of symmetry-like objects in accordance with the above two properties of the system~$\mathcal S$
also manifests itself in recursion operators studied in Section~\ref{sec:IDFMRecursionOperator}.
The independent local recursion operators constructed turn out either to nontrivially act
in the subspace of generalized symmetries arising due to the degeneracy of~$\mathcal S$
and annihilate the counterparts of generalized symmetries of the Klein--Gordon equation or vice versa.
A nonlocal recursion operator is also found.
Section~\ref{sec:IDFMConclusions} is left for the conclusions,
where we underline the nontrivial features encountered
in the course of the study of the system~$\mathcal S$ in the present paper
and discuss further problems to be considered for this system
within the framework of symmetry analysis of differential equations.
%\looseness=-1

\section{Solution through linearization of the essential subsystem}\label{sec:IDFMSolutionsViaEssentialSubsystem}
%\section{Partial linearization}\label{sec:IDFMPartialLinearization}

Using the facts that the system~$\mathcal S$ is partially coupled and the subsystem~$\mathcal S_0$ can be linearized,
we construct an implicit representation of the general solution
for the diagonalized form~\eqref{eq:IDFMDiagonalizedSystem} of the system~$\mathcal S$
in terms of the general solution of the (1+1)-dimensional Klein--Gordon equation;
cf.~\cite[Section~8]{OpanasenkoBihloPopovychSergyeyev2020}.
%
%\looseness=-1
At first, we reduce the system~\eqref{eq:IDFMDiagonalizedSystem}
by a point transformation to a system containing the (1+1)-dimensional Klein--Gordon equation.
It is convenient to derive this transformation as a chain of simpler point transformations.
We begin with the rank-two hodograph transformation, where
%$y=\ri^1/2$, $z=-\ri^2/2$ are the new independent variables and
%$p=t$, $\hat q=x$ and $s=\ri^3$ are the new dependent variables.
\begin{gather*}
y=\ri^1/2,\quad z=-\ri^2/2\quad \mbox{are the new independent variables and}\\
p=t,\quad \hat q=x,\quad s=\ri^3\quad \mbox{are the new dependent variables}.
\end{gather*}
(For convenience of the presentation, we compose the hodograph transformation with scaling of~$\ri^1$ and~$\ri^2$.)
This transformation maps the system~\eqref{eq:IDFMDiagonalizedSystem} to the system
\begin{subequations}\label{eq:IDFM:DiagSystemTranformedByHodograph}
\begin{gather}
\hat q_z-(2y-2z+1)p_z=0,\label{eq:IDFM:DiagSystemTranformedByHodographA}\\
\hat q_y-(2y-2z-1)p_y=0,\label{eq:IDFM:DiagSystemTranformedByHodographB}\\
s_yp_z+s_zp_y=0.\label{eq:IDFM:DiagSystemTranformedByHodographC}
\end{gather}
\end{subequations}
After representing the equation~\eqref{eq:IDFM:DiagSystemTranformedByHodographA}
in the form $\big(\hat q-(2y-2z+1)p\big)_z-2p=0$,
it becomes natural to make the change $\check q=\hat q-(2y-2z+1)p$ of~$\hat q$.
Then the equations~\eqref{eq:IDFM:DiagSystemTranformedByHodographA} and~\eqref{eq:IDFM:DiagSystemTranformedByHodographB}
take the form $p=\check q_z/2$ and $\check q_y+2p_y+2p=0$, respectively.
Excluding~$p$ from the second equation in view of the first one, we obtain
the second-order linear partial differential equation $\check q_{yz}+\check q_y+\check q_z=0$ in~$\check q$,
which reduces by the change~$q={\rm e}^{y+z}\check q$ of~$\check q$ to
the (1+1)-dimensional Klein--Gordon equation for~$q$ in light-cone variables, $q_{yz}=q$.
Carrying out this chain of two transformations in the whole system~\eqref{eq:IDFM:DiagSystemTranformedByHodograph},
we obtain the system~$\mathcal K$, which reads
\begin{subequations}\label{eq:IDFMSupersystem}
\begin{gather}\label{eq:IDFMSupersystemA}%{eq:IDFMKleinGordonEquation}
q_{yz}=q,
\\\label{eq:IDFMSupersystemB}
K^1s_y=K^2s_z, \quad\mbox{where}\quad
K^1:=q_{zz}-2q_z+q,\quad
K^2:=q_y+q_z-2q.
\end{gather}
\end{subequations}
We have $K^1=(\mathrm D_z-1)^2q$ and, on solutions of~\eqref{eq:IDFMSupersystemA},
$K^2=-(\mathrm D_y-1)(\mathrm D_z-1)q$,
$\mathrm D_yK^1=K^2$ and $\mathrm D_zK^2=K^1$.
Here~$\mathrm D_y$ and~$\mathrm D_z$ are the total derivative operators
with respect to~$y$ and~$z$, respectively.
We exclude~$p$ from the system~\eqref{eq:IDFMSupersystem} in view of the equation
\begin{gather}\label{eq:IDFMSupersystemC}
p=\frac12{\rm e}^{-y-z}(q_z-q)
\end{gather}
as well as we neglect this equation itself.
The composition of the above three transformations is the transformation
\begin{gather}\label{eq:IDFMTransReducingToKGEq}
\mathcal T\colon\quad
y=\frac{\ri^1}2,\quad
z=-\frac{\ri^2}2,\quad
p=t,\quad
q={\rm e}^{(\ri^1-\ri^2)/2}\big(x-(\ri^1+\ri^2+1)t\big),\quad
s=\ri^3.
\end{gather}
Therefore, to make the inverse transition from the system~\eqref{eq:IDFMSupersystem}
to the system~\eqref{eq:IDFMDiagonalizedSystem},
we should attach the equation~\eqref{eq:IDFMSupersystemC} to the system~\eqref{eq:IDFMSupersystem},
thus extending the tuple of dependent variables $(q,s)$ by~$p$,
and carry out the inverse to the transformation~\eqref{eq:IDFMTransReducingToKGEq},
\begin{gather}\label{eq:IDFMInverseToTransReducingToKGEq}
\hat{\mathcal T}\colon\quad
t=p,\quad
x={\rm e}^{-y-z}q+(2y-2z+1)p,\quad
\ri^1=2y,\quad
\ri^2=-2z,\quad
\ri^3=s.
\end{gather}

It is convenient to collect the expressions for low-order derivatives of~$p$ and~$q$
and for their combinations in terms of the old variables in view of the system~\eqref{eq:IDFMDiagonalizedSystem},
which will be needed below:
\begin{gather*}
p_y=-\frac1{\ri^1_x},\quad
p_z=-\frac1{\ri^2_x},\quad
K^1=-\frac2{\ri^2_x}{\rm e}^{(\ri^1-\ri^2)/2},\quad
K^2= \frac2{\ri^1_x}{\rm e}^{(\ri^1-\ri^2)/2},\quad
\frac{s_y}{K^2}={\rm e}^{(\ri^1-\ri^2)/2}\frac{\ri^3_x}2,
\\
q_y={\rm e}^{(\ri^1-\ri^2)/2}\left(\frac2{\ri^1_x}+x-V^1t-2t\right),\quad
q_z={\rm e}^{(\ri^1-\ri^2)/2}(x-V^2t),
\\
q_{zz}={\rm e}^{(\ri^1-\ri^2)/2}\left(-\frac2{\ri^2_x}+x-V^2t+2t\right).
\end{gather*}

Following the procedure analogous to that in~\cite{OpanasenkoBihloPopovychSergyeyev2020},
we find the complete set of local solutions of the system~\eqref{eq:IDFMDiagonalizedSystem}
via the linearization of the subsystem~\eqref{eq:IDFMDEq1}--\eqref{eq:IDFMDEq2}.

We are allowed to make the point transformation~\eqref{eq:IDFMTransReducingToKGEq}
if and only if the nondegeneracy condition $\ri^1_t\ri^2_x-\ri^1_x\ri^2_t\neq0$ holds,
which is equivalent, on solutions of~\eqref{eq:IDFMDiagonalizedSystem}, to $\ri^1_x\ri^2_x\neq0$.
Therefore, $\ri^1_t\ri^2_t\neq0$ as well,
and thus both Riemann invariants~$\ri^1$ and~$\ri^2$ are not constants.
In this case, we introduce the ``pseudopotential''~$\Psi$
defined by the potential system $\Psi_y=q-\Psi$, $\Psi_z=q_z-\Psi$ for the equation~\eqref{eq:IDFMSupersystemA}.
In fact, this ``pseudopotential'' is a modification, $\Psi={\rm e}^{-y-z}\tilde\Psi$,
of the standard potential~$\tilde\Psi$ for the equation~\eqref{eq:IDFMSupersystemA}
associated with the conserved current $({\rm e}^{y+z}q_z,-{\rm e}^{y+z}q)$ of this equation
via the potential system $\tilde\Psi_y={\rm e}^{y+z}q$, $\tilde\Psi_z={\rm e}^{y+z}q_z$.
It is easily seen that the function~$\Psi$ satisfies the Klein--Gordon equation~$\Psi_{yz}=\Psi$.
Moreover, solutions of the equations~\eqref{eq:IDFMSupersystemA}, \eqref{eq:IDFMSupersystemB}
and~\eqref{eq:IDFMSupersystemC} are locally expressed in terms of~$\Psi$,
\[
q=\Psi_y+\Psi,\quad
p=\frac12{\rm e}^{-y-z}(\Psi_z-\Psi_y),\quad
s=W\left({\rm e}^{y+z}(\Psi_y+\Psi_z-2\Psi)\right).
\]
Here and in what follows~$W$ is an arbitrary smooth function of its argument.
Returning to the old coordinates, we obtain the regular family of solutions of the system~\eqref{eq:IDFMDiagonalizedSystem},
which is expressed in terms of the general solution of the Klein--Gordon equation.
Note that the nondegeneracy condition for this inverse transformation is $K^1K^2\neq0$, where, in terms of~$\Psi$,
\[
K^1=\Psi_{zz}-\Psi_z+\Psi_y-\Psi,\quad K^2=\Psi_{yy}-\Psi_y+\Psi_z-\Psi.
\]
In view of the Klein--Gordon equation~$\Psi_{yz}=\Psi$,
the inequalities $K^1\ne0$ and $K^2\ne0$ are equivalent to each other
as well as to the condition $\Psi\notin\langle {\rm e}^{-y-z}, {\rm e}^{y+z}, (y-z){\rm e}^{y+z} \rangle$.

If the nondegeneracy condition $\ri^1_t\ri^2_x-\ri^1_x\ri^2_t\ne0$ does not hold,
then at least one of the Riemann invariants~$\ri^1$ and~$\ri^2$ is a constant.
If only one Riemann invariant is a constant,
we derive the singular family of solutions of~\eqref{eq:IDFMDiagonalizedSystem}.
Let $\ri^1$ be a constant,
$\ri^1=c$. Then the equation~\eqref{eq:IDFMDEq1} is trivially satisfied,
and we make the rank-one hodograph transformation
$\bar t=t$, $\bar z=\ri^2$, $\bar q=x$, $\bar s=\ri^3$
in the two remaining equations~\eqref{eq:IDFMDEq2} and~\eqref{eq:IDFMDEq3},
exchanging the roles of~$x$ and~$\ri^2$,
that is, $\bar t$ and $\bar z$ are the new independent variables,
$\bar q$ and~$\bar s$ are the new dependent variables.
This yields the system
$
\bar q_{\bar t}=\bar z+c-1,\ \bar s_{\bar z}+\bar q_{\bar z}\bar s_{\bar t}=0.
$
Integrating the first equation to $\bar q=(\bar z+c-1)\bar t+{\rm e}^{\bar z}\Theta^2_{\bar z}$,
where $\Theta^2$ is an arbitrary function of~$\bar z$. It is chosen with a help of a hindsight to represent
the general solution of the second equation in the form $\bar s=W({\rm e}^{-\bar z}\bar t-\Theta^2_{\bar z}-\Theta^2)$.
The consideration when $\ri^2$ being a constant is similar.

When the both Riemann invariants~$\ri^1$ and~$\ri^2$ are constants,
we obtain an ultra-singular family of solutions of~\eqref{eq:IDFMDiagonalizedSystem}.

\begin{theorem}\label{thm:IDFMCompleteSolutioN}
Any solution of the system~\eqref{eq:IDFMDiagonalizedSystem} (locally) belongs to one of the following families;
below $W$~is an arbitrary function of its argument.

\medskip\par\noindent
1. The regular family, where both the Riemann invariants~$\ri^1$ and~$\ri^2$ are not constants (the general solution):
\begin{gather*}
t=-{\rm e}^{(\ri^2-\ri^1)/2}(\Psi_{\ri^1}+\Psi_{\ri^2}),\quad
x={\rm e}^{(\ri^2-\ri^1)/2}\big((2\Psi_{\ri^1}+\Psi)-(\ri^1+\ri^2+1)(\Psi_{\ri^1}+\Psi_{\ri^2})\big),\\
\ri^3=W\big({\rm e}^{(\ri^1-\ri^2)/2}(\Psi_{\ri^1}-\Psi_{\ri^2}-\Psi)\big).
\end{gather*}
Here the function $\Psi=\Psi(\ri^1,\ri^2)$ runs through the set of solutions of the Klein--Gordon equation $\Psi_{\ri^1\ri^2}=-\Psi/4$
with $\Psi\notin\langle\, {\rm e}^{(\ri^2-\ri^1)/2},\, {\rm e}^{(\ri^1-\ri^2)/2},\, (\ri^1+\ri^2){\rm e}^{(\ri^1-\ri^2)/2}\, \rangle$.

\medskip\par\noindent
2. The two singular families, where exactly one of the Riemann invariants~$\ri^1$ and~$\ri^2$ is a constant:
\begin{gather*}
\ri^1=c, \quad x=(\ri^2+c-1)t+{\rm e}^{ \ri^2}\Theta^2_{\ri^2},\quad \ri^3=W({\rm e}^{-\ri^2}t-\Theta^2_{\ri^2}-\Theta^2);\\
\ri^2=c, \quad x=(\ri^1+c+1)t+{\rm e}^{-\ri^1}\Theta^1_{\ri^1},\quad \ri^3=W({\rm e}^{ \ri^1}t+\Theta^1_{\ri^1}-\Theta^1).
\end{gather*}
Here $c$ is an arbitrary constant
and $\Theta^1=\Theta^1(\ri^1)$ and $\Theta^2=\Theta^2(\ri^2)$ are arbitrary functions of their arguments.

\medskip\par\noindent
3. The ultra-singular family, where $\ri^1$ and~$\ri^2$ are arbitrary constants and $\ri^3=W(x-(\ri^1+\ri^2)t)$.
\end{theorem}

The regular, singular and ultra-singular families of solutions of the system~$\mathcal S$
are associated with solutions of the subsystem~$\mathcal S_0$ of rank 2, 1 and 0, respectively;
cf.~\cite{GrundlandHuard2006}.

Alternatively, to get the subfamily of regular solutions with nonconstant parameter function~$W$,
one can employ the generalized hodograph transformation~\cite{Tsarev1985},
see details in~\cite[Section~9]{OpanasenkoBihloPopovychSergyeyev2020}.

\section{Preliminaries}\label{sec:IDFMPreliminaries}

Given a system~$\mathcal L$ of differential equations,
we denote by~$\mathcal L^{(\infty)}$ the manifold
defined by the system~$\mathcal L$ and its differential consequences in the associated jet space.
A local object associated with~$\mathcal L$
within the framework of symmetry analysis of differential equations,
like a generalized symmetry, a conserved current of a local conservation law,
a conservation-law characteristic or a cosymmetry,
is called trivial if it vanishes on solutions of~$\mathcal L$
or, equivalently, on~$\mathcal L^{(\infty)}$.
Two such local objects of the same kind are naturally assumed equivalent
if their difference is trivial,
and thus such local objects of the same kind in total
are considered up to this equivalence relation.

The system~$\mathcal S$ given by~\eqref{eq:IDFMDiagonalizedSystem} is of the evolution form.
The jet variables~$t$, $x$ and $\ri^i_\kappa=\p^\kappa\ri^i/\p x^\kappa$, $i=1,2,3$, $\kappa\in\mathbb N_0$,
constitute the standard coordinates on the manifold~$\mathcal S^{(\infty)}$.
Therefore, up to the above equivalence relation on solutions of~$\mathcal S$,
for the coset of each of local symmetry-like objects associated with~$\mathcal S$
we can consider a representative whose components do not depend on the derivatives of~$\ri$
involving differentiation with respect to~$t$.%
\footnote{%
Here, for conservation-law characteristics we need to use Lemma~3 in~\cite{MartinezAlonso1979},
see also~\cite[Lemma 4.28]{Olver1993}.
}
A~symbol with~$[\ri]$, like $f[\ri]$, denotes a differential function of~$\ri$
that depends at most on~$t$, $x$ and a finite number of derivatives of~$\ri$ with respect to~$x$,
$f=f(t,x,\ri_0,\dots,\ri_\kappa)$, $\kappa\in\mathbb N_0$.
Below we consider only such differential functions and assume
that the components of any local symmetry-like objects associated with~$\mathcal S$
are such differential functions.
For $i\in\{1,2,3\}$,
the order~$\ord_{\ri^i}f[\ri]$ of a differential function~$f[\ri]$ with respect to~$\ri^i$ is defined
to be equal $\smash{\max\{\kappa\in\mathbb N_0\mid f_{\ri^i_\kappa}\ne0\}}$ unless this set is empty and $-\infty$ otherwise.

We restrict the total derivative operators~$\mathrm D_x$ and~$\mathrm D_t$
with respect to~$x$ and~$t$ to the set of above differential functions of~$\ri$,
and additionally exclude the derivatives of~$\ri$ that involve differentiation with respect to~$t$
from~$\mathrm D_t$ in view of the system~$\mathcal S$,
respectively obtaining the (commuting) operators
\[
\mathscr D_x:=\p_x+\sum_{\kappa=0}^\infty\sum_{i=1}^3\ri^i_{\kappa+1}\p_{\ri^i_\kappa},\quad
\mathscr D_t:=\p_t-\sum_{\kappa=0}^\infty\sum_{i=1}^3\mathscr D_x^\kappa(V^i\ri^i_1)\p_{\ri^i_\kappa}.
\]
We also define the commuting operators
$\mathscr A:={\rm e}^{\ri^2-\ri^1}\mathscr D_x$ and
$\mathscr B:=\mathscr D_t+(\ri^1+\ri^2)\mathscr D_x$,
$\mathscr A\mathscr B=\mathscr B\mathscr A$.

It is convenient to introduce the modified coordinates
$t$, $x$, $r^j_\kappa=\ri^j_\kappa$ and $\omega^\kappa:=\mathscr A^\kappa\ri^3$
for $\kappa\in\mathbb N_0$ and $j=1,2$
on the manifold~$\mathcal S^{(\infty)}$
instead of the standard ones.%
\footnote{%
The operator~$\mathscr A$ and the modified coordinates
are related to the degeneration of $V^3$ meaning, that $V^3_{\ri^3}=0$; cf.\ \cite[Theorem 5.2]{Doyle1994}.
}
In this notation, we have
\begin{gather*}
\mathscr A\omega^\kappa=\omega^{\kappa+1},\quad \mathscr B\omega^\kappa=0,\quad \kappa\in\mathbb N_0, \quad
\mathscr Br^1=-r^1_1,\quad \mathscr Br^2=r^2_1,
\\
\mathscr D_x=\p_x+\sum_{\kappa=0}^\infty\big(
 r^1_{\kappa+1}\p_{r^1_\kappa}
+r^2_{\kappa+1}\p_{r^2_\kappa}
+{\rm e}^{r^1-r^2}\omega^{\kappa+1}\p_{\omega^\kappa}\big),
\\
\mathscr D_t=\p_t-\sum_{\kappa=0}^\infty\big(
 \mathscr D_x^\kappa(V^1r^1_1)\p_{r^1_\kappa}
+\mathscr D_x^\kappa(V^2r^2_1)\p_{r^2_\kappa}
+(r^1+r^2){\rm e}^{r^1-r^2}\omega^{\kappa+1}\p_{\omega^\kappa}\big).
\end{gather*}
We define the orders $\ord_{r^j}f$, $j=1,2$, and $\ord_\omega f$
of a differential function~$f=f[\ri]$ with respect to~$r^j$ and ``$\omega$''
to be equal $\max\{\kappa\mid f_{r^j_\kappa}\ne0\}$ and $\max\{\kappa\mid f_{\omega^\kappa}\ne0\}$,
respectively, unless the corresponding set is empty and $-\infty$ otherwise.
Note that  $\ord_\omega f=\ord_{\ri^3}f$.
The notation like $f[r^1,r^2]$, or equivalently $f[\ri^1,\ri^2]$,
denotes a differential function~$f$ of~$(r^1,r^2)=(\ri^1,\ri^2)$.

\begin{lemma}\label{lem:IDFM:DiffFunctionsOfOmega}
A differential function~$f=f[\ri]$ satisfies the equation $\mathscr Bf=0$ if and only if
it is a smooth function of a finite number of~$\omega$'s,
$f=f(\omega^0,\dots,\omega^\kappa)$ with $\kappa\in\mathbb N_0$.
\end{lemma}

\begin{proof}
Provided~$f$ being a smooth function of a finite number of~$\omega$'s,
it satisfies the equation $\mathscr Bf=0$
because of $\mathscr B\omega^\kappa=0$ for all $\kappa\in\mathbb N_0$.

Conversely, using the modified coordinates on~$\mathcal S^{(\infty)}$ we denote $\kappa_j=\ord_{r^j}f$, $j=1,2$.
Suppose that $\kappa_j\geqslant0$ for some~$j$.
Then collecting coefficients of~$r^j_{\kappa_j+1}$ in the equation $\mathscr Bf=0$
yields $\p f/\p{r^j_{\kappa_j}}=0$, which gives a contradiction.
Hence the function~$f$ does not depend on~$r^j_\kappa$, $\kappa\in\mathbb N_0$.
The equation $\mathscr Bf=0$ takes the form $f_t+(r^1+r^2)f_x=0$,
splitting with respect to $(r^1,r^2)$ to $f_t=f_x=0$.
\end{proof}

As the standard coordinates on the manifold~$\mathcal K^{(\infty)}$
associated with the system~\eqref{eq:IDFMSupersystem},
we can take the jet variables~$y$, $z$,
$q_\iota=\p^\iota q/\p y^\iota$ if $\iota\geqslant0$
and $q_{\iota}=\p^{-\iota} q/\p z^{-\iota}$ if $\iota<0$, $\iota\in\mathbb Z$,
$s_\kappa=\p^\kappa s/\p y^\kappa$, $\kappa\in\mathbb N_0$.
In these coordinates, the restrictions of the total derivative operators
with respect to~$y$ and~$z$ respectively take the form
\[
\mathscr D_y=\p_y+\sum_{\iota=-\infty}^{+\infty}q_{\iota+1}\p_{q_\iota}+\sum_{\kappa=0}^{+\infty}s_{\kappa+1}\p_{s_\kappa},\quad
\mathscr D_z=\p_z+\sum_{\iota=-\infty}^{+\infty}q_{\iota-1}\p_{q_\iota}+\sum_{\kappa=0}^{+\infty}\mathscr D_y^\kappa\left(\frac{K^1}{K^2}s_1\right)\p_{s_\kappa},
\]
where
$K^1:=q_{-2}-2q_{-1}+q_0$,
$K^2:=q_1+q_{-1}-2q_0$.
The infinite prolongation of the transformation~\eqref{eq:IDFMTransReducingToKGEq}
induces pushing forward of the operators~$\mathscr D_t$, $\mathscr D_x$, $\mathscr A$ and $\mathscr B$ to the operators
\begin{gather*}
\hat{\mathscr D}_t=-\frac{{\rm e}^{y+z}}{K^2}(2y-2z+1)\mathscr D_y-\frac{{\rm e}^{y+z}}{K^1}(2y-2z-1)\mathscr D_z,\quad
\hat{\mathscr D}_x=\frac{{\rm e}^{y+z}}{K^2}\mathscr D_y+\frac{{\rm e}^{y+z}}{K^1}\mathscr D_z,\\
\hat{\mathscr A}=\frac{{\rm e}^{-y-z}}{K^2}\mathscr D_y+\frac{{\rm e}^{-y-z}}{K^1}\mathscr D_z,\quad
\hat{\mathscr B}=-\frac{{\rm e}^{y+z}}{K^2}\mathscr D_y+\frac{{\rm e}^{y+z}}{K^1}\mathscr D_z,
\end{gather*}
and thus $\hat{\mathscr A}\hat{\mathscr B}=\hat{\mathscr B}\hat{\mathscr A}$.

A~symbol with~$[q,s]$, like $f[q,s]$, denotes a differential function of~$(q,s)$
that depends at most on~$y$, $z$ and a finite, but unspecified number of $q_\iota$, $\iota\in\mathbb Z$, and $s_\kappa$, $\kappa\in\mathbb N_0$.
The order~$\ord_sf$ of a differential function~$f=f[q,s]$ with respect to~$s$ is defined
to be equal $\smash{\max\{\kappa\in\mathbb N_0\mid f_{s_\kappa}\ne0\}}$ unless this set is empty and $-\infty$ otherwise.
Analogously, a symbol with~$[q]$, like $f[q]$, denotes a differential function of~$q$
that depends at most on~$y$, $z$ and a finite, but unspecified number of $q_\iota$, $\iota\in\mathbb Z$.

We also use the modified coordinates~$y$, $z$,
$\hat q_\iota=q_\iota$, $\iota\in\mathbb Z$
and $\hat\omega^\kappa=\hat{\mathscr A}^\kappa s$, $\kappa\in\mathbb N_0$,
on the manifold~$\mathcal K^{(\infty)}$.

Lemma~\ref{lem:IDFM:DiffFunctionsOfOmega} implies the following assertion.

\begin{corollary}\label{cor:IDFM:DiffFunctionsOfOmegaHat}
A differential function~$f=f[q,s]$ satisfies the equation $\hat{\mathscr B}f=0$,
i.e., \[K^1\mathscr D_yf=K^2\mathscr D_zf,\]
if and only if
it is a smooth function of a finite number of~$\hat\omega$'s,
$f=f(\hat\omega^0,\dots,\hat\omega^\kappa)$ with $\kappa\in\mathbb N_0$.
\end{corollary}

The infinite prolongation of the transformation~\eqref{eq:IDFMInverseToTransReducingToKGEq}
induces pushing forward of the operators~$\mathscr D_y$ and~$\mathscr D_z$ to the (commuting) operators
\[
\tilde{\mathscr D}_y:=-\frac1{\ri^1_x}\big(\mathscr D_t+(\ri^1+\ri^2-1)\mathscr D_x\big),\quad
\tilde{\mathscr D}_z:=-\frac1{\ri^2_x}\big(\mathscr D_t+(\ri^1+\ri^2+1)\mathscr D_x\big).
\]

\section{Generalized symmetries}\label{sec:IDFMGenSyms}

The following two facts allow us to exhaustively describe generalized symmetries of the system~\eqref{eq:IDFMDiagonalizedSystem}.
Firstly, the equation~\eqref{eq:IDFMDEq3} is partially coupled with the equations~\eqref{eq:IDFMDEq1} and~\eqref{eq:IDFMDEq2}.
Secondly, the subsystem~\eqref{eq:IDFMDEq1}--\eqref{eq:IDFMDEq2} is linearized by the hodograph transformation,
and the associated linear system reduces to the (1+1)-dimensional Klein--Gordon equation.

We denote by
$\Sigma$ the algebra of generalized symmetries of the system~\eqref{eq:IDFMDiagonalizedSystem},
and by $\Sigma^{\rm triv}$ the algebra of its trivial generalized symmetries,
whose characteristics vanish on solutions of~\eqref{eq:IDFMDiagonalizedSystem}.
The quotient algebra $\Sigma^{\rm q}=\Sigma/\Sigma^{\rm triv}$
can be identified, e.g., with the subalgebra of canonical representatives in the reduced evolutionary form,
$\hat\Sigma^{\rm q}=\big\{\sum_{i=1}^3\eta^i[\ri]\p_{\ri^i}\in\Sigma\big\}$.
The criterion of invariance of the system~\eqref{eq:IDFMDiagonalizedSystem}
with respect to the generalized vector field~$\sum_{i=1}^3\eta^i[\ri]\p_{\ri^i}$ results in
the system of three determining equations for the components~$\eta^i$,
\begin{subequations}\label{eq:DetEqsForSigmaq}
\begin{gather}
\mathscr D_t\eta^1+(\ri^1+\ri^2+1)\mathscr D_x\eta^1+\ri^1_x(\eta^1+\eta^2)=0,\label{eq:DetEqsForSigmaqA}\\
\mathscr D_t\eta^2+(\ri^1+\ri^2-1)\mathscr D_x\eta^2+\ri^2_x(\eta^1+\eta^2)=0,\label{eq:DetEqsForSigmaqB}\\
\mathscr D_t\eta^3+(\ri^1+\ri^2)\mathscr D_x\eta^3+\ri^3_x(\eta^1+\eta^2)=0.\label{eq:DetEqsForSigmaqC}
\end{gather}
\end{subequations}

\begin{lemma}\label{lem:IDFM:GenSym12}
For any generalized vector field~$\sum_{i=1}^3\eta^i[\ri]\p_{\ri^i}$ from~$\hat\Sigma^{\rm q}$,
its components~$\eta^1$ and~$\eta^2$ do not depend on derivatives of~$\ri^3$,
i.e., $\eta^1=\eta^1[\ri^1,\ri^2]$ and $\eta^2=\eta^2[\ri^1,\ri^2]$.
\end{lemma}

\begin{proof}
Suppose that~$\kappa_j:=\ord_{\ri^3}\eta^j\geqslant0$ for some~$j\in\{1,2\}$.
Collecting the coefficients of the jet variable~$\ri^3_{\kappa_j+1}$ in the $j$th equation of~\eqref{eq:DetEqsForSigmaq}
yields the equation $\p\eta^j/\p\ri^3_{\kappa_j}=0$,
which contradicts the assumption. Hence $\kappa_j=-\infty$ for any $j=1,2$.
\end{proof}

Lemma~\ref{lem:IDFM:GenSym12} is the manifestation of partial coupling of the system~\eqref{eq:IDFMDiagonalizedSystem}.
In view of this lemma, the subalgebra~$\hat\Sigma^{\rm q}_3$ of~$\hat\Sigma^{\rm q}$
constituted by elements with vanishing~$\eta^1$ and~$\eta^2$ is an ideal of~$\hat\Sigma^{\rm q}$,
and the quotient algebra $\Sigma^{\rm q}_{12}:=\hat\Sigma^{\rm q}/\hat\Sigma^{\rm q}_3$
is isomorphic to the subalgebra of reduced generalized symmetries of the subsystem~\eqref{eq:IDFMDEq1}--\eqref{eq:IDFMDEq2}
that admit local prolongations to~$\ri^3$.

The ideal~$\hat\Sigma^{\rm q}_3$ is described by the following corollary of Lemma~\ref{lem:IDFM:DiffFunctionsOfOmega}.

\begin{corollary}\label{cor:IDFM:GenSym3}
A generalized vector field~$\eta^3\p_{\ri^3}$ belongs to~$\hat\Sigma^{\rm q}$ if and only if the coefficient~$\eta^3$
is a smooth function of a finite number of~$\omega$'s.
\end{corollary}

\begin{proof}
The invariance of the system~\eqref{eq:IDFMDiagonalizedSystem} with respect to the generalized vector field~$\eta^3\p_{\ri^3}$
leads to the single determining equation $\mathscr B\eta^3=0$.
Further we use Lemma~\ref{lem:IDFM:DiffFunctionsOfOmega}.
\end{proof}

Therefore, the infinite prolongation of an element~$f\p_{\ri^3}$ of~$\hat\Sigma^{\rm q}$ is equal to
$\sum_{\iota=0}^\infty(\hat{\mathscr A}^\iota f)\p_{\omega^\iota}$,
and thus the commutator of elements~$f^1\p_{\ri^3}$ and~$f^2\p_{\ri^3}$ of~$\hat\Sigma^{\rm q}$ is
$\sum_{\iota=0}^\infty\big((\hat{\mathscr A}^\iota f^1)f^2_{\omega^\iota}-(\hat{\mathscr A}^\iota f^2)f^1_{\omega^\iota}\big)\p_{\ri^3}$,
where $\hat{\mathscr A}=\sum_{\kappa=0}^{\infty}\omega^{\kappa+1}\p_{\omega^\kappa}$.

We specify the form of canonical representatives of cosets of~$\hat\Sigma^{\rm q}_3$.

\begin{lemma}\label{lem:IDFM:GenSym12More}
Each coset of~$\hat\Sigma^{\rm q}_3$ contains a generalized vector field of the form
\begin{gather}\label{eq:RepresentationForElementsOfSigma12}
\eta^1[\ri^1,\ri^2]\p_{\ri^1}+\eta^2[\ri^1,\ri^2]\p_{\ri^2}+{\rm e}^{\ri^2-\ri^1}\ri^3_x\hat\eta^3[\ri^1,\ri^2]\p_{\ri^3},
\end{gather}
where the coefficients $\eta^1$, $\eta^2$ and $\hat\eta^3$ satisfy
the system of determining equations~\eqref{eq:DetEqsForSigmaqA}, \eqref{eq:DetEqsForSigmaqB} and
\begin{gather}\label{eq:DetEqsForSigma12C}
\mathscr D_t\hat\eta^3+(\ri^1+\ri^2)\mathscr D_x\hat\eta^3+{\rm e}^{\ri^1-\ri^2}(\eta^1+\eta^2)=0.
\end{gather}
\end{lemma}

\begin{proof}
In view of Lemma~\ref{lem:IDFM:GenSym12} and Corollary~\ref{cor:IDFM:GenSym3},
it suffices to show that the third components of canonical representatives
for elements from the quotient algebra~\smash{$\Sigma^{\rm q}_{12}$}
can be chosen to be of the form
$\eta^3={\rm e}^{\ri^2-\ri^1}\ri^3_x\hat\eta^3[\ri^1,\ri^2]$.
After substituting the representation $\eta^3={\rm e}^{\ri^2-\ri^1}\ri^3_x\hat\eta^3[\ri]$
into the equation~\eqref{eq:DetEqsForSigmaqC}, we derive the equation~\eqref{eq:DetEqsForSigma12C}.
We use the modified coordinates on the manifold~$\mathcal S^{(\infty)}$.
If the coefficient~$\hat\eta^3$ depends on $\omega^\kappa$ for some~$\kappa\in\mathbb N_0$,
then a differential function of $(\ri^1,\ri^2)$ obtained from~$\hat\eta^3$
by fixing values of all involved $\omega^\kappa$'s in the domain of~$\hat\eta^3$
is also a solution of~\eqref{eq:DetEqsForSigma12C} for the same value of~$(\eta^1,\eta^2)$.
\end{proof}

The elements of the form~\eqref{eq:RepresentationForElementsOfSigma12} from the algebra~$\hat\Sigma^{\rm q}$
constitute a subalgebra of this algebra, which we denote by $\bar\Sigma^{\rm q}_{12}$.
Unfortunately, the algebras~$\Sigma^{\rm q}_{12}$ and~$\bar\Sigma^{\rm q}_{12}$ are not isomorphic.
Although $\hat\Sigma^{\rm q}=\bar\Sigma^{\rm q}_{12}+\hat\Sigma^{\rm q}_3$,
this sum is not direct since
$\bar\Sigma^{\rm q}_{12}\cap\hat\Sigma^{\rm q}_3=\langle {\rm e}^{\ri^2-\ri^1}\ri^3_x\p_{\ri^3}\rangle$.
The algebra $\Sigma^{\rm q}_{12}$ is naturally isomorphic
to the quotient algebra $\bar\Sigma^{\rm q}_{12}/\langle {\rm e}^{\ri^2-\ri^1}\ri^3_x\p_{\ri^3}\rangle$.

Deriving the exhaustive description of the algebra~$\Sigma^{\rm q}_{12}$ is quite complicated.
For this purpose, we reduce the system~\eqref{eq:IDFMDiagonalizedSystem}
to a system~\eqref{eq:IDFMSupersystem} containing the (1+1)-dimensional Klein--Gordon equation.
Similarly to the system~\eqref{eq:IDFMDiagonalizedSystem},
we denote by
$\mathfrak S$ the algebra of generalized symmetries of the system~\eqref{eq:IDFMSupersystem},
and by $\mathfrak S^{\rm triv}$ the algebra of its trivial generalized symmetries,
whose characteristics vanish on solutions of~\eqref{eq:IDFMSupersystem}.
The quotient algebra $\mathfrak S^{\rm q}=\mathfrak S/\mathfrak S^{\rm triv}$
can be identified, e.g., with the subalgebra of canonical representatives in the evolutionary form,
$\hat{\mathfrak S}^{\rm q}=\{\chi[q,s]\p_q+\theta[q,s]\p_s\in\mathfrak S\}$.
The Lie bracket on~$\hat{\mathfrak S}^{\rm q}$ is defined
as the modified Lie bracket of generalized vector fields in the jet space with the independent variables $(y,z)$
and the dependent variables $(q,s)$,
where all arising mixed derivatives of~$q$
and all arising derivatives of~$s$ that involve differentiation with respect to~$y$
are substituted in view of the system~\eqref{eq:IDFMSupersystem} and its differential consequences.
The system of determining equations for components of elements of~$\hat{\mathfrak S}^{\rm q}$ is
\begin{subequations}\label{eq:IDFMSupersystemGenSymsDetEqs}
\begin{gather}\label{eq:IDFMSupersystemGenSymsDetEqA}
\mathscr D_y\mathscr D_z\chi=\chi,
\\\label{eq:IDFMSupersystemGenSymsDetEqB}
s_1(\mathscr D_z-1)^2\chi+K_1\mathscr D_y\theta=\frac{K^1}{K^2}s_1(\mathscr D_y+\mathscr D_z-2)\chi+K^2\mathscr D_z\theta.
\end{gather}
\end{subequations}

The algebra~$\hat{\mathfrak S}^{\rm q}$ is isomorphic to the algebra~$\hat\Sigma^{\rm q}$.
This isomorphism is induced by the pushforward of $\Sigma$ onto~$\mathfrak S$
that is generated by the point transformation~\eqref{eq:IDFMTransReducingToKGEq},
excluding the derivatives of~$p$ (including $p$ itself)
in view of the equation~\eqref{eq:IDFMSupersystemC} and its differential consequences
and the successive projection of the obtained generalized vector fields
to the jet space with the independent variables $(y,z)$ and the dependent variables $(q,s)$.
To map $\mathfrak S$ into~$\Sigma$, we need to prolong the elements of~$\mathfrak S$ to~$p$
according the equation~\eqref{eq:IDFMSupersystemC}
and make the pushforward by the point transformation~\eqref{eq:IDFMInverseToTransReducingToKGEq}.

\begin{lemma}\label{lem:IDFM:GenSymsOfSupersystem1}
The $q$-component of every element of~$\hat{\mathfrak S}^{\rm q}$ does not depend on~$s$ and its derivatives.
\end{lemma}

\begin{proof}
Suppose that $Q=\chi\p_q+\theta\p_s\in\hat{\mathfrak S}^{\rm q}$, and $\kappa:=\ord_s\chi\geqslant0$.
Then invariance criterion for the equation~$q_{yz}=q$
and the generalized vector field~$Q$ implies,
after collecting coefficients of~$s_{\kappa+2}$,
the equation $\chi_{s_\kappa}=0$, which contradicts the assumption.
This is why $\ord_s\chi=-\infty$.
\end{proof}

\begin{remark}\label{rem:IDFM:NnvanishingOfK1K2}
Only simultaneously nonvanishing of~$K^1$ and~$K^2$ is essential for Lemma~\ref{lem:IDFM:GenSymsOfSupersystem1},
but not the specific form of these coefficients.
\end{remark}

Lemma~\ref{lem:IDFM:GenSymsOfSupersystem1} is the counterpart of Lemma~\ref{lem:IDFM:GenSym12}
for the system~\eqref{eq:IDFMSupersystem} and is the manifestation of partial coupling of this system.
In view of Lemma~\ref{lem:IDFM:GenSymsOfSupersystem1}, the subalgebra~$\hat{\mathfrak S}^{\rm q}_s$ of~$\hat{\mathfrak S}^{\rm q}$
constituted by elements with vanishing $q$-components is an ideal of~$\hat{\mathfrak S}^{\rm q}$.
In view of Corollary~\ref{cor:IDFM:DiffFunctionsOfOmegaHat} (or Corollary~\ref{cor:IDFM:GenSym3}),
this ideal consists of generalized vector fields of the form $\theta\p_s$,
where $\theta$ is a smooth function of a finite, but unspecified number of~$\hat\omega$'s.
Since the ideal~$\hat{\mathfrak S}^{\rm q}_s$ of~$\hat{\mathfrak S}^{\rm q}$ corresponds to
and is isomorphic to the ideal~$\hat\Sigma^{\rm q}_3$ of~$\hat\Sigma^{\rm q}$,
for our purpose it suffices to describe the quotient algebra
$\mathfrak S^{\rm q}_q:=\hat{\mathfrak S}^{\rm q}/\hat{\mathfrak S}^{\rm q}_s$.

Denote by $\hat{\mathfrak K}^{\rm q}$ the algebra of reduced generalized symmetries
of the (1+1)-dimensional Klein--Gordon equation~\eqref{eq:IDFMSupersystemA},
$\hat{\mathfrak K}^{\rm q}=\{\chi[q]\p_q\mid\mathscr D_y\mathscr D_z\chi=\chi\}$.
The quotient algebra $\mathfrak S^{\rm q}_q$ is naturally isomorphic to
the subalgebra~$\mathfrak A$ of~$\hat{\mathfrak K}^{\rm q}$
that consists of elements of~$\hat{\mathfrak K}^{\rm q}$ admitting local prolongations to~$s$.
It was proved in \cite{OpanasenkoPopovych2018} that
the algebra~$\hat{\mathfrak K}^{\rm q}$ is the semi-direct sum of
its subalgebra~$\hat\Lambda^{\rm q}$ and its ideal~$\hat{\mathfrak K}^{-\infty}$,
$\hat{\mathfrak K}^{\rm q}=\hat\Lambda^{\rm q}\lsemioplus\hat{\mathfrak K}^{-\infty}$, where
\begin{gather*}
\hat\Lambda^{\rm q}:=\langle\,(\mathscr J^\kappa q)\p_q,\,(\mathscr D_y^\iota\mathscr J^\kappa q)\p_q,\,(\mathscr D_z^\iota\mathscr J^\kappa q)\p_q,\,
\kappa\in\mathbb N_0,\,\iota\in\mathbb N\,\rangle,
\quad
\hat{\mathfrak K}^{-\infty}:=\{f(y,z)\p_q\mid f\in{\rm KG}\},
\end{gather*}
$\mathscr J:=y\mathscr D_y-z\mathscr D_z$,
and ${\rm KG}$ denotes the solution set of the (1+1)-dimensional Klein--Gordon equation~\eqref{eq:IDFMSupersystemA},
i.e., $f\in{\rm KG}$ means that $f_{yz}=f$.

\begin{lemma}\label{lem:IDFM:GenSymsOfSupersystem2B}
$\mathfrak A=\big\{Q^{\zeta,c}:=\big((\mathscr D_y+1)\zeta+cq\big)\p_q\mid\zeta=\zeta[q]\colon\mathscr D_y\mathscr D_z\zeta=\zeta,\,c\in\mathbb R\big\}$,
and an appropriate prolongation of the generalized vector field~$Q^{\zeta,c}$ to~$s$ is given by
\begin{gather}\label{eq:IDFM:GenSymsOfSupersystemProlongationToTheta}
\theta=\frac{s_1}{K^2}(\mathscr D_y+\mathscr D_z-2)\zeta.
\end{gather}
\end{lemma}

\begin{proof}
Denote $\tilde{\mathfrak A}=\big\{Q^{\zeta,c}:=\big((\mathscr D_y+1)\zeta+cq\big)\p_q\mid\zeta=\zeta[q]\colon\mathscr D_y\mathscr D_z\zeta=\zeta,\,c\in\mathbb R\big\}$.
Note that here the form of~$\zeta$ is defined up to summands proportional to ${\rm e}^{-y-z}$.

For any solution~$\zeta$ of the equation $\mathscr D_y\mathscr D_z\zeta=\zeta$,
the differential functions~$\chi=(\mathscr D_y+1)\zeta$
and~$\theta$ defined by~\eqref{eq:IDFM:GenSymsOfSupersystemProlongationToTheta}
satisfy the system~\eqref{eq:IDFMSupersystemGenSymsDetEqs}.
The tuple $(\chi,\theta)=(q,0)$ is a solution of~\eqref{eq:IDFMSupersystemGenSymsDetEqs} as well.
Hence $\mathfrak A\supseteq\tilde{\mathfrak A}$.

Suppose that a generalized vector field $\chi[q]\p_q$ belongs to $\mathfrak A$.
This means that there exists $\theta=\theta[q,s]$ such that
$\chi\p_q+\theta\p_s\in\hat{\mathfrak S}^{\rm q}$.
Then the tuple $(\chi,\theta)$ satisfies the system~\eqref{eq:IDFMSupersystemGenSymsDetEqs}.
By the substitution $\theta=s_1(K^2)^{-1}\tilde\theta$,
the equation~\eqref{eq:IDFMSupersystemGenSymsDetEqB} is reduced to
\begin{gather}\label{eq:IDFMSupersystemGenSymsDetEqBReduced}
K^1(\mathscr D_y+1)\tilde\theta-K^2(\mathscr D_z+1)\tilde\theta
=
K^1(\mathscr D_y+\mathscr D_z-2)\chi-K^2(\mathscr D_z-1)^2\chi.
\end{gather}
We use the modified coordinates on the manifold~$\mathcal K^{(\infty)}$.
If the function~$\tilde\theta$ depends on $\omega^\kappa$ for some~$\kappa\in\mathbb N_0$,
then a differential function of $q$ obtained from~$\tilde\theta$
by fixing values of all involved $\hat\omega^\kappa$'s in the domain of~$\tilde\theta$
is also a solution of~\eqref{eq:IDFMSupersystemGenSymsDetEqBReduced} for the same value of~$\chi$.
Therefore, without loss of generality we can assume that $\tilde\theta=\tilde\theta[q]$.
Then the equation~\eqref{eq:IDFMSupersystemGenSymsDetEqBReduced} rewritten in the form
\[
K^1\big((\mathscr D_y+1)\tilde\theta-(\mathscr D_y+\mathscr D_z-2)\chi\big)
=
K^2\big((\mathscr D_z+1)\tilde\theta-(\mathscr D_z-1)^2\chi\big)
\]
implies that there exists a differential function $\mu=\mu[q]$ such that
\begin{gather}\label{eq:IDFMSupersystemGenSymsIncredibleEqsForTheta}
(\mathscr D_y+1)\tilde\theta-(\mathscr D_y+\mathscr D_z-2)\chi=\mu K^2,\quad
(\mathscr D_z+1)\tilde\theta-(\mathscr D_z-1)^2\chi=\mu K^1.
\end{gather}
We exclude~$\tilde\theta$ from these equations
by acting the operators $\mathscr D_z+1$ and $\mathscr D_y+1$
on the first and the second equations, respectively,
and subtracting the first obtained equation from the second one,
which gives the equation on~$\mu$ alone, $K^1\mathscr D_y\mu=K^2\mathscr D_z\mu$.
In view of Corollary~\ref{cor:IDFM:DiffFunctionsOfOmegaHat}, $\mu$ is a constant,
and hence equations~\eqref{eq:IDFMSupersystemGenSymsIncredibleEqsForTheta}
can be rewritten as
\begin{gather}\label{eq:IDFMSupersystemGenSymsIncredibleEqsForTheta2}
(\mathscr D_y+1)\tilde\theta=(\mathscr D_y+\mathscr D_z-2)(\chi+\mu q),\quad
(\mathscr D_z+1)\tilde\theta=(\mathscr D_z-1)^2(\chi+\mu q).
\end{gather}
We subtract the second equation from the result
of acting the operator~$\mathscr D_z$ on the first equation
and thus derive the equation $(\mathscr D_y\mathscr D_z-1)\tilde\theta=0$.
Then the differential function $\zeta=\zeta[q]$ that is defined by
$\zeta:=-\frac14\big(\tilde\theta-(\mathscr D_z+1)(\chi+\mu q)\big)$
satisfies the same equation, $(\mathscr D_y\mathscr D_z-1)\zeta=0$.
We express $\tilde\theta$ from the equality defining~$\zeta$,
$\tilde\theta=-4\zeta+(\mathscr D_z+1)(\chi+\mu q)$,
and substitute the obtained expression
into~$\eqref{eq:IDFMSupersystemGenSymsIncredibleEqsForTheta2}$,
deriving the equations
\[
-4(\mathscr D_y+1)\zeta=-4(\chi+\mu q),\quad
-4(\mathscr D_z+1)\zeta=-4\mathscr D_z(\chi+\mu q).
\]
The first of these equations gives the required representation for~$\chi$,
$\chi=(\mathscr D_y+1)\zeta-\mu q$.
The second equation is identically satisfied in view of the above representation for~$\chi$
and the equation  $\mathscr D_y\mathscr D_z\zeta=\zeta$.
We also get $\tilde\theta=-4\zeta+(\mathscr D_z+1)(\mathscr D_y+1)\zeta=(\mathscr D_y+\mathscr D_z-2)\zeta$.
Therefore, $\mathfrak A\subseteq\tilde{\mathfrak A}$, i.e., $\mathfrak A=\tilde{\mathfrak A}$,
and the equality~\eqref{eq:IDFM:GenSymsOfSupersystemProlongationToTheta} defines
an appropriate prolongation of $Q^{\zeta,c}\in\mathfrak A$ to~$s$.
\end{proof}

%\begin{remark}\label{rem:IDFM:GenSymsOfKGEqAdmittingProlongation}
In other words, Lemma~\ref{lem:IDFM:GenSymsOfSupersystem2B} implies
that an element of~$\hat\Lambda^{\rm q}$ can be mapped to a generalized symmetry of the system~\eqref{eq:IDFMDiagonalizedSystem}
if and only if the associated operator belongs to the subspace
\[
\langle\,1,\,(\mathscr D_y+1)\mathscr D_y^\iota\mathscr J^\kappa,\,(\mathscr D_z+1)\mathscr D_z^\iota\mathscr J^\kappa,\,
\kappa,\iota\in\mathbb N_0\,\rangle.
\]
In particular, this subspace contains all polynomials of~$\mathscr D_y$ and all polynomials of~$\mathscr D_z$.
A complement subspace to it in the entire space of operators associated with elements of~$\hat\Lambda^{\rm q}$
is $\langle\,\mathscr J^\kappa,\,\kappa\in\mathbb N\,\rangle$.
Elements of~$\hat\Lambda^{\rm q}$ associated with operators from the complement subspace
are mapped to \emph{nonlocal} symmetries of the system~\eqref{eq:IDFMDiagonalizedSystem}.
Such nonlocal symmetries are generalized symmetries of certain potential systems for the system~\eqref{eq:IDFMDiagonalizedSystem}
that are related to  potential systems for the (1+1)-dimensional Klein--Gordon equation~\eqref{eq:IDFMSupersystemA}.
We plan to study generalized potential symmetries of~\eqref{eq:IDFMDiagonalizedSystem}
in the sequel of the present paper.

Completing the above consideration, we prove the following theorem.

\begin{theorem}\label{thm:IDFMGenSyms}
The quotient algebra $\Sigma^{\rm q}$
of generalized symmetries of the system~\eqref{eq:IDFMDiagonalizedSystem}
is naturally isomorphic to the algebra~$\hat\Sigma^{\rm q}$ spanned by the generalized vector fields
\begin{gather*}
\check{\mathcal W}(\Omega)=\Omega\p_{\ri^3},
\quad
\check{\mathcal P}(\Phi)={\rm e}^{(\ri^2-\ri^1)/2}\left(
(\Phi+2\Phi_{\ri^1})\ri^1_x\p_{\ri^1}
+(\Phi-2\Phi_{\ri^2})\ri^2_x\p_{\ri^2}
+2\Phi\ri^3_x\p_{\ri^3}\right),
\\
\check{\mathcal D}=
\big(x-(\ri^1+\ri^2+1)t\big)\ri^1_x\p_{\ri^1}
+\big(x-(\ri^1+\ri^2-1)t\big)\ri^2_x\p_{\ri^2}
+\big(x-(\ri^1+\ri^2)t\big)\ri^3_x\p_{\ri^3},
\\
\check{\mathcal R}(\Gamma)={\rm e}^{(\ri^2-\ri^1)/2}\left(
 (\tilde{\mathscr D}_y\Gamma+\Gamma)\ri^1_x\p_{\ri^1}
+(\tilde{\mathscr D}_z\Gamma+\Gamma)\ri^2_x\p_{\ri^2}
+2\Gamma\ri^3_x\p_{\ri^3}\right),
\end{gather*}
where
$\Gamma$ runs through the set
$
\{\tilde{\mathscr J}^\kappa\tilde q,\,
\tilde{\mathscr D}_y^\iota\tilde{\mathscr J}^\kappa\tilde q,\,
\tilde{\mathscr D}_z^\iota\tilde{\mathscr J}^\kappa\tilde q,\,
\kappa\in\mathbb N_0,\,\iota\in\mathbb N\}
$
with
\begin{gather*}
\tilde{\mathscr D}_y:=-\frac1{\ri^1_x}\big(\mathscr D_t+(\ri^1+\ri^2-1)\mathscr D_x\big),\quad
\tilde{\mathscr D}_z:=-\frac1{\ri^2_x}\big(\mathscr D_t+(\ri^1+\ri^2+1)\mathscr D_x\big),
\\
\tilde{\mathscr J}:=\frac{\ri^1}2\tilde{\mathscr D}_y+\frac{\ri^2}2\tilde{\mathscr D}_z,\quad
\tilde q:={\rm e}^{(\ri^1-\ri^2)/2}\big(x-(\ri^1+\ri^2+1)t\big),
\end{gather*}
the parameter function~$\Phi=\Phi(\ri^1,\ri^2)$ runs through the solution set of
the Klein--Gordon equation $\Phi_{\ri^1\ri^2}=-\Phi/4$,
and the parameter function~$\Omega$ runs through the set of smooth functions of a finite, but unspecified number of
$\omega^\kappa:=({\rm e}^{\ri^2-\ri^1}\mathscr D_x)^\kappa\ri^3$, $\kappa\in\mathbb N_0$.
\end{theorem}

\begin{proof}
For computing the counterpart of an element $Q=\chi\p_q+\theta\p_s\in\hat{\mathfrak S}^{\rm q}$
in~$\hat\Sigma^{\rm q}$, one should make the following steps:
\begin{itemize}\itemsep=0ex
\item
prolong the generalized vector field~$Q$ to~$p$ in view of~\eqref{eq:IDFMSupersystemC},
\item
push forward the prolonged vector field by an appropriate prolongation of the transformation~\eqref{eq:IDFMInverseToTransReducingToKGEq},
\item
convert the obtained image to the evolutionary form and
\item
substitute for all derivatives of~$\ri$ with differentiation with respect to~$t$
in view of the system~\eqref{eq:IDFMDiagonalizedSystem} and its differential consequences.
\end{itemize}
This procedure gives the generalized vector field
\[
\tilde Q=
-{\rm e}^{(\ri^2-\ri^1)/2}\tilde\chi\ri^1_x\p_{\ri^1}
-{\rm e}^{(\ri^2-\ri^1)/2}(\tilde{\mathscr D}_z\tilde\chi)\ri^2_x\p_{\ri^2}
+\big(\theta-\tfrac12{\rm e}^{(\ri^2-\ri^1)/2}(\tilde{\mathscr D}_z\tilde\chi+\tilde\chi)\ri^3_x\big)\p_{\ri^3}.
\]
Here and in what follows tildes mark the counterparts of involved operators and differential functions
that are computed according to the procedure.

The ideal~$\hat{\mathfrak S}^{\rm q}_s$ of~$\hat{\mathfrak S}^{\rm q}$ corresponds to
and is isomorphic to the ideal~$\hat\Sigma^{\rm q}_3$ of~$\hat\Sigma^{\rm q}$,
and the form of elements of~$\hat\Sigma^{\rm q}_3$, $\check{\mathcal W}(\Omega)$,
is already known.
The generalized vector field~$q\p_q$ is mapped to $-\check{\mathcal D}$.
We also prolong each generalized vector field of the form $Q^{\zeta,0}:=(\mathscr D_y+1)\zeta\p_q$ from~$\mathfrak A$
to~$s$ according to~\eqref{eq:IDFM:GenSymsOfSupersystemProlongationToTheta}
and then employ the above procedure, getting the generalized vector field
\begin{gather*}
\tilde Q^{\zeta,0}=-{\rm e}^{(\ri^2-\ri^1)/2}\left(
\ri^1_x(\tilde{\mathscr D}_y+1)\tilde\zeta\p_{\ri^1}
+\ri^2_x(\tilde{\mathscr D}_z+1)\tilde\zeta\p_{\ri^2}
+2\ri^3_x\tilde\zeta\p_{\ri^3}\right),
\end{gather*}
where
$\zeta=\zeta[q]$ runs through the characteristics of generalized vector fields in $\hat{\mathfrak K}^{\rm q}$ and is defined up to summands proportional to ${\rm e}^{-y-z}$,
and $\tilde\zeta$ denotes the pullback of~$\zeta$
by the infinite prolongation of the transformation~\eqref{eq:IDFMTransReducingToKGEq}.
According to the splitting $\hat{\mathfrak K}^{\rm q}=\hat\Lambda^{\rm q}\lsemioplus\hat{\mathfrak K}^{-\infty}$,
for $\zeta\p_q\in\hat\Lambda^{\rm q}$ and $\zeta\p_q\in\hat{\mathfrak K}^{-\infty}$
we obtain generalized vector fields of the forms $-\check{\mathcal R}(\Gamma)$ and $-\check{\mathcal P}(\Phi)$,
respectively, where $\Gamma\p_q$ can be assumed to run through the chosen basis of~$\hat\Lambda^{\rm q}$,
and the parameter function~$\Phi=\Phi(\ri^1,\ri^2)$ runs through the solution set of
the Klein--Gordon equation $\Phi_{\ri^1\ri^2}=-\Phi/4$
and is defined up to summands proportional to~${\rm e}^{(\ri^2-\ri^1)/2}$.
\end{proof}

\begin{remark}\label{rem:IDFM:indeterminacy}
The subspaces~$\mathcal I^1$ and~$\mathcal I^2$
that consist of all generalized vector fields of the forms
$\check{\mathcal P}(\Phi)$ and $\check{\mathcal W}(\Omega)$
from the algebra~$\hat\Sigma^{\rm q}$, respectively,
are (infinite-dimensional) ideals of~$\hat\Sigma^{\rm q}$.
Moreover, the ideal~$\mathcal I^1$ is commutative.
Since $\check{\mathcal P}({\rm e}^{\ri^2-\ri^1})=\check{\mathcal W}(\omega^1)={\rm e}^{\ri^2-\ri^1}\ri^3_x\p_{\ri^3}$,
these ideals are not disjoint,
$\mathcal I^1\cap\mathcal I^2=\langle {\rm e}^{\ri^2-\ri^1}\ri^3_x\p_{\ri^3}\rangle$,
which displays the above indeterminacy of~$\Phi$.
\noprint{
Among the commutation relations of the algebra~$\hat\Sigma^{\rm q}$ are
\begin{gather*}
[\check{\mathcal W}(\Omega^1),\check{\mathcal W}(\Omega^2)]=\check{\mathcal W}(\Omega^{12})
\quad\mbox{with}\quad
\Omega^{12}:=\sum_{\iota=0}^\infty\big((\hat{\mathscr A}^\iota \Omega^1)\Omega^2_{\omega^\iota}-(\hat{\mathscr A}^\iota \Omega^2)\Omega^1_{\omega^\iota}\big),\\
[\check{\mathcal W}(\Omega),\check{\mathcal P}(\Phi)]=0,\quad
[\check{\mathcal P}(\Phi^1),\check{\mathcal P}(\Phi^2)]=0,\quad
[\check{\mathcal D},\check{\mathcal R}(\Gamma)]=0,\quad
[\check{\mathcal D},\check{\mathcal P}(\Phi)]=\check{\mathcal P}(\Phi).
\end{gather*}
}
\end{remark}

\begin{remark}\label{thm:IDFM1stOrderGenSyms}
The algebra of first-order reduced generalized symmetries of the system~\eqref{eq:IDFMDiagonalizedSystem}
can be identified with the subspace of~$\hat\Sigma^{\rm q}$ spanned by
$\check{\mathcal D}$, $\check{\mathcal R}(\tilde q)$, $\check{\mathcal R}(\tilde{\mathscr D}_z\tilde q)$,
$\check{\mathcal P}(\Phi)$, $\check{\mathcal W}(\Omega)$,
where the parameter function~$\Phi=\Phi(\ri^1,\ri^2)$ runs through the solution set of
the Klein--Gordon equation $\Phi_{\ri^1\ri^2}=-\Phi/4$,
and the parameter function~$\Omega$ runs through the set of smooth functions
of $\omega^0=\ri^3$ and $\omega^1={\rm e}^{\ri^2-\ri^1}\ri^3_x$.
As was noted in \cite[Remark~19]{OpanasenkoBihloPopovychSergyeyev2020},
this subspace is closed with respect to the Lie bracket of generalized vector fields,
and thus we can call it an algebra.
The indicated property is shared by all strictly hyperbolic diagonalizable hydrodynamic-type systems. %, see the appendix.
In the notation of \cite[Theorem~18]{OpanasenkoBihloPopovychSergyeyev2020},
\[
\check{\mathcal R}(\tilde q)=2(\check{\mathcal D}-\check{\mathcal G}_1),\quad
\check{\mathcal R}(\tilde{\mathscr D}_z\tilde q)=2(\check{\mathcal D}+\check{\mathcal G}_1+\check{\mathcal G}_2),
\]
where
$\check{\mathcal G_1}=(t \ri^1_x-1)\p_{\ri^1}+t\ri^2_x\p_{\ri^2}+t\ri^3_x\p_{\ri^3}$ and
$\check{\mathcal G_2}=\p_{\ri^1}-\p_{\ri^2}$.
Moreover, the generalized vector fields
\begin{gather}\label{eq:IDFMGenSymsEquivToLieSyms}
\check{\mathcal D},\quad \check{\mathcal G_1},\quad \check{\mathcal G_2},\quad
\check{\mathcal P}\big((\ri^1+\ri^2){\rm e}^{(\ri^1-\ri^2)/2}\big),\quad
\check{\mathcal P}\big({\rm e}^{(\ri^1-\ri^2)/2}\big),\quad \check{\mathcal W}(\Omega)
\end{gather}
with an arbitrary $\Omega$ depending on~$\ri^3$ only are
the evolutionary forms of Lie-symmetry vector fields~\smash{$-\hat{\mathcal D}$},
$-\hat{\mathcal G}_1$, $\hat{\mathcal G}_2$,
$2\hat{\mathcal P}^t$, $-2\hat{\mathcal P}^x$ and~$\hat{\mathcal W}(\Omega)$
of the system~\eqref{eq:IDFMDiagonalizedSystem}, respectively,
which span the entire Lie invariance algebra of this system.
Therefore, any element of~$\hat\Sigma^{\rm q}$ that does not belong to
the span of~\eqref{eq:IDFMGenSymsEquivToLieSyms}
is a genuinely generalized symmetry of the system~\eqref{eq:IDFMDiagonalizedSystem}.
\end{remark}

\section{Cosymmetries}\label{sec:IDFMCosyms}

The space~$\Upsilon$ of cosymmetries of the system~\eqref{eq:IDFMDiagonalizedSystem}
can be computed in a way that is similar to the computation of generalized symmetries
and involves the partial coupling of this system
and the linearizability of the subsystem~\eqref{eq:IDFMDEq1}--\eqref{eq:IDFMDEq2}
by the hodograph transformation.
Let $\Upsilon^{\rm triv}\subset\Upsilon$ denote the space of trivial cosymmetries of the system~\eqref{eq:IDFMDiagonalizedSystem},
which vanish on solutions thereof.
The quotient space $\Upsilon^{\rm q}=\Upsilon/\Upsilon^{\rm triv}$
can be identified, e.g., with the subspace
that consists of canonical representatives of cosymmetries,
$\hat\Upsilon^{\rm q}=\big\{(\lambda^i[\ri],i=1,2,3)\in\Upsilon\big\}$.

\begin{theorem}\label{thm:IDFM:Cosyms}
The space~$\hat\Upsilon^{\rm q}$ of canonical representatives of cosymmetries
is spanned by cosymmetries from three families,
\begin{enumerate}
\item
${\rm e}^{\ri^1-\ri^2}\big(\Omega,-\Omega,(\hat{\mathscr A}\Omega)/\omega^1\big)$
with the operator $\hat{\mathscr A}=\sum_{\kappa=0}^{\infty}\omega^{\kappa+1}\p_{\omega^\kappa}$ and
with %the parameter function
$\Omega$ running through the space of smooth functions
of a finite, but unspecified number of \smash{$\omega^\kappa=({\rm e}^{\ri^2-\ri^1}\mathscr D_x)^\kappa\ri^3$}, $\kappa\in\mathbb N_0$.
\item
${\rm e}^{(\ri^1-\ri^2)/2}(-2\Phi_{\ri^1},\,\Phi,0)$,
where the parameter function~$\Phi=\Phi(\ri^1,\ri^2)$
runs through the solution space of the Klein--Gordon equation $\Phi_{\ri^1\ri^2}=-\Phi/4$.
\item
${\rm e}^{(\ri^1-\ri^2)/2}\big(-\tilde{\mathscr D}_y\tilde{\mathfrak Q}\tilde q,\, \tilde{\mathfrak Q}\tilde q,\,0\,\big)$,
where
the operator~$\tilde{\mathfrak Q}$ runs through the set
\begin{gather*}
\big\{\tilde{\mathscr J}^\kappa\!,\,
\tilde{\mathscr J}^\kappa\tilde{\mathscr D}_y^\iota,\,
\tilde{\mathscr J}^\kappa\tilde{\mathscr D}_z^\iota,\,\kappa\in\mathbb N_0,\,\iota\in\mathbb N\big\},
\\
\hspace*{-\mathindent}\mbox{and}
\\[-1ex]
\tilde{\mathscr D}_y:=-\frac1{\ri^1_x}\big(\mathscr D_t+(\ri^1+\ri^2-1)\mathscr D_x\big),\quad
\tilde{\mathscr D}_z:=-\frac1{\ri^2_x}\big(\mathscr D_t+(\ri^1+\ri^2+1)\mathscr D_x\big),
\\
\tilde{\mathscr J}:=\frac{\ri^1}2\tilde{\mathscr D}_y+\frac{\ri^2}2\tilde{\mathscr D}_z,\quad
\tilde q:={\rm e}^{(\ri^1-\ri^2)/2}\big(x-(\ri^1+\ri^2+1)t\big).
\end{gather*}
\end{enumerate}
\end{theorem}

\begin{proof}
The space~$\hat\Upsilon^{\rm q}$ coincides with the solution space of the system
\begin{subequations}\label{eq:DetEqsForUpsilonq}
\begin{gather}
\mathscr D_t\lambda^1+(\ri^1+\ri^2+1)\mathscr D_x\lambda^1=\ri^2_x(\lambda^2-\lambda^1)+\ri^3_x\lambda^3,\label{eq:DetEqsForUpsilonqA}\\
\mathscr D_t\lambda^2+(\ri^1+\ri^2-1)\mathscr D_x\lambda^2=\ri^1_x(\lambda^1-\lambda^2)+\ri^3_x\lambda^3,\label{eq:DetEqsForUpsilonqB}\\
\mathscr D_t\lambda^3+(\ri^1+\ri^2)\mathscr D_x\lambda^3+(\ri^1_x+\ri^2_x)\lambda^3=0,\label{eq:DetEqsForUpsilonqC}
\end{gather}
\end{subequations}
which is formally adjoint to the system~\eqref{eq:DetEqsForSigmaq} for generalized symmetries of~\eqref{eq:IDFMDiagonalizedSystem}.
The substitution $(\lambda^1,\lambda^2,\lambda^3)={\rm e}^{\ri^1-\ri^2}(\tilde\lambda^1,\tilde\lambda^2,\tilde\lambda^3)$
reduces the system~\eqref{eq:DetEqsForUpsilonq} to
\begin{subequations}\label{eq:DetEqsForUpsilonqMod}
\begin{gather}
\mathscr D_t\tilde\lambda^1+(\ri^1+\ri^2+1)\mathscr D_x\tilde\lambda^1=\ri^2_x(\tilde\lambda^1+\tilde\lambda^2)+\ri^3_x\tilde\lambda^3,\label{eq:DetEqsForUpsilonqModA}\\
\mathscr D_t\tilde\lambda^2+(\ri^1+\ri^2-1)\mathscr D_x\tilde\lambda^2=\ri^1_x(\tilde\lambda^1+\tilde\lambda^2)+\ri^3_x\tilde\lambda^3,\label{eq:DetEqsForUpsilonqModB}\\
\mathscr D_t\tilde\lambda^3+(\ri^1+\ri^2)\mathscr D_x\tilde\lambda^3=0.\label{eq:DetEqsForUpsilonqModC}
\end{gather}
\end{subequations}
Again using the modified coordinates on~$\mathcal S^{(\infty)}$,
we will show below that the general solution 
of the system~\eqref{eq:DetEqsForUpsilonqMod} 
can be represented in the form  
\begin{gather}\label{eq:IDFMCosymsLambda3Representation}
\tilde\lambda^1=\tilde\lambda^{1\rm h}+\Omega,\quad 
\tilde\lambda^2=\tilde\lambda^{2\rm h}-\Omega,\quad
\tilde\lambda^3=\frac{\hat{\mathscr A}\Omega}{\omega^1},
\end{gather}
where $\hat{\mathscr A}=\sum_{\kappa=0}^{\infty}\omega^{\kappa+1}\p_{\omega^\kappa}$,
$\Omega$ runs through the space of smooth functions of a finite, but unspecified number of~$\omega$'s, 
and $(\tilde\lambda^{1\rm h},\tilde\lambda^{2\rm h})$ with $\tilde\lambda^{j\rm h}=\tilde\lambda^{j\rm h}[\ri^1,\ri^2]$, $j=1,2$,
is the general solution of the subsystem~\eqref{eq:DetEqsForUpsilonqModA}--\eqref{eq:DetEqsForUpsilonqModB} with $\tilde\lambda^3=0$,
%\begin{subequations}\label{eq:DetEqsForUpsilonqMod2}
\begin{gather*}
\mathscr D_t\tilde\lambda^{1\rm h}+(\ri^1+\ri^2+1)\mathscr D_x\tilde\lambda^{1\rm h}=\ri^2_x(\tilde\lambda^{1\rm h}+\tilde\lambda^{2\rm h}),%\label{eq:DetEqsForUpsilonqMod2A}
\\
\mathscr D_t\tilde\lambda^{2\rm h}+(\ri^1+\ri^2-1)\mathscr D_x\tilde\lambda^{2\rm h}=\ri^1_x(\tilde\lambda^{1\rm h}+\tilde\lambda^{2\rm h}).%\label{eq:DetEqsForUpsilonqMod2B}
\end{gather*}
The counterpart $(\lambda^{1\rm h},\lambda^{2\rm h})={\rm e}^{\ri^1-\ri^2}(\tilde\lambda^{1\rm h},\tilde\lambda^{2\rm h})$
of $(\tilde\lambda^{1\rm h},\tilde\lambda^{2\rm h})$ satisfies
the subsystem~\eqref{eq:DetEqsForUpsilonqA}--\eqref{eq:DetEqsForUpsilonqB} with $\lambda^3=0$,
\begin{subequations}\label{eq:DetEqsForUpsilonq0}
\begin{gather}
\mathscr D_t\lambda^{1\rm h}+(\ri^1+\ri^2+1)\mathscr D_x\lambda^{1\rm h}=\ri^2_x(\lambda^{2\rm h}-\lambda^{1\rm h}),\label{eq:DetEqsForUpsilonq0A}\\
\mathscr D_t\lambda^{2\rm h}+(\ri^1+\ri^2-1)\mathscr D_x\lambda^{2\rm h}=\ri^1_x(\lambda^{1\rm h}-\lambda^{2\rm h}).\label{eq:DetEqsForUpsilonq0B}
\end{gather}
\end{subequations}
Therefore, the triple $\lambda=(\lambda^1,\lambda^2,\lambda^3)$ belongs to~$\hat\Upsilon^{\rm q}$
if and only if it can be represented, in the above notation,  in the form
\begin{gather}\label{eq:IDFMCosymsGenForm}
\lambda={\rm e}^{\ri^1-\ri^2}\big(\Omega,-\Omega,(\hat{\mathscr A}\Omega)/\omega^1\big)+(\lambda^{1\rm h},\lambda^{2\rm h},0).
\end{gather}
The substitution $(\lambda^{1\rm h},\lambda^{2\rm h})={\rm e}^{(\ri^1-\ri^2)/2}(\hat\lambda^1,\hat\lambda^2)$
reduces the system~\eqref{eq:DetEqsForUpsilonq0} to the system
\begin{gather*}
\mathscr D_t\hat\lambda^1+(\ri^1+\ri^2+1)\mathscr D_x\hat\lambda^1=\ri^2_x\hat\lambda^2,\\
\mathscr D_t\hat\lambda^2+(\ri^1+\ri^2-1)\mathscr D_x\hat\lambda^2=\ri^1_x\hat\lambda^1,
\end{gather*}
which can be rewritten in terms of the operators~$\tilde{\mathscr D}_y$ and~$\tilde{\mathscr D}_z$,
\[\tilde{\mathscr D}_z\hat\lambda^1=-\hat\lambda^2,\quad \tilde{\mathscr D}_y\hat\lambda^2=-\hat\lambda^1.\]
Therefore, both the components~$\hat\lambda^1$ and~$\hat\lambda^2$ satisfy
the image of the equation~\eqref{eq:IDFMSupersystemGenSymsDetEqA}
under the transformation~\eqref{eq:IDFMInverseToTransReducingToKGEq}
and thus are the reduced forms of the pullbacks of characteristics of generalized vector fields from $\hat{\mathfrak K}^{\rm q}$ by this transformation.
As a result, we obtain the families of cosymmetries of the system~\eqref{eq:IDFMDiagonalizedSystem} 
that are presented in the theorem.
The first and second summands in~\eqref{eq:IDFMCosymsGenForm} correspond to 
the first family and the span of the second and the third families, respectively. 

Now we prove the representation~\eqref{eq:IDFMCosymsLambda3Representation} 
by induction on the order $\ord_\omega(\tilde\lambda^1-\tilde\lambda^2)\in\{-\infty\}\cup\mathbb N_0$.
In view of Lemma~\ref{lem:IDFM:DiffFunctionsOfOmega},
any solution of the equation~\eqref{eq:DetEqsForUpsilonqModC}, which can be shortly rewritten as $\mathscr B\tilde\lambda^3=0$,
is a smooth function of a finite number of~$\omega$'s.
We take the sum and the difference of the equations~\eqref{eq:DetEqsForUpsilonqModA} and~\eqref{eq:DetEqsForUpsilonqModB},
additionally writing them, after multiplying by ${\rm e}^{\ri^2-\ri^1}$, in terms of the operator~$\mathscr A$ and~$\mathscr B$, 
\begin{subequations}\label{eq:DetEqsForUpsilonqPM}
\begin{gather}\label{eq:DetEqsForUpsilonqP}
{\rm e}^{\ri^2-\ri^1}\mathscr B(\tilde\lambda^1+\tilde\lambda^2)+\mathscr A(\tilde\lambda^1-\tilde\lambda^2)
=(\tilde\lambda^1+\tilde\lambda^2)\mathscr A(\ri^1+\ri^2)+2\omega^1\tilde\lambda^3,
\\\label{eq:DetEqsForUpsilonqM}
{\rm e}^{\ri^2-\ri^1}\mathscr B(\tilde\lambda^1-\tilde\lambda^2)+\mathscr A(\tilde\lambda^1+\tilde\lambda^2)
=(\tilde\lambda^1+\tilde\lambda^2)\mathscr A(\ri^2-\ri^1).
\end{gather}
\end{subequations}

\noindent{\it Base case.} 
Let $\ord_\omega(\tilde\lambda^1-\tilde\lambda^2)=-\infty$.
The equation~\eqref{eq:DetEqsForUpsilonqM} implies that $\ord_\omega(\tilde\lambda^1+\tilde\lambda^2)=-\infty$ as well, 
i.e., both $\tilde\lambda^1$ and $\tilde\lambda^2$ do not depend on~$\omega$'s.
Then we obtain from the equation~\eqref{eq:DetEqsForUpsilonqP} that 
the summand $2\omega^1\tilde\lambda^3$ does not depend on~$\omega$'s as well.
Recalling that $\tilde\lambda^3$ depends at most on a finite number of~$\omega$'s, 
we educe that $c:=\omega^1\tilde\lambda^3$ is a constant, i.e., 
$\tilde\lambda^3=c/\omega^1=c{\rm e}^{\ri^1-\ri^2}/\ri^3_x$. 
We substitute $(\lambda^1,\lambda^2)={\rm e}^{(\ri^1-\ri^2)/2}(\hat\lambda^1,\hat\lambda^2)$
into the equations~\eqref{eq:DetEqsForUpsilonqModA} and~\eqref{eq:DetEqsForUpsilonqModB}
and rewrite them in the notation of Section~\ref{sec:IDFMSolutionsViaEssentialSubsystem} as 
\[
\tilde{\mathscr D}_z\hat\lambda^1=-\hat\lambda^2+\frac c2{\rm e}^{\ri^1-\ri^2}K^1,\quad
\tilde{\mathscr D}_y\hat\lambda^2=-\hat\lambda^1-\frac c2{\rm e}^{\ri^1-\ri^2}K^2.
\]
We carry out the transformation~\eqref{eq:IDFMTransReducingToKGEq} 
restricted to the spaces with the coordinates $(t,x,\ri^1,\ri^2)$ and $(y,z,p,q)$ 
and then exclude derivatives of~$p$ in view of the equation~\eqref{eq:IDFMSupersystemC} 
and its differential consequences. 
As a result, we derive the system 
\begin{gather}\label{eq:IDFMCosymsReducedSystemForBaseCase}
\mathscr D_z\breve\lambda^1=-\breve\lambda^2+\frac c2{\rm e}^{2y+2z}K^1,\quad
\mathscr D_y\breve\lambda^2=-\breve\lambda^1-\frac c2{\rm e}^{2y+2z}K^2,
\end{gather}
where the differential function~$\breve\lambda^i=\breve\lambda^i[q]$ is the image of $\hat\lambda^i$
under the above transformation, $i=1,2$, see the notation in Section~\ref{sec:IDFMPreliminaries}.
We solve the first equation of~\eqref{eq:IDFMCosymsReducedSystemForBaseCase} 
with respect to~$\breve\lambda^2$ and substitute the obtained expression 
$\breve\lambda^2=-\mathscr D_z\breve\lambda^1+\frac12c{\rm e}^{2y+2z}K^1$ 
into the second equation, deriving %the equation
\begin{gather}\label{eq:IDFMCosymsReducedEqForBaseCase}
\mathscr D_y\mathscr D_z\beta[q]=\beta[q]+c{\rm e}^{2y+2z}(q_{zz}+q_y-q_z-q).
\end{gather}
with respect to $\beta:=\breve\lambda^1$.
Therefore, the system~\eqref{eq:DetEqsForUpsilonqMod} with $\tilde\lambda^3=c/\omega^1$ 
has a solution if and only if the equation~\eqref{eq:IDFMCosymsReducedEqForBaseCase} has. 
In this way, we reduce the proof in the base case to studying 
the existence of solutions of the equation~\eqref{eq:IDFMCosymsReducedEqForBaseCase}.

Given a differential function~$\alpha=\alpha[q]$ that is affine in totality of involved derivatives of~$q$, 
any solution~$\beta=\beta[q]$ of the equation $\mathscr D_y\mathscr D_z\beta=\beta+\alpha$ 
has the same property. 
Indeed, we fix an arbitrary solution~$\beta$ of~\eqref{eq:IDFMCosymsReducedEqForBaseCase}
and substitute $q=q^0+\varepsilon_1q^1+\varepsilon_2q^2$ into it. 
Here $\varepsilon_1$ and~ $\varepsilon_2$ are constant parameters, 
and  $q^0$, $q^1$ and~$q^2$ are arbitrary solutions of  the equation~\eqref{eq:IDFMSupersystemA}, 
which is the (1+1)-dimensional Klein--Gordon equation for~$q$ in light-cone variables, $q^i_{yz}=q^i$, $i=0,1,2$.
We take the mixed derivative of the equation for~$\beta$ with substituted~$q$ with respect to $(\varepsilon_1,\varepsilon_2)$ 
at $(\varepsilon_1,\varepsilon_2)=(0,0)$ to derive the equation
\begin{gather}\label{eq:IDFMCosymsAuxiliaryEqForAffineness}
\sum_{\iota,\iota'=-\ord\beta}^{\ord\beta}\big(
\mathscr D_y\mathscr D_z(\beta_{q_\iota q_{\iota'}}[q^0]q^1_\iota q^2_{\iota'})
-\beta_{q_\iota q_{\iota'}}[q^0]q^1_\iota q^2_{\iota'}\big)=0,
\end{gather}
which can be split with respect to $\{q_\iota,q_{\iota'},\iota,\iota'=-\ord\beta-1,\dots,\ord\beta+1\}$.
Suppose that $\beta_{q_\iota q_{\iota'}}\ne0$ for some $(\iota,\iota')$. 
Let $\iota_0=\max\big\{\iota\mid\exists\iota'\colon\beta_{q_\iota q_{\iota'}}\ne0\big\}$ and 
$\iota'_0=\min\big\{\iota'\mid\beta_{q_{\iota_0}q_{\iota'}}\ne0\big\}$.
Collecting the coefficients of~$q^1_{\iota_0+1}q^2_{\iota'_0-1}$ in the equation~\eqref{eq:IDFMCosymsAuxiliaryEqForAffineness} 
gives \smash{$\beta_{q_{\iota_0}q_{\iota'_0}}=0$} 
contradicting the inequality \smash{$\beta_{q_{\iota_0}q_{\iota'_0}}\ne0$}.
Therefore, $\beta_{q_\iota q_{\iota'}}=0$ for any $(\iota,\iota')\in\mathbb Z^2$.

In view of the claim proved in the previous paragraph, 
we can represent each fixed solution of the equation~\eqref{eq:IDFMCosymsReducedEqForBaseCase}
in the form \[\beta=\sum_{\iota=-n}^n\beta^\iota(y,z)q_\iota+\beta^{00}(y,z),\] 
where $n:=\ord\beta$, and 
the coefficients $\beta^\iota$, $\iota=-n,\dots,n$, and $\beta^{00}$ are smooth functions of~$(y,z)$. 
Up to adding a solution of the homogeneous counterpart of the equation~\eqref{eq:IDFMCosymsReducedEqForBaseCase}, 
without loss of generality we can assume that $n>2$ and $\beta^{00}=0$.
The equation~\eqref{eq:IDFMCosymsReducedEqForBaseCase} splits into the following system 
for the coefficients of~$\beta$:
\begin{gather*}%\label{eq:IDFMCosymsSplitReducedEqForBaseCase}
\Delta_{-n-1}\colon\ \beta^{-n}_y=0,\quad
\Delta_{-n}\colon\ \beta^{-n+1}_y=0,\\
\Delta_{\kappa-n}\colon\ \beta^{\kappa-n-1}_z+\beta^{\kappa-n}_{yz}+\beta^{\kappa-n+1}_y=\alpha^{\kappa-n},\quad \kappa=1,\dots,2n-1,\\
\Delta_n\colon\ \beta^{n-1}_z=0,\quad
\Delta_{n+1}\colon\ \beta^n_z=0,
\end{gather*}
where $\alpha^{-2}=-\alpha^{-1}=-\alpha^0=\alpha^1=c{\rm e}^{2y+2z}$ and the other $\alpha^\iota$ are zero. 
The equation $\Delta_\iota$ is constituted by the coefficients of~$q_\iota$.
For each $\kappa\in\{1,\dots,2n-1\}$,
we solve the equation~$\Delta_{\kappa-n}$ with respect to $\beta^{\kappa-n+1}_y$, 
differentiate the result $\kappa$ times with respect to~$y$, 
\[
\frac{\p^{\kappa+1}\beta^{\kappa-n+1}}{\p y^{\kappa+1}}=\frac{\p^\kappa\alpha^{\kappa-n}}{\p y^\kappa}
-\frac{\p^{\kappa+1}\beta^{\kappa-n-1}}{\p y^\kappa\p z}-\frac{\p^{\kappa+2}\beta^{\kappa-n}}{\p y^{\kappa+1}\p z}
\]
and substitute for the last two derivatives in view of differential consequences of the previous equations. 
As a result, we obtain 
\[\frac{\p^{\kappa+1}\beta^{\kappa-n+1}}{\p y^{\kappa+1}}=c_\kappa{\rm e}^{2y+2z},\] 
where $c_\kappa=0$, $\kappa=1,\dots,n-3$, 
$c_{n-2}=2^{n-2}c$, $c_{n-1}=-3\cdot2^{n-1}c$, $c_n=2^{n+2}c$, and $c_{n+1}=-2^{n+3}c$. 

We can prove by induction on~$\kappa$ that $c_\kappa=(-1)^{\kappa-n}2^{\kappa+2}c$, $\kappa=n+1,\dots,2n-1$.
The base case $\kappa=n+1$ is given by the above equality $c_{n+1}=-2^{n+3}c$, 
and the induction step follows from the equality $c_{\kappa+1}=-4(c_\kappa+c_{\kappa-1})$ for $\kappa>n+1$. 
Therefore, the equation $\Delta_{n+1}$: $\beta^n_z=0$ implies 
that $c=0$, and we obtain the representation~\eqref{eq:IDFMCosymsLambda3Representation} with $\Omega=0$. 

\medskip\par\noindent{\it Induction step.} 
Suppose that the representation~\eqref{eq:IDFMCosymsLambda3Representation} holds 
if $\ord_\omega(\tilde\lambda^1-\tilde\lambda^2)<\kappa\in\mathbb N_0$ 
and prove this representation for $\ord_\omega(\tilde\lambda^1-\tilde\lambda^2)=\kappa$.
In view of the equation~\eqref{eq:DetEqsForUpsilonqM}, 
under the last condition we have $\ord_\omega(\tilde\lambda^1+\tilde\lambda^2)<\kappa$.
Then the equation~\eqref{eq:DetEqsForUpsilonqP} implies that $\ord_\omega(\omega^1\tilde\lambda^3)=\kappa+1$. 
Differentiating the equation~\eqref{eq:DetEqsForUpsilonqP} with respect to~$\omega^{\kappa+1}$, 
we derive $(\tilde\lambda^1-\tilde\lambda^2)_{\omega^\kappa}=2(\omega^1\tilde\lambda^3)_{\omega^{\kappa+1}}$, 
and thus both the left- and the right-hand sides of the last equality depends at most on $\omega^0$,~\dots,~$\omega^\kappa$, 
i.e., there exists a smooth function $\Upsilon=\Upsilon(\omega^0,\dots,\omega^\kappa)$ 
such that \[(\tilde\lambda^1-\tilde\lambda^2)_{\omega^\kappa}=2(\omega^1\tilde\lambda^3)_{\omega^{\kappa+1}}=2\Upsilon.\]
Let $\breve\Upsilon=\breve\Upsilon(\omega^0,\dots,\omega^\kappa)$ 
be a fixed antiderivative of~$\Upsilon$ with respect to~$\omega^\kappa$, $\breve\Upsilon_{\omega^\kappa}=\Upsilon$. 
Define
\[
\breve\lambda^1:=\tilde\lambda^1-\breve\Upsilon,\quad 
\breve\lambda^2:=\tilde\lambda^2+\breve\Upsilon,\quad
\breve\lambda^3:=\tilde\lambda^3-\frac{\hat{\mathscr A\breve\Upsilon}}{\omega^1}.
\]
The tuple $(\breve\lambda^1,\breve\lambda^2,\breve\lambda^3)$ satisfies the system~\eqref{eq:DetEqsForUpsilonqMod}, 
and \[(\breve\lambda^1-\breve\lambda^2)_{\omega^\kappa}=(\tilde\lambda^1-\tilde\lambda^2)_{\omega^\kappa}-2\Upsilon=0,\] 
i.e., \mbox{$\breve\kappa:=\ord_\omega(\breve\lambda^1-\breve\lambda^2)<\kappa$}.
By the induction hypothesis, this tuple can be represented in the form~\eqref{eq:IDFMCosymsLambda3Representation} 
with some smooth function~$\breve\Omega=\breve\Omega(\omega^0,\dots,\omega^{\breve\kappa})$. 
Setting $\Omega=\breve\Omega+\breve\Upsilon$, 
we derive the representation~\eqref{eq:IDFMCosymsLambda3Representation} for 
$(\tilde\lambda^1,\tilde\lambda^2,\tilde\lambda^3)$ with the same~$\lambda^{1\rm h}$ and~$\lambda^{2\rm h}$.
\end{proof}

\begin{remark}\label{rem:IDFM:IntersectionOfCosymSubspaces}
The first and the second families from Theorem~\ref{thm:IDFM:Cosyms}, which are linear spaces,
are not disjoint in the sense of linear spaces.
Their intersection is one-dimensional and is spanned by the cosymmetry ${\rm e}^{\ri^1-\ri^2}(1,-1,0)$
corresponding to $\Omega=1$ and $\Phi=-{\rm e}^{(\ri^1-\ri^2)/2}$.
The span of these two families has the zero intersection with the span of the third family.
\end{remark}

\section{Conservation laws}\label{sec:IDFM:CLs}

\begin{theorem}\label{thm:IDFM:CLs}
The space of conservation laws of the system~\eqref{eq:IDFMDiagonalizedSystem}
is naturally isomorphic to the space spanned by the following conserved currents of this system:
\begin{enumerate}
\item
$\big(\,{\rm e}^{\ri^1-\ri^2}\Omega,\ (\ri^1+\ri^2){\rm e}^{\ri^1-\ri^2}\Omega\,\big)$,
where the parameter function~$\Omega$ runs through the space of smooth functions
of a finite, but unspecified number of \smash{$\omega^\kappa=({\rm e}^{\ri^2-\ri^1}\mathscr D_x)^\kappa\ri^3$}, $\kappa\in\mathbb N_0$,
and such two functions should be assumed equivalent if their difference
belongs to the image of the operator $\hat{\mathscr A}=\sum_{\kappa=0}^{\infty}\omega^{\kappa+1}\p_{\omega^\kappa}$.
\item
$\big(\,{\rm e}^{(\ri^1-\ri^2)/2}(2\Phi_{\ri^1}+\Phi),\
{\rm e}^{(\ri^1-\ri^2)/2}(2(\ri^1+\ri^2+1)\Phi_{\ri^1}+(\ri^1+\ri^2-1)\Phi)\,\big)$,
where the parameter function~$\Phi=\Phi(\ri^1,\ri^2)$
runs through the solution space of the Klein--Gordon equation $\Phi_{\ri^1\ri^2}=-\Phi/4$.
\item
$\big(\ri^2_x\tilde\rho+\ri^1_x\tilde\sigma,\ (\ri^1+\ri^2-1)\ri^2_x\tilde\rho+(\ri^1+\ri^2+1)\ri^1_x\tilde\sigma\big)$
with
$\tilde\rho=-\tilde q\tilde{\mathscr D}_z\tilde{\mathfrak Q}\tilde q$,
$\tilde\sigma=(\tilde{\mathscr D}_y\tilde q)\tilde{\mathfrak Q}\tilde q$,
where
the operator~$\tilde{\mathfrak Q}$ runs through the set
\begin{gather*}
\big\{\,\tilde{\mathscr J}^{\kappa'}\!,\,\kappa'\in2\mathbb N_0+1,\
(\tilde{\mathscr J}+\iota/2)^\kappa\tilde{\mathscr D}_y^\iota,\,
(\tilde{\mathscr J}-\iota/2)^\kappa\tilde{\mathscr D}_z^\iota,\,\kappa\in\mathbb N_0,\,\iota\in\mathbb N,\,\kappa+\iota\in2\mathbb N_0+1\,\big\},
\\
\hspace*{-\mathindent}\mbox{and}
\\[-1ex]
\tilde{\mathscr D}_y:=-\frac1{\ri^1_x}\big(\mathscr D_t+(\ri^1+\ri^2-1)\mathscr D_x\big),\quad
\tilde{\mathscr D}_z:=-\frac1{\ri^2_x}\big(\mathscr D_t+(\ri^1+\ri^2+1)\mathscr D_x\big),
\\
\tilde{\mathscr J}:=\frac{\ri^1}2\tilde{\mathscr D}_y+\frac{\ri^2}2\tilde{\mathscr D}_z,\quad
\tilde q:={\rm e}^{(\ri^1-\ri^2)/2}\big(x-(\ri^1+\ri^2+1)t\big).
\end{gather*}
\end{enumerate}
\end{theorem}

\begin{proof}
We compute the space of local conservation laws of the system~\eqref{eq:IDFMDiagonalizedSystem}
combining the direct method of finding conservation laws~\cite{PopovychIvanova2005b,Wolf2002},
which is based on the definition of conserved currents,
with using the linearization of the essential subsystem~\eqref{eq:IDFMDEq1}--\eqref{eq:IDFMDEq2}
to the (1+1)-dimensional Klein--Gordon equation.
Up to the equivalence of conserved currents,
meaning that they coincide on the solution set of the corresponding system of differential equations,
it suffices to consider only reduced conserved currents of the system~\eqref{eq:IDFMDiagonalizedSystem},
which are of the form $(\rho,\sigma)$, where $\rho=\rho[\ri]$ and $\sigma=\sigma[\ri]$.
A tuple $(\rho[\ri],\sigma[\ri])$ is a conserved current of the system~\eqref{eq:IDFMDiagonalizedSystem}
if and only if $\mathscr D_t\rho+\mathscr D_x\sigma=0$.
We should also take into account the equivalence of conserved currents up to adding null divergences,
which means that conserved currents $(\rho[\ri],\sigma[\ri])$ and $(\rho'[\ri],\sigma'[\ri])$
belong to the same conservation law if and only if
there exists a differential function $f=f[\ri]$
such that $\rho'=\rho+\mathscr D_xf$ and $\sigma'=\sigma-\mathscr D_tf$.

We associate an arbitrary reduced conserved current $(\rho[\ri],\sigma[\ri])$
of the system~\eqref{eq:IDFMDiagonalizedSystem}
with the modified density $\breve\rho:={\rm e}^{\ri^2-\ri^1}\rho$
and the modified flux $\breve\sigma=\sigma-(\ri^1+\ri^2)\rho$, i.e.,
\[(\rho,\sigma)=\big({\rm e}^{\ri^1-\ri^2}\breve\rho,(\ri^1+\ri^2){\rm e}^{\ri^1-\ri^2}\breve\rho+\breve\sigma\big),\]
and $\mathscr D_t\rho+\mathscr D_x\sigma={\rm e}^{\ri^1-\ri^2}(\mathscr B\breve\rho+\mathscr A\breve\sigma)$.
Therefore, the equality $\mathscr D_t\rho+\mathscr D_x\sigma=0$ for conserved currents
is equivalent to the equality $\mathscr B\breve\rho+\mathscr A\breve\sigma=0$
for modified conserved currents,
and the equivalence of conserved currents up to adding a null divergence
is modified to $\breve\rho'=\breve\rho+\mathscr Af$ and $\breve\sigma'=\breve\sigma-\mathscr Bf$.

Fixing a reduced conserved current $(\rho[\ri],\sigma[\ri])$
and using the modified coordinates on~$\mathcal S^{(\infty)}$,
we define $\kappa:=\max(\ord_\omega\breve\rho,\ord_\omega\breve\sigma)$
and prove by mathematical induction with respect to~$\kappa\in\{-\infty\}\cup\mathbb N_0$ that
up to adding a modified null divergence we have the representation
$\breve\rho=\breve\rho^1[r^1,r^2]+\breve\rho^0(\omega^0,\dots,\omega^\kappa)$
for some differential functions
$\breve\rho^0=\breve\rho^0(\omega^0,\dots,\omega^\kappa)$ and $\breve\rho^1=\breve\rho^1[r^1,r^2]$,
and $\breve\sigma=\breve\sigma[r^1,r^2]$.

The base case $\kappa=-\infty$ is obvious.

For the inductive step, we fix $\kappa\in\mathbb N_0$,
suppose that the above claim is true for all $\kappa'<\kappa$ and prove it for~$\kappa$.
Collecting coefficients of~$\omega^{\kappa+1}$ in the equality $\mathscr B\breve\rho+\mathscr A\breve\sigma=0$,
we derive $\breve\sigma_{\omega^\kappa}=0$, i.e., in fact $\ord_\omega\breve\sigma<\kappa$.
Then we differentiate the same equality twice with respect to $\omega^\kappa$,
which leads to $\mathscr B\breve\rho_{\omega^\kappa\omega^\kappa}=0$.
In view of Lemma~\ref{lem:IDFM:DiffFunctionsOfOmega},
this means that the $\breve\rho_{\omega^\kappa\omega^\kappa}$ can depend at most on $(\omega^0,\dots,\omega^\kappa)$.
Therefore, there exist differential functions
$\breve\rho^{10}=\breve\rho^{10}(\omega^0,\dots,\omega^\kappa)$,
$\breve\rho^{11}=\breve\rho^{11}[\ri]$ and $\breve\rho^{12}=\breve\rho^{12}[\ri]$ such that
$\ord_\omega\breve\rho^{11}<\kappa$, $\ord_\omega\breve\rho^{12}<\kappa$ and
$\breve\rho=\breve\rho^{12}[\ri]\omega^\kappa+\breve\rho^{11}[\ri]+\breve\rho^{10}(\omega^0,\dots,\omega^\kappa)$.
Since $\mathscr B\breve\rho^{10}=0$, the tuple $(\breve\rho^{10},0)$
is a modified conserved current of the system~\eqref{eq:IDFMDiagonalizedSystem}.
Hence the tuple $(\breve\rho^{12}\omega^\kappa+\breve\rho^{11},\breve\sigma)$
is a modified conserved current of this system as well.
Adding the modified null divergence
$(-\mathscr A\int\breve\rho^{12}\,{\rm d}\omega^{\kappa-1},
   \mathscr B\int\breve\rho^{12}\,{\rm d}\omega^{\kappa-1})$
to the latter modified conserved current,
we obtain an equivalent modified conserved current $(\breve\rho',\breve\sigma')$
with $\max(\ord_\omega\breve\rho',\ord_\omega\breve\sigma')<\kappa$.
The induction hypothesis implies that up to adding a modified null divergence,
the component~$\breve\rho'$ admits the representation
$\breve\rho'=\breve\rho^{21}[r^1,r^2]+\breve\rho^{20}(\omega^0,\dots,\omega^\kappa)$
for some differential functions
$\breve\rho^{20}=\breve\rho^{20}(\omega^0,\dots,\omega^\kappa)$ and $\breve\rho^{21}=\breve\rho^{21}[r^1,r^2]$,
and $\breve\sigma'=\breve\sigma'[r^1,r^2]$.
Setting $\breve\rho^0=\breve\rho^{10}+\breve\rho^{20}$, $\breve\rho^1=\breve\rho^{21}$
and $\breve\sigma=\breve\sigma'$, we complete the inductive step.

In other words, we have proved that up to adding a null divergence,
any conserved current of the system~\eqref{eq:IDFMDiagonalizedSystem}
can be represented as the sum of a conserved current from the first theorem's family
and of a conserved current of the form $(\rho[r^1,r^2],\sigma[r^1,r^2])$.
The subspace of conserved currents of the latter forms
is the pullback of the space of reduced conserved currents
of the essential subsystem~\eqref{eq:IDFMDEq1}--\eqref{eq:IDFMDEq2}
by the projection $(t,x,\ri)\to(t,x,\ri^1,\ri^2)$; cf.\ \cite[Proposition~3]{KunzingerPopovych2008}.
The latter space is naturally isomorphic to the space of conservation laws
of the essential subsystem~\eqref{eq:IDFMDEq1}--\eqref{eq:IDFMDEq2},
which is the pullback of the space of conservation laws
of the Klein--Gordon equation~\eqref{eq:IDFMSupersystemA}
with respect to the composition of the restriction
of the transformation~\eqref{eq:IDFMTransReducingToKGEq} to the space with coordinates $(t,x,\ri^1,\ri^2)$
(i.e., the $s$-component of this transformation should be neglected)
with the projection $(y,z,q,p)\to(y,z,q)$.
We take the space of conservation laws of the (1+1)-dimensional Klein--Gordon equation,
which was constructed in~\cite{OpanasenkoPopovych2018},
and perform the above pullbacks,
\[
\rho=-\frac12(\ri^2_x\tilde\rho_{\text{\tiny\rm KG}}+\ri^1_x\tilde\sigma_{\text{\tiny\rm KG}}),\quad
\sigma=-\frac12(V^2\ri^2_x\tilde\rho_{\text{\tiny\rm KG}}+V^1\ri^1_x\tilde\sigma_{\text{\tiny\rm KG}}),
\]
where $\tilde\rho_{\text{\tiny\rm KG}}$ and $\tilde\sigma_{\text{\tiny\rm KG}}$ are,
as differential functions, the pullbacks
of the density~$\rho_{\text{\tiny\rm KG}}$ and the flux~$\sigma_{\text{\tiny\rm KG}}$
of a conserved current of~\eqref{eq:IDFMSupersystemA}, respectively;
see \cite[Section~III]{PopovychIvanova2005b} or \cite[Proposition~1]{PopovychKunzingerIvanova2008}.
As a result, we obtain, up to the equivalence on solutions of the system~\eqref{eq:IDFMDiagonalizedSystem}
and up to rescaling of conserved currents,
the other families of the conserved currents of this system that are presented in the theorem.

More specifically, the equation~\eqref{eq:IDFMSupersystemA} is
the Euler--Lagrange equation for the Lagrangian $\mathrm L=-(q_yq_z+q^2)/2$.
Hence characteristics of generalized symmetries of this equation are also its cosymmetries, and vice versa.
The quotient algebra~$\mathfrak K^{\rm q}=\mathfrak K/\mathfrak K^{\rm triv}$ of generalized symmetries of~\eqref{eq:IDFMSupersystemA},
where $\mathfrak K$ and~$\mathfrak K^{\rm triv}$ are
the algebra (of evolutionary representatives) of generalized symmetries of the Lagrangian~\eqref{eq:IDFMSupersystemA}
and its ideal of trivial generalized symmetries,
is naturally isomorphic to the algebra $\tilde{\mathfrak K}^{\rm q}=\tilde\Lambda^{\rm q}\lsemioplus\tilde{\mathfrak K}^{-\infty}$,
which is the semidirect sum of the subalgebra
\[
\tilde\Lambda^{\rm q}:=
\big\langle\,(\mathrm J^\kappa q)\p_q,\,(\mathrm J^\kappa\mathrm D_y^\iota q)\p_q,\,(\mathrm J^\kappa\mathrm D_z^\iota q)\p_q,\,\kappa\in\mathbb N_0,\,\iota\in\mathbb N\big\rangle
\]
with the abelian ideal $\tilde{\mathfrak K}^{-\infty}:=\{f(y,z)\p_q\mid f\in{\rm KG}\}$ \cite[Theorem~4]{OpanasenkoPopovych2018}.
Here  $\mathrm D_y$ and $\mathrm D_z$ are the operators of total derivatives in~$y$ and~$z$, respectively,
and $\mathrm J:=y\mathrm D_y-z\mathrm D_z$.
Denote by $\Upsilon$, $\Upsilon^{\rm triv}$ and $\Upsilon^{\rm q}$
the algebra (of evolutionary representatives) of variational symmetries of the Lagrangian~$\mathrm L$,
its ideal of trivial variational symmetries and
the quotient algebra of variational symmetries of this Lagrangian, i.e.,
$\Upsilon\subset\mathfrak K$,
$\Upsilon^{\rm triv}:=\Upsilon\cap\mathfrak K^{\rm triv}$ and
$\Upsilon^{\rm q}:=\Upsilon/\Upsilon^{\rm triv}$.
The quotient algebra~$\Upsilon^{\rm q}$ is naturally isomorphic to
the algebra $\tilde\Upsilon^{\rm q}=\tilde\Lambda^{\rm q}_-\lsemioplus\tilde\Sigma^{-\infty}$,
where
\begin{gather*}
\tilde\Lambda^{\rm q}_-:=\big\langle\,(\mathfrak Q_{\kappa'0}q)\p_q,\,\kappa'\in2\mathbb N_0+1,\
(\mathfrak Q_{\kappa\iota}q)\p_q,\,(\bar{\mathfrak Q}_{\kappa\iota}q)\p_q,\, \kappa\in\mathbb N_0,\,\iota\in\mathbb N,\,\kappa+\iota\in2\mathbb N_0+1\,\big\rangle
\\[1ex]
\hspace*{-\mathindent}\mbox{with}
\\[-1ex]
\mathfrak Q_{\kappa\iota}=\left(\mathrm J+\frac \iota2\right)^\kappa\mathrm D_y^\iota, \quad \kappa,\iota\in\mathbb N_0,\qquad
\bar{\mathfrak Q}_{\kappa\iota}=\left(\mathrm J-\frac \iota2\right)^\kappa\mathrm D_z^\iota, \quad \kappa\in\mathbb N_0,\quad \iota\in\mathbb N,
\end{gather*}
is the subspace of~$\tilde\Lambda^{\rm q}$ that is associated with the space of formally skew-adjoint differential operators
generated by $\mathrm D_y$, $\mathrm D_z$ and $\mathrm J$.
Note that in the context of Noether's theorem, we need to consider the algebra~$\tilde{\mathfrak K}^{\rm q}$
instead of the algebra~$\hat{\mathfrak K}^{\rm q}$ of reduced generalized symmetries of~\eqref{eq:IDFMSupersystemA},
which is mentioned in Section~\ref{sec:IDFMGenSyms},
since cosets of~$\Upsilon^{\rm triv}$ in~$\Upsilon$
do not necessarily intersect the algebra~$\hat{\mathfrak K}^{\rm q}$,
see Remark~9 in~\cite{OpanasenkoPopovych2018}.
The space of conservation laws of~\eqref{eq:IDFMSupersystemA}
is naturally isomorphic to the space spanned by the conserved currents
\[
\bar{\rm C}^0_f=(-f_zq,fq_y), \quad
{\rm C}_{\mathfrak Q}=(-q\mathrm D_z\mathfrak Qq,\,q_y\mathfrak Qq),
\]
where the parameter function $f=f(y,z)$ runs through the solution set of~\eqref{eq:IDFMSupersystemA},
and the operator~$\mathfrak Q$ runs through the basis of~$\tilde\Lambda^{\rm q}_-$
\cite[Proposition~10]{OpanasenkoPopovych2018}.
The conserved current~\smash{$\bar{\rm C}^0_f$} is equivalent to
the conserved current ${\rm C}^0_f=(fq_z,-f_y q)$.

We map conserved currents of the form~${\rm C}_{\mathfrak Q}$,
where $\mathfrak Qq\p_q$ runs through the basis of~$\tilde\Lambda^{\rm q}_-$,
to conserved currents of the system~\eqref{eq:IDFMDiagonalizedSystem},
which leads to the third family of the theorem.
Possible modifications of the form of these conserved currents up to recombining them and adding null divergences
are discussed in Remark~\ref{rem:IDFM:3rdFamilyOfCLs} below.

At the same time, it is convenient to modify conserved currents of the form~\smash{$\bar{\rm C}^0_f$}
before their mapping in order to directly obtain hydrodynamic conservation laws.%
\footnote{%
Recall that a conservation law is called \textit{hydrodynamic}
if its density~$\rho$ is a function of dependent variables only.
}
We reparameterize these conserved currents,
representing the parameter function~$f$ in the form $f=\bar f_y+\bar f_z+2\bar f$,
where the function $\bar f=\bar f(y,z)$ also runs through the solution set of the (1+1)-dimensional Klein--Gordon equation~\eqref{eq:IDFMSupersystemA}.
Then $f_z=\bar f_{zz}+2\bar f_z+\bar f$.
Adding the null divergence $(\mathrm D_zR,-\mathrm D_yR)$ with $R:=\bar f q_z-\bar f_zq-2\bar fq$ to \smash{$-\bar{\rm C}^0_f$},
we obtain the equivalent conserved current $(\bar f K_1,-\bar f_y K_2)$,
which is mapped to the conserved current from the second family %of Theorem~\ref{thm:IDFM:CLs}
with $\Phi=\bar f(\ri^1/2,-\ri^2/2)$.

Note that the first and second theorem's families are in fact subspaces
in the space of conserved currents of the system~\eqref{eq:IDFMDiagonalizedSystem}.
Analyzing the equivalence of modified conserved currents, we see that
conserved currents from the first theorem's family are equivalent
if and only if the difference of corresponding~$\Omega$'s
belongs to the image of the operator $\hat{\mathscr A}=\sum_{\kappa=0}^{\infty}\omega^{\kappa+1}\p_{\omega^\kappa}$.
The intersection of the first and the second families is one-dimensional
and spanned by the conserved current $\big(\,{\rm e}^{\ri^1-\ri^2},\,(\ri^1+\ri^2){\rm e}^{\ri^1-\ri^2}\,\big)$.
The sum of these two families does not intersect the span of the third family.
The equivalence of conserved currents within the span of all the three families
is generated by the equivalence of conserved currents within the first family.
\end{proof}

\begin{remark}\label{rem:IDFM:ImageOfA}
The kernel $\ker\mathsf E$ of the operator
$\mathsf E=\sum_{\kappa=1}^\infty\sum_{\kappa'=0}^{\kappa-1}\omega^{\kappa-\kappa'}(-\hat{\mathscr A})^{\kappa'}\p_{\omega^\kappa}-1$
is contained in the kernel $\ker\mathsf E'$ of the operator
$\mathsf E'=\sum_{\kappa=0}^{\infty}(-\hat{\mathscr A})^\kappa\p_{\omega^\kappa}$, $\ker\mathsf E\subset\ker\mathsf E'$,
since the operator identity $\hat{\mathscr A}\mathsf E=-\omega^1\mathsf E'$ holds.
In view of \cite[Theorem 4.26]{Olver1993},
Theorem~\ref{thm:IDFM:CLChars} below implies that (locally) the image of the operator $\hat{\mathscr A}$
coincides with $\ker\mathsf E\cap\ker\mathsf E'=\ker\mathsf E$.
The kernel $\ker\mathsf E'$ of~$\mathsf E'$ is spanned by the constant function~1 and the image of~$\hat{\mathscr A}$.
Hence $\mathop{\rm im}\hat{\mathscr A}=\ker\mathsf E\varsubsetneq\ker\mathsf E'$.
\end{remark}

\begin{remark}
The conserved currents from Theorem~\ref{thm:IDFM:CLs}
that are associated with
\begin{gather*}
\Omega=\frac{\ri^3}{\ri^3+1},\quad
\Omega=\frac1{\ri^3+1},\quad
\Omega=1,\\
\Phi={\rm e}^{(\ri^1-\ri^2)/2}(\ri^1+\ri^2-1),\quad
\Phi=\frac18{\rm e}^{(\ri^1-\ri^2)/2}\left((\ri^1+\ri^2)^2-4\ri^2\right)
\end{gather*}
correspond to the conservation of masses of the both individual phases and of mixture mass
as well as the conservation of mixture momentum and of energy in the drift flux model, respectively,
cf.~\cite[Chapter~13]{IshiiHibiki2006}.
The related equations in conserved form are
\begin{gather*}
\rho^1_t+(\rho^1u)_x=0,\quad
\rho^2_t+(\rho^2u)_x=0,\quad
(\rho^1+\rho^2)_t+\left((\rho^1+\rho^2)u\right)_x=0,\\
\left((\rho^1+\rho^2)u\right)_t+\left((\rho^1+\rho^2)(u^2+1)\right)_x=0,\\
\left((\rho^1+\rho^2)\left(\frac{u^2}2+\ln(\rho^1+\rho^2)\right)\right)_t+
\left((\rho^1+\rho^2)\left(\frac{u^2}2+\ln(\rho^1+\rho^2)+1\right)u\right)_x=0.
\end{gather*}
In particular, the magnitude $\ln(\rho^1+\rho^2)$ can be interpreted as (proportional to) the internal mixture energy.
The first, second and fourth equations constitute the conserved form of the system~$\mathcal S$
in the original variables $(\rho^1,\rho^2,u)$.
\end{remark}

\begin{theorem}\label{thm:IDFM:CLChars}
In the notation of Theorem~\ref{thm:IDFM:CLs},
the associated reduced conservation-law characteristics of the system~\eqref{eq:IDFMDiagonalizedSystem}
are respectively
\begin{enumerate}
\item
$\displaystyle
{\rm e}^{\ri^1-\ri^2}\left(\,
 \Omega-\sum_{\kappa=1}^\infty\sum_{\kappa'=0}^{\kappa-1}\omega^{\kappa-\kappa'}(-\hat{\mathscr A})^{\kappa'}\Omega_{\omega^\kappa},\
-\Omega+\sum_{\kappa=1}^\infty\sum_{\kappa'=0}^{\kappa-1}\omega^{\kappa-\kappa'}(-\hat{\mathscr A})^{\kappa'}\Omega_{\omega^\kappa},\
\sum_{\kappa=0}^\infty(-\hat{\mathscr A})^\kappa\Omega_{\omega^\kappa}\right)$.
\item
${\rm e}^{(\ri^1-\ri^2)/2}\big(\,2\Phi_{\ri^1\ri^1}+2\Phi_{\ri^1}+\frac12\Phi,\ \Phi_{\ri^2}-\Phi_{\ri^1}-\Phi,\ 0\,\big)$.
\item
${\rm e}^{(\ri^1-\ri^2)/2}\big(\,-\tilde{\mathscr D}_y\tilde{\mathfrak Q}\tilde q,\, \tilde{\mathfrak Q}\tilde q,\,0\,\big)$.%
\footnote{%
Here we omitted the multiplier $-2$, which is needed for the direct correspondence
between these conservation-law characteristics and conserved currents from the third family of Theorem~\ref{thm:IDFM:CLs}.
}
\end{enumerate}
The space spanned by these characteristics is naturally isomorphic
to the quotient space of con\-servation-law characteristics of the system~\eqref{eq:IDFMDiagonalizedSystem}.
\end{theorem}

\begin{proof}
Since the system~\eqref{eq:IDFMDiagonalizedSystem} is a system of evolution equations,
its conservation-law characteristics can be found from reduced densities of the associated conservation laws
by acting the Euler operator,
\[
\mathsf E=\left(\sum_{\kappa=0}^\infty(-\mathrm D_x)^\kappa\p_{\ri^i_\kappa},\ i=1,2,3\right),
\]
see e.g. \cite[Proposition 7.41]{Tsujishita1982}.
This perfectly works for characteristics related to the second family of conserved currents
presented in Theorem~\ref{thm:IDFM:CLs}
but does not give reasonable representations for characteristics related to the first and third families,
for which we use different methods.

Characteristics related to the third family can be obtained from conservation-law characteristics
of the (1+1)-dimensional Klein--Gordon equation~\eqref{eq:IDFMSupersystemA}.
A characteristic of the conservation law of~\eqref{eq:IDFMSupersystemA} containing the conserved current~${\rm C}_{\mathfrak Q}$
 is $\lambda=(\mathfrak Q-\mathfrak Q^\dag)q=2\mathfrak Qq$ for $(\mathfrak Qq)\p_q\in\tilde\Lambda^{\rm q}_-$.
It is trivially prolonged to the conservation-law characteristic $(\lambda,0)$
of the system~\eqref{eq:IDFMSupersystemA}, \eqref{eq:IDFMSupersystemC}.
Denote by~$R^1$, $R^2$, $L^1$ and~$L^2$ the differential functions
associated with the equations~\eqref{eq:IDFMDEq1}, \eqref{eq:IDFMDEq2}, \eqref{eq:IDFMSupersystemA} and~\eqref{eq:IDFMSupersystemC}, respectively,
\begin{gather*}
R^1:=\ri^1_t+(\ri^1+\ri^2+1)\ri^1_x,\quad
R^2:=\ri^2_t+(\ri^1+\ri^2-1)\ri^2_x,\\
L^1:=q_{yz}-q,\quad
L^2:=p-\frac12{\rm e}^{-y-z}(q_z-q).
\end{gather*}
These differential functions are related via the transformation~$\mathcal T$,
$\hat{\mathcal T}^*(R^1,R^2)^{\mathsf T}=\mathfrak M(L^1,L^2)^{\mathsf T}$ with
\[
\mathfrak M=\begin{pmatrix}0&-\dfrac4\Delta\\ \dfrac2\Delta {\rm e}^{-y-z}&\dfrac4\Delta(\mathrm D_y+1)\end{pmatrix},
\quad\mbox{and}\quad
\mathfrak M^\dag=\begin{pmatrix}0&\dfrac2\Delta {\rm e}^{-y-z}\\ -\dfrac4\Delta&-(\mathrm D_y-1)\circ\dfrac4\Delta\end{pmatrix},
\]
\noprint{
\[
\hat{\mathcal T}^*R^1=-\frac4\Delta L^2,\quad
\hat{\mathcal T}^*R^2=\frac2\Delta {\rm e}^{-y-z} L^1+\frac4\Delta(\mathrm D_y+1)L^2,
\]
}
where $\Delta=(\mathrm D_y\hat{\mathcal T}^t)(\mathrm D_z\hat{\mathcal T}^x)-(\mathrm D_z\hat{\mathcal T}^t)(\mathrm D_y\hat{\mathcal T}^x)$,
$\mathcal T^*\Delta=-4(\ri^1_t\ri^2_x-\ri^1_x\ri^2_t)$.
The conservation-law characteristic $(\lambda^1,\lambda^2)$ of the system~\eqref{eq:IDFMDEq1}, \eqref{eq:IDFMDEq2}
that is associated with the conservation-law characteristic $(\lambda,0)$
of the system~\eqref{eq:IDFMSupersystemA}, \eqref{eq:IDFMSupersystemC}
is defined by $\mathfrak M^\dag(\Delta\hat{\mathcal T}^*\lambda^1,\Delta\hat{\mathcal T}^*\lambda^2)^{\mathsf T}=(\lambda,0)^{\mathsf T}$.
Therefore, the conservation-law characteristic $\lambda$ of~\eqref{eq:IDFMSupersystemA}
is mapped to the conservation-law characteristic $\frac12{\rm e}^{y+z}(-\mathscr D_y\lambda,\lambda,0)$ of the system~$\mathcal S$,
where all values should be expressed in terms of the variables $(t,x,\ri)$.
This gives a conservation-law characteristic from the third family of the theorem.

Characteristics related to the first family are found
following the procedure of defining them via the formal integration by parts, cf.\ \cite[p.~266]{Olver1993}.
We denote by~$\mathrm A$ and~$\mathrm B$ the counterparts of the operators~$\mathscr A$ and~$\mathscr B$, respectively,
in the complete total derivative operators with respect to~$t$ and~$x$,
$\mathrm A:={\rm e}^{\ri^2-\ri^1}\mathrm D_x$, $\mathrm B:=\mathrm D_t+(\ri^1+\ri^2)\mathrm D_x$.
Then
\begin{gather}\label{eq:IDFM:DivOf1stCC}
\mathrm D_t\big({\rm e}^{\ri^1-\ri^2}\Omega\big)+\mathrm D_x\big((\ri^1+\ri^2){\rm e}^{\ri^1-\ri^2}\Omega\big)
={\rm e}^{\ri^1-\ri^2}\Omega\mathcal E^1-{\rm e}^{\ri^1-\ri^2}\Omega\mathcal E^2+\sum_{\kappa=0}^\infty {\rm e}^{\ri^1-\ri^2}\Omega_{\omega^\kappa}\mathrm B\omega^\kappa.
\end{gather}
Here $\mathcal E^k$ denotes the left-hand side of the $k$th equation of the system~\eqref{eq:IDFMDiagonalizedSystem},
$\mathcal E^k=\ri^k_t+V^k\ri^k_x$, $k=1,2,3$. Note that $\mathcal E^3=\mathrm B\ri^3$.
Since $\Omega$ depends on a finite number of~$\omega$'s, there is no issue with convergence.

We derive using the mathematical induction with respect to~$\iota$ that
\begin{gather}\label{eq:IDFM:Bomegaiota}
\mathrm B\omega^\kappa=
\mathrm A^\kappa\mathcal E^3+\sum_{\kappa'=0}^{\kappa-1}\mathrm A^{\kappa'}\big(\omega^{\kappa-\kappa'}(\mathcal E^2-\mathcal E^1)\big).
\end{gather}
Indeed, for the base case $\kappa=0$, we have $\mathrm B\omega^0=\mathrm B\ri^3=\mathcal E^3$.
The induction step follows from the equality
$\mathrm B\omega^{\kappa+1}=\mathrm B\mathrm A\omega^\kappa=\mathrm A\mathrm B\omega^\kappa+\omega^{\kappa+1}(\mathcal E^2-\mathcal E^1)$.

Using again the mathematical induction with respect to~$\kappa$,
we prove the counterpart of the Lagrange identity in terms of the operator~$\mathrm A$,
\[
{\rm e}^{\ri^1-\ri^2}F\mathrm A^\kappa G={\rm e}^{\ri^1-\ri^2}\big((-\mathrm A)^\kappa F\big)G
+\mathrm D_x\sum_{\kappa'=0}^{\kappa-1}\big((-\mathrm A)^{\kappa'}F\big)\mathrm A^{\kappa-\kappa'-1}G,\quad \kappa\in\mathbb N_0,
\]
for any differential functions~$F$ and~$G$ of~$\ri$.
We apply this identity to each summand of the expression ${\rm e}^{\ri^1-\ri^2}\Omega_{\omega^\kappa}\mathrm B\omega^\kappa$
expanded in view of~\eqref{eq:IDFM:Bomegaiota}, which gives
\[
{\rm e}^{\ri^1-\ri^2}\Omega_{\omega^\kappa}\mathrm B\omega^\kappa
={\rm e}^{\ri^1-\ri^2}\big((-\mathrm A)^\kappa\Omega_{\omega^\kappa}\big)\mathcal E^3
+{\rm e}^{\ri^1-\ri^2}\sum_{\kappa'=0}^{\kappa-1}\big((-\mathrm A)^{\kappa'}\Omega_{\omega^\kappa}\big)\omega^{\kappa-\kappa'}(\mathcal E^2-\mathcal E^1)+\mathrm D_xH,
\]
where $H$ is a differential function of~$\ri$ that vanishes on the manifold~$\mathcal S^{(\infty)}$ and whose precise form is not essential.
When acting on functions of~$\omega$'s, the operator~$\mathrm A$ can be replaced
by the operator $\hat{\mathscr A}=\sum_{\kappa=0}^{\infty}\omega^{\kappa+1}\p_{\omega^\kappa}$.
Substituting the derived expression for ${\rm e}^{\ri^1-\ri^2}\Omega_{\omega^\kappa}\mathrm B\omega^\kappa$ into~\eqref{eq:IDFM:DivOf1stCC}
and collecting coefficients of~$\mathcal E^1$, $\mathcal E^2$ and~$\mathcal E^3$,
we obtain a characteristic from the first family of the theorem.
\looseness=-1
\end{proof}

\begin{remark}\label{rem:IDFM:IntersectionOfCLSubspaces}\looseness=-1
Since the common element ${\rm e}^{\ri^1-\ri^2}(1,-1,0)$ of cosymmetry families,
which is mentioned in Remark~\ref{rem:IDFM:IntersectionOfCosymSubspaces},
is a conservation-law characteristic of the system~$\mathcal S$, it was expected that
the families of conserved currents and of conservation-law characteristics
from Theorems~\ref{thm:IDFM:CLs} and~\ref{thm:IDFM:CLChars} have the same properties
as the properties of cosymmetry families indicated in Remark~\ref{rem:IDFM:IntersectionOfCosymSubspaces}.
%In particular,
Thus,
the above conservation-law characteristic,
which spans the intersection of the first and the second families from Theorem~\ref{thm:IDFM:CLChars},
corresponds to the conserved current ${\rm e}^{\ri^1-\ri^2}(1,\ri^1+\ri^2)$
spanning the intersection of the respective families from Theorem~\ref{thm:IDFM:CLs},
cf.\ the end of the proof of this theorem.
\end{remark}

\begin{remark}\label{rem:IDFM:CorrespondenceBetweenCosymsAndCLChars}
The second family of cosymmetries from Theorem~\ref{thm:IDFM:Cosyms} coincides with
the second family of conservation-law characteristics from Theorem~\ref{thm:IDFM:CLChars} up to reparameterization.
In other words, each cosymmetry in this family is a conservation-law characteristic.
This is not the case for the first%
\footnote{%
In the notation of Remark~\ref{rem:IDFM:ImageOfA},
upon formally interpreting $\omega^0$ as a single dependent variable of a single independent variable, say $\varsigma$,
and $\omega^1$, $\omega^2$, \dots\ as the successive derivatives of $\omega^0$,
the operators $\p_\varsigma+\hat{\mathscr A}$ and $\mathsf E'$
become the total derivative operator with respect to~$\varsigma$
and the Euler operator with respect to~$\omega^0$, respectively.
Suppose that a smooth function~$\Omega$ of a finite number of~$\omega$'s
belongs to $\mathop{\rm im}\mathsf E$.
Then $(\hat{\mathscr A}\Omega)/\omega^1\in\mathop{\rm im}\mathsf E'$ and
thus the Fr\'echet derivative of $(\hat{\mathscr A}\Omega)/\omega^1$ with respect to $\omega^0$
is a formally self-adjoint operator.
This is not the case for any $\Omega$ of even positive order.
Therefore, any cosymmetry from the first family of Theorem~\ref{thm:IDFM:Cosyms}
with $\Omega$ of even positive order is not a conservation-law characteristic of the system~\eqref{eq:IDFMDiagonalizedSystem}.}
and third families of cosymmetries from Theorem~\ref{thm:IDFM:Cosyms},
which properly contain the first and third families of conservation-law characteristics from Theorem~\ref{thm:IDFM:CLChars}, respectively.
\end{remark}

\begin{theorem}\label{thm:IDFMGeneratingSetOfCLs}
Under the action of generalized symmetries of the system~\eqref{eq:IDFMDiagonalizedSystem}
on its space of conservation laws,
a generating set of conservation laws of this system
is constituted by the two zeroth-order conservation laws
respectively containing the conserved currents
\begin{subequations}\label{eq:IDFM:GeneratingCCs}
\begin{gather}\label{eq:IDFM:GeneratingCC1}
{\rm e}^{\ri^1-\ri^2}\big(\,\ri^3,\ (\ri^1+\ri^2)\ri^3\,\big),
\\\label{eq:IDFM:GeneratingCC2}
{\rm e}^{\ri^1-\ri^2}\big(x-V^3t,V^3(x-V^3t)-t\big)\quad\mbox{with}\quad V^3:=\ri^1+\ri^2.
\end{gather}
\end{subequations}
\end{theorem}

\begin{proof}
The action of the generalized symmetry $\Omega\p_{\ri^3}$ on the conserved current~\eqref{eq:IDFM:GeneratingCC1}
gives the conserved current $\big(\,{\rm e}^{\ri^1-\ri^2}\Omega,\ (\ri^1+\ri^2){\rm e}^{\ri^1-\ri^2}\Omega\,\big)$.
Varying the parameter function~$\Omega$ through the space of smooth functions
of a finite, but unspecified number of \smash{$\omega^\kappa=({\rm e}^{\ri^2-\ri^1}\mathscr D_x)^\kappa\ri^3$}, $\kappa\in\mathbb N_0$,
we obtain the first family of conserved currents from Theorem~\ref{thm:IDFM:CLs}.

Conserved currents from the other two families are constructed by mapping
conserved currents of the (1+1)-dimensional Klein--Gordon equation~\eqref{eq:IDFMSupersystemA}
in the way described in the proof of Theorem~\ref{thm:IDFM:CLs}.
In view of \cite[Corollary~11]{OpanasenkoPopovych2018},
a generating set of conservation laws of~\eqref{eq:IDFMSupersystemA}
is constituted,
under the action of generalized symmetries of~\eqref{eq:IDFMSupersystemA}
on conservation laws thereof,
by the single conservation law containing the conserved current $(q_z^2,-q^2)$.
The counterpart of this conserved current for the system~\eqref{eq:IDFMDiagonalizedSystem}
is the conserved current
\[
%\Big({\rm e}^{\ri^1-\ri^2}\big(\ri^2_x(x-V^2t)^2-\ri^1_x(x-V^1t)^2\big),\ {\rm e}^{\ri^1-\ri^2}\big(V^2\ri^2_x(x-V^2t)^2-V^1\ri^1_x(x-V^1t)^2\big)\Big),
{\rm e}^{\ri^1-\ri^2}\Big(\ri^2_x(x-V^2t)^2-\ri^1_x(x-V^1t)^2,\ V^2\ri^2_x(x-V^2t)^2-V^1\ri^1_x(x-V^1t)^2\Big),
\]
which is equivalent to the conserved current~\eqref{eq:IDFM:GeneratingCC2} multiplied by~2.
It follows from Lemma~\ref{lem:IDFM:GenSymsOfSupersystem2B}
that not all generalized symmetries of~\eqref{eq:IDFMSupersystemA}
can be naturally mapped to those of the system~\eqref{eq:IDFMDiagonalizedSystem}.
This is why we need to carefully analyze the result on generating conservation laws of~\eqref{eq:IDFMSupersystemA}
before adopting it for the system~\eqref{eq:IDFMDiagonalizedSystem}.

The conserved current \smash{${\rm C}^0_f=(fq_z,-f_y q)$} of the equation~\eqref{eq:IDFMSupersystemA}
can be obtained by acting the generalized symmetry $\frac12f_y\p_q\in\hat{\mathfrak K}^{-\infty}$ of this equation
on the chosen conserved current $(q_z^2,-q^2)$.
Here the parameter function $f=f(y,z)$ runs through the solution set of~\eqref{eq:IDFMSupersystemA}.
Each conserved current from the second family of Theorem~\ref{thm:IDFM:CLs}
is the image of a conserved current of the form \smash{${\rm C}^0_f$},
and each Lie symmetry vector field $f\p_q$ of~\eqref{eq:IDFMSupersystemA}
is mapped to an element of the ideal~$\mathcal I^1$ of the algebra~$\hat\Sigma^{\rm q}$.
Therefore, the second family of conserved currents from Theorem~\ref{thm:IDFM:CLs}
is generated by acting the elements of~$\mathcal I^1$ on the conserved current~\eqref{eq:IDFM:GeneratingCC2}.

The action of the generalized symmetry $\frac12(\mathrm D_y\mathfrak Qq)\p_q$,
where $(\mathfrak Qq)\p_q\in\tilde\Lambda^{\rm q}$,
on the conserved current $(q_z^2,-q^2)$ gives the conserved current
$(q_z\mathrm D_y\mathrm D_z\mathfrak Qq,-q\mathrm D_y\mathfrak Qq)$,
which is equivalent to the conserved currents $(q_z\mathfrak Qq,-q\mathrm D_y\mathfrak Qq)$
and, therefore, to ${\rm C}_{\mathfrak Q}=(-q\mathrm D_z\mathfrak Qq,\,q_y\mathfrak Qq)$.
The conservation law containing the obtained conserved currents has the characteristic $(\mathfrak Q-\mathfrak Q^\dag)q$.

We denote by~$\mathfrak V$ the subalgebra of~$\tilde\Lambda^{\rm q}$ constituted by
the elements of~$\tilde\Lambda^{\rm q}$
that have counterparts among generalized symmetries of the system~\eqref{eq:IDFMDiagonalizedSystem},
and $\mathfrak J:=\langle(\mathrm J^\kappa q)\p_q,\,\kappa\in\mathbb N\rangle$.
We also introduce the corresponding spaces~$\mathfrak V_-$ and~$\mathfrak J_-$ of linear generalized symmetries
associated with formally skew-adjoint counterparts $\frac12(\mathfrak Q-\mathfrak Q^\dag)$
of operators~$\mathfrak Q$ from~$\mathfrak V$ and~$\mathfrak J$, respectively.
Note that $\mathfrak V_-\supsetneq\mathfrak V\cap\tilde\Lambda^{\rm q}_-$ and
$\mathfrak J_-=\mathfrak J\cap\tilde\Lambda^{\rm q}_-$.
In view of Lemma~\ref{lem:IDFM:GenSymsOfSupersystem2B}, $(\mathfrak Qq)\p_q\in\mathfrak V$
if and only if the operator~$\mathfrak Q$ is represented in the form
$\mathfrak Q=(\mathrm D_y+1)\mathfrak Q_1+(\mathrm D_z+1)\mathfrak Q_2+c$
for some \smash{$\mathfrak Q_1\in\langle\,\mathrm D_y^\iota\mathrm J^\kappa,\,\kappa,\iota\in\mathbb N_0\rangle$},
some $\mathfrak Q_2\in\langle\,\mathrm D_z^\iota\mathrm J^\kappa,\,\kappa,\iota\in\mathbb N_0\rangle$
and some $c\in\mathbb R$.
Hence $\tilde\Lambda^{\rm q}$ is the direct sum of~$\mathfrak V$ and $\mathfrak J$ as vector spaces,
$\tilde\Lambda^{\rm q}=\mathfrak V\dotplus\mathfrak J$,
and thus $\tilde\Lambda^{\rm q}_-=\mathfrak V_-+\mathfrak J_-$,
where the sum is not direct by now.
We are going to show that $\mathfrak V_-\supset\mathfrak J_-$, which implies that $\tilde\Lambda^{\rm q}_-=\mathfrak V_-$.
Indeed, for any $\mathfrak Q:=(\mathrm D_y+1)(\mathrm J-1/2)^\kappa$ with $\kappa\in2\mathbb N_0+1$ we have
\[
\mathfrak Q-\mathfrak Q^\dag
=(\mathrm D_y+1)(\mathrm J-1/2)^\kappa-(\mathrm J+1/2)^\kappa(\mathrm D_y-1)
=(\mathrm J-1/2)^\kappa+(\mathrm J+1/2)^\kappa,
\]
i.e., $\big((\mathrm J-1/2)^\kappa q+(\mathrm J+1/2)^\kappa q\big)\p_q\in\mathfrak V_-$
since $(\mathfrak Q q)\p_q\in\mathfrak V$.
Therefore,
\[
\mathfrak J_-=\big\langle(\mathrm J^\kappa q)\p_q,\,\kappa\in2\mathbb N_0+1\big\rangle
=\big\langle\big((\mathrm J-1/2)^\kappa q+(\mathrm J+1/2)^\kappa q\big)\p_q,\,\kappa\in2\mathbb N_0+1\big\rangle\subset\mathfrak V_-.
\]

As a result, for any $(\mathfrak Qq)\p_q\in\tilde\Lambda^{\rm q}$
the conserved current ${\rm C}_{\mathfrak Q}$ is equivalent to a conserved current of~\eqref{eq:IDFMSupersystemA}
that is obtained by the action of a generalized symmetry from $\mathfrak V$ on the chosen conserved current $(q_z^2,-q^2)$.
For the system~\eqref{eq:IDFMDiagonalizedSystem}, this means that
the third family of conserved currents from Theorem~\ref{thm:IDFM:CLs}
is generated by acting the generalized symmetries of the form~$\check{\mathcal R}(\Gamma)$
on the conserved current~\eqref{eq:IDFM:GeneratingCC2}.
\end{proof}

\begin{remark}\label{rem:IDFM:2rdFamilyOfCLs}
The conserved currents from the second family of Theorem~\ref{thm:IDFM:CLs}
can be represented in a more symmetrical form.
Reparameterizing them in terms of the potential~$\bar\Phi$ defined via~$\Phi$
by the system $\bar\Phi_{\ri^1}+\frac12\bar\Phi=2\Phi_{\ri^1}$, $-\bar\Phi_{\ri^2}+\frac12\bar\Phi=\Phi$,
cf.\ Section~\ref{sec:IDFMSolutionsViaEssentialSubsystem},
we obtain another representation for these conserved currents,
\[%\label{eq:IDFMAnotherRepresentationOf2ndFamilyOfCCs}
%\big(\,{\rm e}^{(\ri^1-\ri^2)/2}(\bar\Phi_{\ri^1}-\bar\Phi_{\ri^2}+\bar\Phi),\
%{\rm e}^{(\ri^1-\ri^2)/2}((\ri^1+\ri^2+1)\bar\Phi_{\ri^1}-(\ri^1+\ri^2-1)\bar\Phi_{\ri^2}+(\ri^1+\ri^2)\bar\Phi)\,\big),
{\rm e}^{(\ri^1-\ri^2)/2}
\big(\,\bar\Phi_{\ri^1}-\bar\Phi_{\ri^2}+\bar\Phi,\
(\ri^1+\ri^2+1)\bar\Phi_{\ri^1}-(\ri^1+\ri^2-1)\bar\Phi_{\ri^2}+(\ri^1+\ri^2)\bar\Phi\,\big),
\]
where the parameter function~$\bar\Phi=\bar\Phi(\ri^1,\ri^2)$
runs through the solution space of the Klein--Gordon equation $\bar\Phi_{\ri^1\ri^2}=-\bar\Phi/4$ as well.
The successive point transformation $\tilde\Phi={\rm e}^{(\ri^1-\ri^2)/2}\bar\Phi$
reduces the above representation to
$\big(\,\tilde\Phi_{\ri^1}-\tilde\Phi_{\ri^2},\ (\ri^1+\ri^2+1)\tilde\Phi_{\ri^1}-(\ri^1+\ri^2-1)\tilde\Phi_{\ri^2}\,\big)$,
where the parameter function~$\tilde\Phi=\tilde\Phi(\ri^1,\ri^2)$
runs through the solution space of the equation \smash{$2\tilde\Phi_{\ri^1\ri^2}=\tilde\Phi_{\ri^2}-\tilde\Phi_{\ri^1}$}.
It is the last representation that was employed in \cite[Theorem~22]{OpanasenkoBihloPopovychSergyeyev2020}.
In terms of~$\tilde\Phi$, the associated characteristics take the form
$(\tilde\Phi_{\ri^1\ri^1}-\tilde\Phi_{\ri^1\ri^2},\tilde\Phi_{\ri^1\ri^2}-\tilde\Phi_{\ri^2\ri^2},0)$.
\end{remark}

\begin{remark}\label{rem:IDFM:3rdFamilyOfCLs}
The advantage of using conserved currents of the form~${\rm C}_{\mathfrak Q}$
for mapping to conserved currents of the system~$\mathcal S$
is that we obtain a uniform representation for elements of the third family of Theorem~\ref{thm:IDFM:CLs}.
At the same time, it is not obvious how to find equivalent conserved currents of minimal order for elements of this family
or how to single out conserved currents in this family that are equivalent to ones not depending on $(t,x)$ explicitly.
The former problem can be solved by replacing conserved currents of the form~${\rm C}_{\mathfrak Q}$
in the mapping by equivalent conserved currents
${\rm C}^1_{\kappa\iota}$, $\kappa\in\mathbb N_0$, $\iota\in\mathbb N$,
$\bar{\rm C}^1_{\kappa\iota}$, ${\rm C}^2_{\kappa\iota}$, $\bar{\rm C}^2_{\kappa\iota}$, $\kappa,\iota\in\mathbb N_0,$
presented in~\cite[Section~4]{OpanasenkoPopovych2018}
although an additional ``integration by parts'' may still be needed for lowest values of~$(\kappa,\iota)$ after the mapping,
cf. the proof of Theorem~\ref{thm:IDFMGeneratingSetOfCLs}.
For solving the latter problem, we use an analog of the trick used in the proof of Theorem~\ref{thm:IDFM:CLs}
for deriving the second family of conserved currents,
which leads to Theorem~\ref{thm:IDFM:TranslationInvCLs} below.
\end{remark}

\begin{corollary}\label{cor:IDFM0th-OrderCLs}
(i) The space of hydrodynamic conservation laws of the system~\eqref{eq:IDFMDiagonalizedSystem}
is infinite-dimensional and is naturally isomorphic to the space spanned by the conserved currents
from the second family of Theorem~\ref{thm:IDFM:CLs} and
from the first family with~$\Omega$ running through the space of smooth functions of~$\omega^0:=\ri^3$.

(ii) The space of zeroth-order conservation laws of the system~\eqref{eq:IDFMDiagonalizedSystem}
is naturally isomorphic to the space spanned by its hydrodynamic conserved currents
and the conserved current~\eqref{eq:IDFM:GeneratingCC2}.
\end{corollary}

\begin{proof}
This assertion was proved in~\cite[Theorem~22]{OpanasenkoBihloPopovychSergyeyev2020} by the direct computation.
 At the same time, it is a simple corollary of Theorems~\ref{thm:IDFM:CLs} and~\ref{thm:IDFM:CLChars}.
Indeed, when linearly combining conserved currents from different families of Theorem~\ref{thm:IDFM:CLs},
the maximum of their orders is preserved.
The selection of zeroth-order conserved currents from the first and the second families is obvious.
Theorem~\ref{thm:IDFM:CLChars} implies
that the space of zeroth-order characteristics related to the third family
is one-dimensional and spanned by the characteristic
${\rm e}^{(\ri^1-\ri^2)/2}\big(\,\tilde q,\, -\tilde{\mathscr D}_z\tilde q,\,0\,\big)$
of the conservation law containing the conserved current~\eqref{eq:IDFM:GeneratingCC2}.
\end{proof}

\begin{corollary}\label{cor:IDFM1st-OrderCLs}
The space of zeroth- and first-order conservation laws of the system~\eqref{eq:IDFMDiagonalizedSystem}
is naturally isomorphic to the space spanned by the conserved currents
from the second family of Theorem~\ref{thm:IDFM:CLs}
and from the first family, where the parameter function~$\Omega$ runs through the space of smooth functions
of~$(\omega^0,\omega^1):=(\ri^3,{\rm e}^{\ri^2-\ri^1}\ri^3_x)$
and such two functions should be assumed equivalent if their difference is of the form $f(\omega^0)\omega^1$,
as well as the conserved currents from the third family,
where the operator~$\tilde{\mathfrak Q}$ runs through the set
$
\big\{\tilde{\mathscr D}_z,\,\tilde{\mathscr D}_y,\,\tilde{\mathscr J},\,
\tilde{\mathscr D}_z^3,\,(\tilde{\mathscr J}-1)\tilde{\mathscr D}_z^2\big\}
$.
\end{corollary}

\begin{proof}
In the same spirit as in the proof of Corollary~\ref{cor:IDFM0th-OrderCLs},
we select the zeroth- and first-order conserved currents equivalent to those listed in Theorem~\ref{thm:IDFM:CLs}
using Theorem~\ref{thm:IDFM:CLChars} for estimating the orders of the associated conservation laws.
Thus, the selection of the conserved currents from the second family is again obvious
since all of then are of order zero.
The order of a conservation law related to the first family coincides with
the minimal order of the associated~$\Omega$'s.
In general, for zeroth- and first-order conservation laws of the system~\eqref{eq:IDFMDiagonalizedSystem},
the order of corresponding reduced characteristics is not greater than two.
This is why a conservation law related to the span of the third family is of order not greater than one
if and only if it contains a conserved current corresponding to
$\tilde{\mathfrak Q}\in\big\langle\tilde{\mathscr D}_z,\,\tilde{\mathscr D}_y,\,\tilde{\mathscr J},\,
\tilde{\mathscr D}_z^3,\,(\tilde{\mathscr J}-1)\tilde{\mathscr D}_z^2\big\rangle$.
\end{proof}

\begin{theorem}\label{thm:IDFM:TranslationInvCLs}
The space of $(t,x)$-translation-invariant conservation laws of the system~\eqref{eq:IDFMDiagonalizedSystem}
is naturally isomorphic to the space spanned by the conserved currents
from the first and second families of Theorem~\ref{thm:IDFM:CLs}
as well as the conserved currents from the span of the third family
that have the form $\tilde{\rm C}_{\tilde{\mathfrak Q}}$ of elements of this family,
\begin{gather}\label{eq:IDFM:FormOfCCsFrom3rdFamily}
\big(\ri^2_x\tilde\rho+\ri^1_x\tilde\sigma,\ (\ri^1+\ri^2-1)\ri^2_x\tilde\rho+(\ri^1+\ri^2+1)\ri^1_x\tilde\sigma\big)
\ \mbox{with}\
\tilde\rho=-\tilde q\tilde{\mathscr D}_z\tilde{\mathfrak Q}\tilde q,\
\tilde\sigma=(\tilde{\mathscr D}_y\tilde q)\tilde{\mathfrak Q}\tilde q,
\end{gather}
where the operator~$\tilde{\mathfrak Q}$ runs through the set $\mathfrak T$
constituted by the operators
\begin{gather*}
\tilde{\mathfrak Z}_{\kappa\iota}:=(\tilde{\mathscr D}_z+1)^2(\tilde{\mathscr J}-\iota/2)^\kappa\tilde{\mathscr D}_z^\iota(\tilde{\mathscr D}_z-1)^2,\quad
\tilde{\mathfrak Y}_{\kappa,\iota+4}:=(\tilde{\mathscr D}_y+1)^2(\tilde{\mathscr J}+\iota/2)^\kappa\tilde{\mathscr D}_y^\iota(\tilde{\mathscr D}_y-1)^2,\\
\qquad\kappa,\iota\in\mathbb N_0\ \mbox{with}\ \kappa+\iota\in2\mathbb N_0+1,
\\
\tilde{\mathfrak Y}_{\kappa1}:=(\tilde{\mathscr J}+1/2)^\kappa(\tilde{\mathscr D}_y+\tilde{\mathscr D}_z-2)
+(\tilde{\mathscr D}_z+2)(\tilde{\mathscr J}-1/2)^\kappa(\tilde{\mathscr D}_z-1)^2,\quad \kappa\in2\mathbb N_0,
\\
\tilde{\mathfrak Y}_{\kappa2}:=2\tilde{\mathscr J}^\kappa(\tilde{\mathscr D}_y+\tilde{\mathscr D}_z-2)
+(\tilde{\mathscr J}+1)^\kappa(\tilde{\mathscr D}_y-1)^2+(\tilde{\mathscr J}-1)^\kappa(\tilde{\mathscr D}_z-1)^2,\quad \kappa\in2\mathbb N_0+1,
\\
\tilde{\mathfrak Y}_{\kappa3}:=(\tilde{\mathscr J}-1/2)^\kappa(\tilde{\mathscr D}_y+\tilde{\mathscr D}_z-2)
+(\tilde{\mathscr D}_y+2)(\tilde{\mathscr J}+1/2)^\kappa(\tilde{\mathscr D}_y-1)^2,\quad \kappa\in2\mathbb N_0.
\end{gather*}
\end{theorem}

\begin{proof}
Denote by~$\bar{\mathfrak T}$ a complementary subspace of the span of~$\mathfrak T$
in the span of the set run by~$\tilde{\mathfrak Q}$ in the third family of Theorem~\ref{thm:IDFM:CLs}.
Since conserved currents from the first and second families of Theorem~\ref{thm:IDFM:CLs}
are $(t,x)$-translation-invariant,
it suffices to prove that conserved currents of the form~\eqref{eq:IDFM:FormOfCCsFrom3rdFamily}
with $\tilde{\mathfrak Q}\in\mathfrak T$ (resp.\ with nonzero $\tilde{\mathfrak Q}\in\bar{\mathfrak T}$)
are equivalent (resp.\ not equivalent) to $(t,x)$-translation-invariant ones.

For each $\tilde{\mathfrak Q}\in\mathfrak T$ we explicitly construct
a related $(t,x)$-translation-invariant conserved current.
To this end, we consider the associated operator~$\mathfrak Q$ in~$\tilde\Lambda^{\rm q}_-$,
choose an appropriate conserved current of the Klein--Gordon equation~\eqref{eq:IDFMSupersystemA}
among those equivalent to~${\rm C}_{\mathfrak Q}$
and map it to a conserved current of the system~\eqref{eq:IDFMDiagonalizedSystem}.
Each operator~$\mathfrak Q\in\tilde\Lambda^{\rm q}_-$ associated with some $\tilde{\mathfrak Q}\in\mathfrak T$
is equivalent, on solutions of~\eqref{eq:IDFMSupersystemA}, to an operator
of the form $(\mathrm D_z+1)^2\mathfrak P(\mathrm D_z-1)^2$ with $(\mathfrak Pq)\p_q\in\tilde\Lambda^{\rm q}_-$,
where the operator~$\mathfrak P$ coincides with
$(\mathrm J-\iota/2)^\kappa\mathrm D_z^\iota$,
$(\mathrm J+\iota/2+2)^\kappa\mathrm D_z^{\iota+4}$,
$(\mathrm J+1/2)^\kappa\mathrm D_z$,
$(\mathrm J+1)^\kappa\mathrm D_z^2$,
$(\mathrm J+3/2)^\kappa\mathrm D_z^3$
for
$\tilde{\mathfrak Z}_{\kappa\iota}$,
$\tilde{\mathfrak Y}_{\kappa,\iota+4}$,
$\tilde{\mathfrak Y}_{\kappa1}$,
$\tilde{\mathfrak Y}_{\kappa2}$ and
$\tilde{\mathfrak Y}_{\kappa3}$, respectively.
For such~$\mathfrak Q$ we obtain
\begin{gather*}
{\rm C}_{\mathfrak Q}\sim(-K^1\mathscr D_z\mathfrak PK^1,K^2\mathfrak PK^1)\\
\qquad \mapsto\ 2{\rm e}^{(\ri^1-\ri^2)/2}\bigg(
(\tilde{\mathscr D}_z+1)\tilde{\mathfrak P}\frac{{\rm e}^{(\ri^1-\ri^2)/2}}{\ri^2_x},\
(V^2\tilde{\mathscr D}_z+V^1)\tilde{\mathfrak P}\frac{{\rm e}^{(\ri^1-\ri^2)/2}}{\ri^2_x}
\bigg),
\end{gather*}
which is obviously a $(t,x)$-translation-invariant conserved current of the system~\eqref{eq:IDFMDiagonalizedSystem}.

As a subspace complementary to the span of~$\mathfrak T$, we can choose
\[
\bar{\mathfrak T}=\big\langle\,
\mathrm J^{2\kappa+1},\,
(\mathrm J+1)^{2\kappa+1}\mathrm D_z^2,\,
(\mathrm J+1/2)^{2\kappa}\mathrm D_z,\,
(\mathrm J+3/2)^{2\kappa}\mathrm D_z^3,\, \kappa\in\mathbb N_0 \,\big\rangle.
\]
We prove by contradiction that for any nonzero $\tilde{\mathfrak Q}\in\bar{\mathfrak T}$, i.e.,
\[
\tilde{\mathfrak Q}=\sum_{\kappa=0}^N\Big(
c_{0\kappa}\mathrm J^{2\kappa+1}
+c_{2\kappa}(\mathrm J+1)^{2\kappa+1}\mathrm D_z^2
+c_{1\kappa}(\mathrm J+1/2)^{2\kappa}\mathrm D_z
+c_{3\kappa}(\mathrm J+3/2)^{2\kappa}\mathrm D_z^3
\Big)
\]
for some $N\in\mathbb N_0$ and some constants~$c$'s with $(c_{0N},c_{1N},c_{2N},c_{3N})\ne(0,0,0,0)$,
the corresponding conserved current of the form~\eqref{eq:IDFM:FormOfCCsFrom3rdFamily}
is not equivalent to a $(t,x)$-translation-invariant one.
Suppose that this is not the case.
If a conservation law of the system~\eqref{eq:IDFMDiagonalizedSystem} is $(t,x)$-translation-invariant,
then its characteristic is also $(t,x)$-translation-invariant.
The conservation-law characteristic associated with $\tilde{\mathfrak Q}$ (see Theorem~\ref{thm:IDFM:CLChars})
does not depend on $x$ and~$t$ if and only if
$(\tilde{\mathfrak Q}\tilde q)_x= \tilde{\mathfrak Q}{\rm e}^{(\ri^1-\ri^2)/2}=0$ and
$(\tilde{\mathfrak Q}\tilde q)_t=-\tilde{\mathfrak Q}\big((\ri^1+\ri^2+1){\rm e}^{(\ri^1-\ri^2)/2}\big)=0$.
In the coordinates~\eqref{eq:IDFMTransReducingToKGEq},
these conditions, after re-combining, take the form
$\mathfrak Q{\rm e}^{y+z}=0$, $\mathfrak Q\big((y-z){\rm e}^{y+z}\big)=\mathfrak Q\mathrm J{\rm e}^{y+z}=0$,
or, equivalently,
\begin{gather*}
R^1:=\sum_{\kappa=0}^N\Big(
c_{0\kappa}\mathrm J^{2\kappa+1}
+c_{2\kappa}(\mathrm J+1)^{2\kappa+1}
+c_{1\kappa}(\mathrm J+1/2)^{2\kappa}
+c_{3\kappa}(\mathrm J+3/2)^{2\kappa}
\Big){\rm e}^{y+z}=0,
\\
R^2:=\sum_{\kappa=0}^N\Big(
c_{0\kappa}\mathrm J^{2\kappa+2}
+c_{2\kappa}(\mathrm J+1)^{2\kappa+1}(\mathrm J-2)
+c_{1\kappa}(\mathrm J+1/2)^{2\kappa}(\mathrm J-1)\\ \qquad\qquad\quad{}
+c_{3\kappa}(\mathrm J+3/2)^{2\kappa}(\mathrm J-3)
\Big){\rm e}^{y+z}=0.
\end{gather*}
The left-hand sides of these equations, $R^1$ and~$R^2$,
are polynomials of~$y-z$ and~$y+z$ multiplied by ${\rm e}^{y+z}$,
and the highest degrees of~$y-z$ correspond to the highest degrees of~$\mathrm J$.
Recombining these equations to
\begin{gather*}
R^2-\mathrm JR^1=-\sum_{\kappa=0}^N\Big(
2c_{2\kappa}(\mathrm J+1)^{2\kappa+1}
+c_{1\kappa}(\mathrm J+1/2)^{2\kappa}
+3c_{3\kappa}(\mathrm J+3/2)^{2\kappa}
\Big){\rm e}^{y+z}=0,
\\
R^2-(\mathrm J-2)R^1=\sum_{\kappa=0}^N\Big(
2c_{0\kappa}\mathrm J^{2\kappa+1}
+c_{1\kappa}(\mathrm J+1/2)^{2\kappa}
-c_{3\kappa}(\mathrm J+3/2)^{2\kappa}
\Big){\rm e}^{y+z}=0,
\end{gather*}
we easily see that $c_{0N}=c_{2N}=0$ and thus also $c_{1N}=c_{2N}=0$,
which contradicts the supposition $(c_{0N},c_{1N},c_{2N},c_{3N})\ne(0,0,0,0)$.
\end{proof}

In order to construct a lowest-order $(t,x)$-translation-invariant conserved current
for conservation laws associated with operators from~$\mathfrak T$,
for the respective operator~$\mathfrak P$ we should take the respective (up to a constant multiplier)
conserved current among
${\rm C}^1_{\kappa\iota}$, $\kappa\in\mathbb N_0$, $\iota\in\mathbb N$,
$\bar{\rm C}^1_{\kappa\iota}$, ${\rm C}^2_{\kappa\iota}$, $\bar{\rm C}^2_{\kappa\iota}$, $\kappa,\iota\in\mathbb N_0,$
presented in~\cite[Section~4]{OpanasenkoPopovych2018},
formally replace $(x,y,u)$ by $(y,z,K^1)$ and map the obtained conserved current.
In particular, linearly independent $(t,x)$-translation-invariant inequivalent conserved currents up to order two
from the span of the third family of Theorem~\ref{thm:IDFM:CLs} are exhausted by the following:
\begin{gather*}
\tilde{\mathfrak Q}=\tilde{\mathfrak Y}_{01}=\tilde{\mathscr D}_z^3-2\tilde{\mathscr D}_z+\tilde{\mathscr D}_y\colon\quad
\mathfrak P=\mathscr D_y,\quad
{\rm C}_{\mathfrak Q}\sim\big(-(K^1)^2,\,(K^2)^2\big) \\
\qquad \mapsto\ 2{\rm e}^{\ri^1-\ri^2}\left(\frac1{\ri^2_x}-\frac1{\ri^1_x},\ \frac{V^2}{\ri^2_x}-\frac{V^1}{\ri^1_x}\right),
\\[1ex]
\tilde{\mathfrak Q}=\tilde{\mathfrak Y}_{03}=\tilde{\mathscr D}_y^3-2\tilde{\mathscr D}_y+\tilde{\mathscr D}_z\colon\quad
\mathfrak P=\mathscr D_y^3,\quad
{\rm C}_{\mathfrak Q}\sim\big((K^2)^2,\,-(\mathscr D_yK^2)^2\big) \\
\qquad \mapsto\ \frac2{(\ri^1_x)^5}{\rm e}^{\ri^1-\ri^2}\Big(
(2\ri^1_{xx}+\ri^1_x\ri^2_x)^2-\ri^2_x(\ri^1_x)^3,\
V^1(2\ri^1_{xx}+\ri^1_x\ri^2_x)^2-V^2\ri^2_x(\ri^1_x)^3
\Big),
\\[1ex]
\tilde{\mathfrak Q}=\tilde{\mathfrak Z}_{01}=\tilde{\mathscr D}_z^5-2\tilde{\mathscr D}_z^3+\tilde{\mathscr D}_z\colon\quad
\mathfrak P=\mathscr D_z,\quad
{\rm C}_{\mathfrak Q}\sim\big((\mathscr D_zK^1)^2,\,-(K^1)^2\big) \\
\qquad \mapsto\ \frac{-2}{(\ri^2_x)^5}{\rm e}^{\ri^1-\ri^2}\Big(
(2\ri^2_{xx}-\ri^1_x\ri^2_x)^2-\ri^1_x(\ri^2_x)^3,\
V^2(2\ri^2_{xx}-\ri^1_x\ri^2_x)^2-V^1\ri^1_x(\ri^2_x)^3
\Big),
\\[1ex]
\tilde{\mathfrak Q}=\tilde{\mathfrak Z}_{10}:=(\tilde{\mathscr D}_z+1)^2\tilde{\mathscr J}(\tilde{\mathscr D}_z-1)^2\colon\quad
\mathfrak P=\mathscr J,\\
\qquad {\rm C}_{\mathfrak Q}\sim\big(-y(K^1)^2-z(\mathscr D_zK^1)^2,\,y(K^2)^2+z(K^1)^2\big)\ \mapsto
-{\rm e}^{\ri^1-\ri^2}(\mathfrak z^1+\mathfrak z^2,\,V^1\mathfrak z^1+V^2\mathfrak z^2),
\\
\qquad
\mathfrak z^1:=\frac{\ri^1}{\ri^1_x}-\frac{\ri^2\ri^1_x}{(\ri^2_x)^2},\quad
\mathfrak z^2:=\frac{\ri^2}{(\ri^2_x)^5}(2\ri^2_{xx}-\ri^1_x\ri^2_x)^2-\frac{\ri^1}{\ri^2_x}.
\end{gather*}

\section{Hamiltonian structures of hydrodynamic type}\label{sec:IDFMHamiltonianStructure}

A system~$\mathcal E$ of evolution differential
equations $\mathbf u_t-K[\mathbf u]=0$, where~$K$ is a tuple of functions of independent variables~$(t,\mathbf x)$
and spatial derivatives (including ones of order zero) of the dependent variables $\mathbf u=(u^1,\dots,u^n)^{\mathsf T}$, is called Hamiltonian
if it can be represented in the form
$
\mathbf u_t=\mathfrak H\,\delta\mathcal H.
$
Here~$\mathfrak H$ is a Hamiltonian differential operator, i.e.\ a formally skew-adjoint matrix differential operator,
whose associated bracket~$\{\cdot,\cdot\}$
defined by $\{\mathcal I,\mathcal J\}=\int\delta\mathcal I\cdot \mathfrak H\,\delta\mathcal J\,\mathrm d\mathbf x$
for appropriate functionals~$\mathcal I$ and~$\mathcal J$,
satisfies the Jacobi identity and thus is a Poisson bracket,
$\delta$~stands for the variational derivative,
and the functional $\mathcal H$ is called a Hamiltonian of~$\mathcal E$ with respect to~$\mathfrak H$, see~\cite{DubrovinNovikov1989}.

A procedure for finding a Hamiltonian structure for the system~$\mathcal E$ is as follows:
\begin{itemize}\itemsep=0ex

\item
For the left hand side~$F:=\mathbf u_t-K[\mathbf u]$ of the system~$\mathcal E$, one defines the universal linearization
operator~$\ell_F$ of $F$~\cite{Bocharov1999} (also known as the Fr\'echet derivative of~$F$~\cite{Olver1993})
and its formally adjoint~$\ell^\dag_F$ to determine the linearization of the system~$\mathcal E$ and the system adjoint to the linearization,
\[
\ell_F(\eta)=0,\quad \ell^\dag_F(\lambda)=0.
\]
The differential vector functions~$\eta$ and $\lambda$ of~$\mathbf u$, that is, vector functions of $t$, $\mathbf x$, $\mathbf u$
and their spatial derivatives (time derivatives are excluded in view of the evolutionary form of the equations), satisfying the above
systems in view of the system~$\mathcal E$ are nothing else but symmetries (more precisely, characteristic-tuples of generalized symmetries) and cosymmetries for the
system~$\mathcal E$, respectively.

\item
By making an ansatz one finds Noether operators, which are by definition matrix differential operators mapping cosymmetries
of the system to its symmetries.

\item
One selects a Hamiltonian operator~$\mathfrak H$ amongst Noether ones, by requiring that it is skew-adjoint
and the associated bracket satisfies the Jacobi identity.

\item
Choosing an ansatz for a Hamiltonian~$\mathcal H$, one finds it from the condition~$\mathfrak H\delta\mathcal H=K$.

\end{itemize}

Skew-adjoint Noether operators of systems of evolution equations are believed to satisfy the Jacobi identity automatically
except for first-order scalar equations~\cite[Theorem~5]{KerstenKrasilshchikVerbovetsky2006}.
This result was rigorously proved for systems of evolution equations of order greater than one in~\cite{Gessler1997},
while the same assertion for non-scalar systems of first-order evolution equations was conjectured in~\cite{KerstenKrasilshchikVerbovetsky2004}.
In spite of the fact that, in general, the verification of this conjecture for the system~$\mathcal S$ can be done directly,
we use a geometrical interpretation of hydrodynamic-type Hamiltonian differential operators for hydrodynamic-type systems.
Hereafter we consider a (1+1)-dimensional (translation-invariant) hydrodynamic-type system~$\mathcal E$,
the indices $i$, $j$, $k$ and~$l$ run from~1 to~$n$, and the Einstein summation convention is assumed for the index~$l$.
A matrix differential operator~$\mathfrak D=(\mathfrak D^{ij})$ and the associated bracket
are said to be \emph{of hydrodynamic type} or \emph{of Dubrovin--Novikov type}
if the entries of~$\mathfrak D$ are of the form $\mathfrak D^{ij}=g^{ij}(u)\mathrm D_x+b^{ij}_l(u)u^l_x$.

The cornerstone of the geometric interpretation of hydrodynamic-type Hamiltonian operators,
discovered in the seminal paper~\cite{DubrovinNovikov1983},
is the fact that under a point transformation~$\tilde {\mathbf u}=U(\mathbf u)$ of dependent variables only,
the coefficients~$g^{ij}$ of~$\mathfrak D$ are transformed as components of a second-order contravariant tensor
on the space with the coordinates~$\mathbf u$
and, if the tensor~$(g^{ij})$ is nondegenerate (which is a perpetual assumption below),
the coefficients~$b^{ij}_k$ are transformed so
that~$\Gamma^j_{lk}$ defined by $g^{il}\Gamma^j_{lk}=-b^{ij}_k$ are the Christoffel symbols of
a connection~$\nabla$ on this space.
The bracket associated with~$\mathfrak D$ is skew-symmetric
if and only if the tensor~$(g^{ij})$ is symmetric, i.e.,
$g=(g_{i'\!j'\!})=(g^{ij})^{-1}$ is a \mbox{(pseudo-)}Riemannian metric,
and the connection~$\nabla$ agrees with~$g$, $\nabla g=0$.
The bracket satisfies the Jacobi identity if and only if the metric~$g$ is flat
and the connection~$\nabla$ is the Levi-Civita connection of~$g$,
i.e., the curvature tensor of~$g$ and the torsion tensor of~$\nabla$ vanish.

Recall that two Hamiltonian operators are called compatible if any their linear combination is a Hamiltonian operator as well.
Two nondegene\-rate hydro\-dyna\-mic-type Hamiltonian operators for a hydrodynamic-type system is compatible
if the Nijenhuis tensor~$\mathcal N$ of the tensor~$(s^i_j)$ defined by~$s^i_j=\tilde g^{il}g_{lj}$ vanishes,
\[
\mathcal N^i_{jk}:=s^l_j\p_{u^l}s^i_k-s^l_k\p_{u^l}s^i_j-s^i_l(\p_{u^j}s^l_k-\p_{u^k}s^l_j)=0,
\]
see~\cite{Ferapontov2001,Mokhov1999}.
Here~$g$ and~$\tilde g$ are the metrics corresponding to the Hamiltonian operators.
In terms of~$g$ and~$\tilde g$, the condition of vanishing the Nijenhuis tensor~$\mathcal N$
takes the form
\begin{gather}\label{eq:IDFMCompatibilityCondition}
\nabla^i\nabla^j\tilde g^{kl}+\nabla^k\nabla^l\tilde g^{ij}-\nabla^i\nabla^k\tilde g^{jl}-\nabla^j\nabla^l\tilde g^{ik}=0.
\end{gather}
The covariant differentiation in~\eqref{eq:IDFMCompatibilityCondition} corresponds to the metric~$g$.
The conditions~\eqref{eq:IDFMCompatibilityCondition} are preserved by the permutation of~$g$ and~$\tilde g$,
so that they are indeed the compatibility conditions of the two metrics.

When the tensor~$g$ degenerates at some point,
the associated hydrodynamic-type system loses its geometric charm and one needs to proceed otherwise.
To show that the bracket of a skew-symmetric Noether operator~$\mathfrak N$ for~$\mathcal E$ satisfies the Jacobi identity,
one may equivalently check that the variational Schouten bracket $[\![\mathfrak N,\mathfrak N]\!]$ vanishes.
To show the compatibility of two hydrodynamic-type Hamiltonian operators~$\mathfrak H_1$ and~$\mathfrak H_2$,
$\mathfrak H_k^{ij}=g^{ij}_k\mathrm D_x+b^{ij}_{kl}\ri^l_x$, $k=1,2$,
one may check that $[\![\mathfrak H_1,\mathfrak H_2]\!]=0$, cf.~\cite[Section~10.1]{KrasilshchikVerbovetskyVitolo2017}.
Since $\mathcal E$ is a system of evolution equations, one may consider the
cotangent covering~$T^\ast\mathcal E$ of~$\mathcal E$ (i.e., the joint system $F=0$, $\ell^\dag_F(\lambda)=0$)
and substitute the latter condition by the equivalent one
\[
\mathsf E\sum\limits_{j=1}^n\Big(({\mathsf E_{u^j}\mathrm F_{\mathfrak H_1}})(\mathsf E_{\lambda^j}\mathrm F_{\mathfrak H_2})
+(\mathsf E_{\lambda^j}\mathrm F_{\mathfrak H_1})(\mathsf E_{u^j}\mathrm F_{\mathfrak H_2})\Big)=0,
\]
where $\mathrm F_{\mathfrak H_k}=\sum_{i,j}\big(g^{ij}_{k}(\mathrm D_x\lambda^i)\lambda^j+b^{ij}_{kl}\ri^l_x \lambda^i\lambda^j\big)$, $k=1,2$,
and $\mathsf E=(\mathsf E_{u^1},\dots,\mathsf E_{u^n},\mathsf E_{\lambda^1},\dots,\mathsf E_{\lambda^n})$ is the Euler operator on~$T^\ast\mathcal E$.

\begin{theorem}\label{theorem:IDFMHamiltonianStructures}
The system~\eqref{eq:IDFMDiagonalizedSystem} admits an infinite family of compatible Hamiltonian structures
$\mathfrak H_\Theta$ parameterized by a smooth function~$\Theta$ of~$\ri^3$,
\begin{gather}\label{eq:IDFMHamiltonianOperator}
\mathfrak H_\Theta={\rm e}^{\ri^2-\ri^1}\left(\mathop{\rm diag}\big(-1,1,\Theta(\ri^3){\rm e}^{\ri^2-\ri^1}\big)\mathrm D_x-\frac12
\begin{pmatrix}
\ri^2_x-\ri^1_x  &  \ri^1_x-\ri^2_x  &  -2\ri^3_x\\[1ex]
\ri^2_x-\ri^1_x  &  \ri^1_x-\ri^2_x  &  -2\ri^3_x\\[1ex]
2\ri^3_x         &  2\ri^3_x         &  -2f^{33}{\rm e}^{\ri^1-\ri^2}
\end{pmatrix}\right)
\end{gather}
with the corresponding family of Hamiltonians $\mathcal H_{c_0,\Xi}=\int H_{c_0,\Xi}\mathrm dx$ defined by densities
\begin{gather}\label{eq:IDFMHamiltonian}
H_{c_0,\Xi}=(\ri^1+\ri^2)^2{\rm e}^{\ri^1-\ri^2}+c_0(\ri^1+\ri^2)+2\big(\ri^1-\ri^2+\Xi(\ri^3)\big){\rm e}^{2(\ri^1-\ri^2)}.
\end{gather}
Here $f^{33}:={\rm e}^{2\ri^2-2\ri^1}\left((\ri^2_x-\ri^1_x)\Theta+\ri^3_x\Theta_{\ri^3}\right)$,
$c_0$ is an arbitrary constant and the function~$\Xi$ of~$\ri^3$ satisfies the auxiliary condition
\begin{gather*}
|\Theta|^{1/2}\Xi_{\ri^3\ri^3}+\smash{\frac12}{\Theta}_{\ri^3}\Xi_{\ri^3}=c_0.
\end{gather*}
\end{theorem}

\begin{proof}
For the first step of the algorithm expounded above we need to consider the joint system of equations
\begin{subequations}\label{eq:IDFMSystem}
\begin{gather}
\ri^k_t+V^k\ri^k_x=0,\label{eq:IDFMRiemannianSystem}\\
\mathrm D_t\eta^k+V^k\mathrm D_x\eta^k+(\eta^1+\eta^2)\ri^k_x=0,\label{eq:IDFMLinearizedSystem}\\
\mathrm D_t\lambda^k+\mathrm D_x(V^k\lambda^k)-\ri^l_x\lambda^l(\delta^1_k+\delta^2_k)=0,\label{eq:IDFMAdjointSystem}
\end{gather}
\end{subequations}
where the system~\eqref{eq:IDFMRiemannianSystem} is the system~$\mathcal S$ itself,
the system~\eqref{eq:IDFMLinearizedSystem} is the linearization of~$\mathcal S$,
and the system~\eqref{eq:IDFMAdjointSystem} is adjoint to~\eqref{eq:IDFMLinearizedSystem}.
If $\eta^i$ and $\lambda^j$ are differential functions of~$\ri$,
then the tuples $\eta=(\eta^1,\eta^2,\eta^3)^{\mathsf T}$ and $\lambda=(\lambda^1,\lambda^2,\lambda^3)$ are
a symmetry characteristic and a cosymmetry of the system~$\mathcal S$, respectively.
Here and in what follows a summation with respect to~$i$, $j$ and~$l$,
which run through the set $\{1,2,3\}$, is assumed,
there is no summation with respect to~$k$, which is a fixed number from the set $\{1,2,3\}$,
and $\delta^i_j$ stands for the Kronecker delta.
We are looking for a Noether operator~\smash{$\mathfrak N=(\mathfrak{N}^{ij})$} with entries of the form
\begin{gather}\label{eq:AnsatzForNoetherOps}
\mathfrak N^{ij}=h^{ij}(\ri)\mathrm D_x+f^{ij}(\ri,\ri_x)
\end{gather}
for some smooth functions~$h^{ij}$ and~$f^{ij}$ of their arguments.
By definition of Noether operators, for any solution $(\ri,\lambda)$ of~\eqref{eq:IDFMRiemannianSystem}, \eqref{eq:IDFMAdjointSystem}
the expressions $\eta^i:=\mathfrak N^{ij}\lambda^j=h^{ij}\mathrm D_x\lambda^j+f^{ij}\lambda^j$
give a solution of~\eqref{eq:IDFMLinearizedSystem}.
This implies the system
\begin{gather*}
-h^{kj}_{\ri^l}V^l\ri^l_x\mathrm D_x\lambda^j+h^{kj}\mathrm D_x\left(-\mathrm D_x(V^j\lambda^j)+
\ri^l_x\lambda^l\left(\delta^1_j+\delta^2_j\right)\right)-
f^{kj}_{\ri^l_x}\left(V^l\ri^l_{xx}+(\ri^1_x+\ri^2_x)\ri^l_x\right)\lambda^j\\\qquad{}
-f^{kj}_{\ri^l}V^l\ri^l_x\lambda^j+f^{kj}\left(-\mathrm D_x(V^j\lambda^j)+
\ri^l_x\lambda^l\left(\delta^1_j+\delta^2_j\right)\right)+
V^k\big(h^{kj}_{\ri^l}\ri^l_x\mathrm D_x\lambda^j+h^{kj}\mathrm D^2_x\lambda^j\\\qquad{}
+f^{kj}_{\ri^l}\ri^l_x\lambda^j+f^{kj}_{\ri^l_x}\ri^l_{xx}\lambda^j+f^{kj}\mathrm D_x\lambda^j\big)
+\ri^k_x\left(h^{1j}\mathrm D_x\lambda^j+f^{1j}\lambda^j+h^{2j}\mathrm D_x\lambda^j+f^{2j}\lambda^j\right)=0.
\end{gather*}
Collecting coefficients of~$\mathrm D^2_x\lambda^j$ immediately leads to~$h^{kj}=0$
for all $j\ne k$, and further collecting coefficients of~$\mathrm D_x\lambda^j$ yields
\begin{gather*}
f^{21}=-f^{12}=\frac12\left(h^{11}\ri^2_x+h^{22}\ri^1_x\right),\quad f^{31}=
-f^{13}=h^{11}\ri^3_x,\quad f^{23}=-f^{32}=h^{22}\ri^3_x,\\
h^{11}_{\ri^3}=0,\quad h^{11}_{\ri^2}=h^{11},\quad h^{22}_{\ri^3}=0,\quad h^{22}_{\ri^1}=
-h^{22},\quad h^{33}_{\ri^1}=-2h^{33},\quad h^{33}_{\ri^2}=2h^{33}.
\end{gather*}
Finally, splitting with respect to~$\lambda^j$ and derivatives of~$\ri^l$ allows us to deduce
that the operator~$\mathfrak N$ is of the form~\eqref{eq:IDFMHamiltonianOperator} with
\[
f^{33}=\Theta(\ri^3){\rm e}^{2\ri^2-2\ri^1}(\ri^2_x-\ri^1_x)+\Psi(\ri^3,\ri^3_x{\rm e}^{\ri^2-\ri^1}){\rm e}^{\ri^2-\ri^1},
\]
where~$\Theta$ and~$\Psi$ are arbitrary smooth functions of their arguments. To be qualified as a Hamiltonian
operator, the operator~$\mathfrak N$ should be formally skew-adjoint, $\mathfrak N^\dag=-\mathfrak N$, yielding
the Noether operator~$\mathfrak H_\Theta$ of the form~\eqref{eq:IDFMHamiltonianOperator} with~$f^{33}$ as in the statement of the theorem.

First consider the case when~$\Theta$ is nonvanishing.
The operator~$\mathfrak H_\Theta$ is of hydrodynamic type with the pseudo-Riemannian metric
\begin{gather}\label{eq:IDFMMetric}
g=\mathop{\rm diag}\big( -{\rm e}^{\ri^2-\ri^1}\!,\, {\rm e}^{\ri^2-\ri^1}\!,\, \Theta(\ri^3) {\rm e}^{2\ri^2-2\ri^1}\big).
\end{gather}
It is easy to show that the coordinates~$\ri$ are Liouville ones, cf.~\cite{DubrovinNovikov1989}.
The connection~$\nabla$ associated with~$\mathfrak H_\Theta$ in the sense discussed above
is the Levi-Civita connection for~$g$.
Thus we should check that the corresponding Riemann curvature tensor vanishes.
Due to its symmetries, we only need to verify that $R^i_{\ jij}=0$ for~$i\ne j$.
This is easily computed to be true.
Thus the Noether operator~$\mathfrak H_\Theta$ is a Hamiltonian one.

Finally, we are looking for a Hamiltonian~$\mathcal H=\int H(\ri)\,\mathrm{d}x$ of~$\mathcal S$ with respect to $\mathfrak H_\Theta$. It satisfies the condition
\begin{gather*}
\mathfrak H_\Theta\frac{\delta\mathcal H}{\delta \ri}=
-\begin{pmatrix}
V^1\ri^1_x \\[.5ex] V^2\ri^2_x \\[.5ex] V^3\ri^3_x
\end{pmatrix},
\end{gather*}
where $\delta\mathcal H/\delta \ri$ is the vector of variational derivatives of~$\mathcal H$ with respect to the Riemann
invariants~$\ri^1$, $\ri^2$ and~$\ri^3$, $\delta\mathcal H/\delta \ri=(H_{\ri^1},H_{\ri^2},H_{\ri^3})^{\mathsf T}$
due to the fact that~$H$ is a function of~$\ri$ only. Expanding this condition we find the system of differential equations on~$H$,
\begin{subequations}\label{eq:IDFMHamiltonianSystem}
\begin{gather}
{\rm e}^{\ri^2-\ri^1}\left(-\mathrm{D}_xH_{\ri^1}+\frac{\ri^2_x-\ri^1_x}{2}(H_{\ri^2}-H_{\ri^1})+\ri^3_xH_{\ri^3}\right)=-2V^1\ri^1_x,\label{eq:IDFMHamiltonianSystemEq1}\\
{\rm e}^{\ri^2-\ri^1}\left(\mathrm{D}_xH_{\ri^2}+\frac{\ri^2_x-\ri^1_x}{2}(H_{\ri^2}-H_{\ri^1})+\ri^3_xH_{\ri^3}\right)=-2V^2\ri^2_x,\label{eq:IDFMHamiltonianSystemEq2}\\
\begin{split}\label{eq:IDFMHamiltonianSystemEq3}
&{\rm e}^{2\ri^2-2\ri^1}\left(-\ri^3_x(H_{\ri^1}+H_{\ri^2}){\rm e}^{\ri^1-\ri^2}+\mathrm{D}_x(H_{\ri^3})\Theta+ \left((\ri^2_x-\ri^1_x)\Theta+\ri^3_x\Theta_{\ri^3}\right)H_{\ri^3}\right)=-2V^3\ri^3_x.
\end{split}
\end{gather}
\end{subequations}
Successively splitting with respect to~$\ri^1_x$, $\ri^2_x$ and~$\ri^3_x$
and solving the obtained overdetermined system of differential equations,
we find the final form~\eqref{eq:IDFMHamiltonian} for Hamiltonian densities
and the auxiliary condition on~$\Xi$.

For the system~\eqref{eq:IDFMDiagonalizedSystem}
the tensor~\smash{$(s^i_j)$} takes a particularly simple form, \smash{$(s^i_j)=\mathop{\rm diag}(1,1,\tilde{\Theta}/\Theta)$}, where~$\Theta$ and~\smash{$\tilde\Theta$}
are functions of~$\ri^3$ parameterizing the metrics~$g$ and~$\tilde g$.
It is trivial to verify that its Nijenhuis tensor vanishes. Since eigenvalues of~$(s^i_j)$ are not distinct,
we need also to verify the conditions~\eqref{eq:IDFMCompatibilityCondition}, and they also hold.

If $\Theta$ is a somewhere vanishing function, then the geometric reasoning for Hamiltonian operators is no longer available,
and we should proceed by establishing that the corresponding variational Schouten brackets vanish, which is done symbolically.
\end{proof}

\begin{remark}
It is worth noting that provided~$\Xi_{\ri^3}\neq0$ the condition on~$\Xi$ can be equivalently represented as
\begin{gather*}
\Theta=\frac{c_0\Xi+c_1}{\Xi_{\ri^3}^{\ 2}},
\end{gather*}
where $c_1$ is an arbitrary constant.
\end{remark}

For preliminary computations and testing the above results,
we used the package {\sf Jets} \cite{{BaranMarvan},Marvan2009} for {\sf maple}.

Below we consider only canonical representatives of symmetry-type objects,
where derivatives involving differentiations with respect to~$t$ are
replaced by their expressions in view of the system~$\mathcal S$,
which is necessary for relating different kinds of such objects via Hamiltonian structures.

For any Hamiltonian operator~$\mathfrak H_\Theta$ from Theorem~\ref{theorem:IDFMHamiltonianStructures},
we can endow the space~$\hat\Upsilon^{\rm q}$ of canonical representatives for cosymmetries of~$\mathcal S$
with a Lie-algebra structure, cf.\ \cite{Fuchssteiner1982b} and \cite[Section~3.1]{Blaszak1998a},
where the corresponding Lie bracket is defined~by
\[
[\gamma^1,\gamma^2]_{\mathfrak H_\Theta}
=\ell_{\gamma^2}\mathfrak H_\Theta\gamma^1
+\ell_{\mathfrak H_\Theta\gamma^1}^\dag\gamma^2
+(\ell_{\gamma^1}-\ell_{\gamma^1}^\dag)\mathfrak H_\Theta\gamma^2
\]
for any~$\gamma^1,\gamma^2\in\hat\Upsilon^{\rm q}$.
Here $\ell_\gamma$ and \smash{$\ell_\gamma^\dag$}
are the universal linearization operator of~$\gamma\in\hat\Upsilon^{\rm q}$ and its formal adjoint, respectively.
Denote the Lie algebra with the underlying space~$\hat\Upsilon^{\rm q}$
and the Lie bracket $[\cdot,\cdot]_{\mathfrak H_\Theta}$ by~\smash{$\hat\Upsilon^{\rm q}_\Theta$}.
The operator~$\mathfrak H_\Theta$ establishes
a homomorphism from the Lie algebra~\smash{$\hat\Upsilon^{\rm q}_\Theta$} to the Lie algebra~$\hat\Sigma^{\rm q}$.
The image \smash{$\mathfrak H_\Theta\hat\Upsilon^{\rm q}_\Theta$} of this homomorphism is a proper subalgebra of~$\hat\Sigma^{\rm q}$
of canonical representatives for generalized symmetries of the system~$\mathcal S$.
More specifically, the image \smash{$\mathfrak H_\Theta\hat\Upsilon^{\rm q}_\Theta$}
is spanned by generalized symmetries from three families
that are the images of the respective families from Theorem~\ref{thm:IDFM:Cosyms}
and whose elements are, in the notation of Theorems~\ref{thm:IDFMGenSyms} and~\ref{thm:IDFM:Cosyms},
of the following form:
\begin{enumerate}
\item
$\check{\mathcal W}(\bar\Omega^\Theta)$,
where $\bar\Omega^\Theta=\hat{\mathscr A}\big((\hat{\mathscr A}\Omega)\Theta/\omega^1\big)$,
\item
$\check{\mathcal P}(\bar\Phi)$,
where $\bar\Phi=\Phi_{\ri^1}-\frac12\Phi$,
and thus the parameter function~$\bar\Phi=\bar\Phi(\ri^1,\ri^2)$
runs through the solution space of the Klein--Gordon equation $\bar\Phi_{\ri^1\ri^2}=-\bar\Phi/4$ as well,
\item
$\check{\mathcal R}(\bar\Gamma)$, where $\bar\Gamma=\frac12(\tilde{\mathscr D}_y-1)\tilde{\mathfrak Q}\tilde q$.
\end{enumerate}
For the nonvanishing function~$\Theta$, the kernel of the above homomorphism is two-dimensional and spanned by the cosymmetries
${\rm e}^{\ri^1-\ri^2}(1,-1,0)$ and ${\rm e}^{\ri^1-\ri^2}(\bar\Theta,-\bar\Theta,\bar\Theta_{\ri^3})$
with an antiderivative~$\bar\Theta$ of~$1/\Theta$, $\bar\Theta_{\ri^3}=1/\Theta$.
The former cosymmetry is special due to
being a single (up to linear independence) common element of the first and the second families from Theorem~\ref{thm:IDFM:Cosyms},
see Remark~\ref{rem:IDFM:IntersectionOfCosymSubspaces}.
Both the cosymmetries are conservation-law characteristics
and are associated with the conserved currents
${\rm e}^{\ri^1-\ri^2}(1,\ri^1+\ri^2)$ and ${\rm e}^{\ri^1-\ri^2}\big(\bar\Theta,(\ri^1+\ri^2)\bar\Theta\big)$,
which belong to the first family of Theorem~\ref{thm:IDFM:CLs}.
As a result, the space of distinguished (Casimir) functionals of the Hamiltonian operator~$\mathfrak H_\Theta$
is spanned by two functionals,
\[
\mathcal C_1:=\int {\rm e}^{\ri^1-\ri^2}\,{\rm d}x,\quad
\mathcal C_2^\Theta:=\int {\rm e}^{\ri^1-\ri^2}\bar\Theta(\ri^3)\,{\rm d}x.
\]
In the degenerate case with $\Theta\equiv0$,
the kernel of the above homomorphism is infinite-dimensional
and coincides with the first family of Theorem~\ref{thm:IDFM:Cosyms}.
Elements of this family are conservation-law characteristics if and only if
they belong to the first family of Theorem~\ref{thm:IDFM:CLChars}
and are thus associated with conserved currents from the first family of Theorem~\ref{thm:IDFM:CLs}.
This means that the space of distinguished (Casimir) functionals of the Hamiltonian operator~$\mathfrak H_0$
consists of the functionals
\[\int {\rm e}^{\ri^1-\ri^2}\Omega(\omega^0,\omega^1,\dots)\,{\rm d}x,\]
where the parameter function~$\Omega$ runs through the space of smooth functions
of a finite, but unspecified number of \smash{$\omega^\kappa=({\rm e}^{\ri^2-\ri^1}\mathscr D_x)^\kappa\ri^3$}, $\kappa\in\mathbb N_0$.

Consider the constraints that single out
the space of canonical representatives conservation-law characteristics of~$\mathcal S$,
which is described in Theorem~\ref{thm:IDFM:CLChars},
from the space~$\hat\Upsilon^{\rm q}$ of canonical representatives of cosymmetries of~$\mathcal S$.
Imposing these constraints on $\Omega$ and~$\tilde{\mathfrak Q}$
that parameterize families spanning $\mathfrak H_\Theta\hat\Upsilon^{\rm q}$,
we single out the algebra of Hamiltonian symmetries of~$\mathcal S$
associated with the Hamiltonian operator~$\mathfrak H_\Theta$.

\begin{theorem}\label{thm:IDFHamiltonianSyms}
Given a smooth function~$\Theta$ of~$\omega^0:=\ri^3$,
the algebra of Hamiltonian symmetries of the system~\eqref{eq:IDFMDiagonalizedSystem}
for the Hamiltonian operator~$\mathfrak H_\Theta$ is spanned by the generalized vector fields
%$\check{\mathcal W}(\bar\Omega^\Theta)$, $\check{\mathcal P}(\Phi)$ and $\check{\mathcal R}(\bar\Gamma)$,
\begin{gather*}
\check{\mathcal W}(\bar\Omega^\Theta)=\bar\Omega^\Theta\p_{\ri^3},
\quad
\check{\mathcal P}(\Phi)={\rm e}^{(\ri^2-\ri^1)/2}\left(
(\Phi+2\Phi_{\ri^1})\ri^1_x\p_{\ri^1}
+(\Phi-2\Phi_{\ri^2})\ri^2_x\p_{\ri^2}
+2\Phi\ri^3_x\p_{\ri^3}\right),
\\
\check{\mathcal R}(\bar\Gamma)={\rm e}^{(\ri^2-\ri^1)/2}\left(
 (\tilde{\mathscr D}_y\bar\Gamma+\bar\Gamma)\ri^1_x\p_{\ri^1}
+(\tilde{\mathscr D}_z\bar\Gamma+\bar\Gamma)\ri^2_x\p_{\ri^2}
+2\bar\Gamma\ri^3_x\p_{\ri^3}\right),
\end{gather*}
where
$\bar\Omega^\Theta=\hat{\mathscr A}\big(\Theta\sum_{\kappa=0}^\infty(-\hat{\mathscr A})^\kappa\Omega_{\omega^\kappa}\big)$
with the operator $\hat{\mathscr A}=\sum_{\kappa=0}^{\infty}\omega^{\kappa+1}\p_{\omega^\kappa}$ and
with %the parameter function
$\Omega$ running through the space of smooth functions
of a finite, but unspecified number of \smash{$\omega^\kappa=({\rm e}^{\ri^2-\ri^1}\mathscr D_x)^\kappa\ri^3$}, $\kappa\in\mathbb N_0$,
the parameter function~$\Phi=\Phi(\ri^1,\ri^2)$
runs through the solution space of the Klein--Gordon equation $\Phi_{\ri^1\ri^2}=-\Phi/4$,
and $\bar\Gamma=\frac12(\tilde{\mathscr D}_y-1)\tilde{\mathfrak Q}\tilde q$
with the operator~$\tilde{\mathfrak Q}$ running through the~set
\begin{gather*}
\big\{\tilde{\mathscr J}^{\kappa'}\!,\,\kappa'\in2\mathbb N_0+1,\
(\tilde{\mathscr J}+\iota/2)^\kappa\tilde{\mathscr D}_y^\iota,\,
(\tilde{\mathscr J}-\iota/2)^\kappa\tilde{\mathscr D}_z^\iota,\,\kappa\in\mathbb N_0,\,\iota\in\mathbb N,\,\kappa+\iota\in2\mathbb N_0+1\big\}.
\\
\hspace*{-\mathindent}\mbox{Here}
\\[-1ex]
\tilde{\mathscr D}_y:=-\frac1{\ri^1_x}\big(\mathscr D_t+(\ri^1+\ri^2-1)\mathscr D_x\big),\quad
\tilde{\mathscr D}_z:=-\frac1{\ri^2_x}\big(\mathscr D_t+(\ri^1+\ri^2+1)\mathscr D_x\big),
\\
\tilde{\mathscr J}:=\frac{\ri^1}2\tilde{\mathscr D}_y+\frac{\ri^2}2\tilde{\mathscr D}_z,\quad
\tilde q:={\rm e}^{(\ri^1-\ri^2)/2}\big(x-(\ri^1+\ri^2+1)t\big).
\end{gather*}
\end{theorem}

\section{Recursion operators}\label{sec:IDFMRecursionOperator}

Some semi-Hamiltonian hydrodynamic-type systems admit Teshukov's recursion operators~\cite{Teshukov1989}
which are specific first-order differential operators without pseudo-differential part.
According to~\cite{Sheftel1994a}, such recursion operators exist
if the Darboux rotation coefficients for an associated metric~$g$,
$\beta_{ik}:=\p_{\ri^i}(\sqrt{|g_{kk}|})/\sqrt{|g_{ii}|}$ for $i\neq k$ and $\beta_{kk}:=0$,
depend at most on pairwise differences of Riemann invariants.
For the system~$\mathcal S$, this condition is satisfied
by the metric~$g$ of the form~\eqref{eq:IDFMMetric} with constant~$\Theta$.
A canonical Teshukov's recursion operator for the system~$\mathcal S$ and such a metric is easily computed,
cf.~\cite[Eq.~(8.1)]{Tsarev1991},
\[
\mathfrak R_{\rm T}=\mathrm D_x\circ\mathop{\rm diag}\left(\dfrac1{\ri^1_x},\dfrac1{\ri^2_x},\dfrac1{\ri^3_x}\right)+
\begin{pmatrix}
\dfrac{\ri^1_x-\ri^2_x}{2\ri^1_x} & \dfrac{\ri^2_x-\ri^1_x}{2\ri^2_x} & 0\\[2.3ex]
\dfrac{\ri^2_x-\ri^1_x}{2\ri^1_x} & \dfrac{\ri^1_x-\ri^2_x}{2\ri^2_x} & 0\\[2.3ex]
\dfrac{\ri^3_x-\ri^1_x}{ \ri^1_x} & \dfrac{\ri^2_x-\ri^3_x}{ \ri^2_x} & \dfrac{\ri^1_x-\ri^2_x}{\ri^3_x}
\end{pmatrix}.
\]
The operator~$\mathfrak R_{\rm T}$ acts
on the generalized vector fields spanning the algebra~$\hat\Sigma^{\rm q}$ as follows
\begin{gather*}
\check{\mathcal D}\mapsto-2\check{\mathcal G}_1-\check{\mathcal G}_2+\check{\mathcal W}(1),\quad
\check{\mathcal R}(\Gamma)\mapsto\frac12\check{\mathcal R}\big(\tilde{\mathscr D}_y\Gamma-\tilde{\mathscr D}_z\Gamma\big),\quad
\check{\mathcal P}(\Phi)\mapsto\check{\mathcal P}\big(\Phi_{\ri^1}+\Phi_{\ri^2}\big),\\
\check{\mathcal W}(\Omega)\mapsto\check{\mathcal W}\left(\mathscr A(\Omega/\omega^1)\right).
\end{gather*}

At the same time, we can construct many more local recursion operators, including higher-order ones.
For this purpose, we use the complete description of generalized symmetries of the system~$\mathcal S$
that is presented in Theorem~\ref{thm:IDFMGenSyms}.
Here the basic fact is again that the algebra~$\hat\Sigma^{\rm q}$
is decomposed into a (non-direct) sum of its subalgebras~$\bar\Sigma^{\rm q}_{12}$ and~$\hat\Sigma^{\rm q}_3$.
The subalgebra~$\bar\Sigma^{\rm q}_{12}$ is a counterpart of the algebra of generalized symmetries
of the (1+1)-dimensional Klein--Gordon equation~\eqref{eq:IDFMSupersystemA},
and thus the recursion operators preserving~$\bar\Sigma^{\rm q}_{12}$
are related to recursion operators of this equation.
The ideal~$\hat\Sigma^{\rm q}_3$ underlaid by the degeneracy of the system~$\mathcal S$
is preserved by the operators of the form
$\mathop{\rm diag}(0,0,\Omega\mathscr A^\kappa)$,
where the coefficient~$\Omega$ is a smooth function
of a finite but unspecified number of $\omega^\iota=\mathscr A^\iota\ri^3$, $\iota\in\mathbb N_0$,
and $\kappa\in\mathbb N_0$.
The above gives a hint about the form of more local recursion operators for the system~$\mathcal S$.

\begin{theorem}\label{thm:IDFMrecursionOps}
The system~\eqref{eq:IDFMDiagonalizedSystem} admits recursion operators of the form
\begin{gather*}
\mathfrak R_{1,\mathfrak Q}={\rm e}^{(\ri^2-\ri^1)/2}
\begin{pmatrix}
\ri^1_x(\tilde{\mathscr D}_y+1)&0&0\\
\ri^2_x(\tilde{\mathscr D}_z+1)&0&0\\
2\ri^3_x&0&0\\
\end{pmatrix}
\mathfrak Q\circ\frac{{\rm e}^{(\ri^1-\ri^2)/2}}{\ri^1_x},
\\[1ex]
\mathfrak R_{2,\mathfrak Q}={\rm e}^{(\ri^2-\ri^1)/2}
\begin{pmatrix}
0&\ri^1_x(\tilde{\mathscr D}_y+1)&0\\
0&\ri^2_x(\tilde{\mathscr D}_z+1)&0\\
0&2\ri^3_x&0\\
\end{pmatrix}
\mathfrak Q\circ\frac{{\rm e}^{(\ri^1-\ri^2)/2}}{\ri^2_x},
\\[1ex]
\mathfrak R_{3,\mathfrak P}=
\mathfrak P
\begin{pmatrix}
0&0&0\\
0&0&0\\
-1&1&\mathscr A\circ(\omega^1)^{-1}\\
\end{pmatrix},
\end{gather*}
where $\mathfrak Q\in\langle
\tilde{\mathscr J}^\kappa,\,
\tilde{\mathscr D}_y^\iota\tilde{\mathscr J}^\kappa,\,
\tilde{\mathscr D}_z^\iota\tilde{\mathscr J}^\kappa,\,
\kappa\in\mathbb N_0,\,\iota\in\mathbb N\rangle$,
$\mathfrak P=\sum_{\kappa=0}^N\Omega^\kappa\mathscr A^\kappa$ for some $N\in\mathbb N_0$,
$\mathscr A:={\rm e}^{\ri^2-\ri^1}\mathscr D_x$,
the coefficients~$\Omega^\kappa$ are smooth functions
of a finite but unspecified number of $\omega^\iota=\mathscr A^\iota\ri^3$, $\iota\in\mathbb N_0$, and
\begin{gather*}
\tilde{\mathscr D}_y:=-\frac1{\ri^1_x}\big(\mathscr D_t+V^2\mathscr D_x\big),\quad % V^2 <-- (\ri^1+\ri^2-1)
\tilde{\mathscr D}_z:=-\frac1{\ri^2_x}\big(\mathscr D_t+V^1\mathscr D_x\big),\quad % V^1 <-- (\ri^1+\ri^2+1)
\tilde{\mathscr J}:=\frac{\ri^1}2\tilde{\mathscr D}_y+\frac{\ri^2}2\tilde{\mathscr D}_z.
\end{gather*}
\end{theorem}

\begin{proof}
We directly compute the action of the operators~$\mathfrak R_{1,\mathfrak Q}$, $\mathfrak R_{2,\mathfrak Q}$ and
$\mathfrak R_{3,\mathfrak P}$
on the generalized vector fields spanning the algebra~$\hat\Sigma^{\rm q}$, obtaining
\begin{gather*}
\mathfrak R_{1,\mathfrak Q}\colon\
\check{\mathcal D}\mapsto\check{\mathcal R}(\mathfrak Q\tilde q),\
\check{\mathcal R}(\Gamma)\mapsto\check{\mathcal R}\big(\mathfrak Q(\tilde{\mathscr D}_y+1)\Gamma\big),\
\check{\mathcal P}(\Phi)\mapsto\check{\mathcal P}\big(\mathfrak Q(\Phi+2\Phi_{\ri^1})\big),\
\check{\mathcal W}(\Omega)\mapsto0.
\\[1ex]
\mathfrak R_{2,\mathfrak Q}\colon\
\check{\mathcal D}\mapsto\check{\mathcal R}(\mathfrak Q\tilde{\mathscr D}_z\tilde q),\
\check{\mathcal R}(\Gamma)\mapsto\check{\mathcal R}\big(\mathfrak Q(\tilde{\mathscr D}_z+1)\Gamma\big),\
\check{\mathcal P}(\Phi)\mapsto\check{\mathcal P}\big(\mathfrak Q(\Phi-2\Phi_{\ri^2})\big),\
\check{\mathcal W}(\Omega)\mapsto0.
\\[1ex]
\mathfrak R_{3,\mathfrak P}\colon\
\check{\mathcal D}\mapsto\check{\mathcal W}(\mathfrak P1)=\check{\mathcal W}(\Omega^0),\quad
\check{\mathcal R}(\Gamma),\check{\mathcal P}(\Phi)\mapsto0,\quad
\check{\mathcal W}(\Omega)\mapsto\check{\mathcal W}\big(\mathfrak P\mathscr A(\Omega/\omega^1)\big).
\end{gather*}

This means that the above operators are recursion operators of the system~$\mathcal S$.
\end{proof}

\begin{remark}
The action of the Teshukov's recursion operator~$\mathfrak R_{\rm T}$ on symmetries of the system~\eqref{eq:IDFMDiagonalizedSystem}
coincides with that of the recursion operator
$\frac12 \mathfrak R_{1,1}-\frac12\mathfrak R_{2,1}+\mathfrak R_{3,1}$.
\end{remark}

One can also find nonlocal recursion operators for the system~\eqref{eq:IDFMDiagonalizedSystem}.
We construct an example of such an operator.
Let~$\eta=(\eta^1,\eta^2,\eta^3)^{\mathsf T}$ be an arbitrary solution of the system~\eqref{eq:IDFMLinearizedSystem}.
Consider a first-order pseudo-differential operator~$\mathfrak R_4$ acting nonlocally on~$\eta$ as
\begin{gather}\label{eq:IDFMRecursionOperator}
\mathfrak R_4\eta=A(\ri)\mathrm D_x\eta+B(\ri,\ri_x)\eta+C(\ri,\ri_x)Y,
\end{gather}
where~$A=(A^{ij})$ and~$B=(B^{ij})$ are smooth $3\times3$ matrix functions of their arguments,
$C$~is a 3-component column of smooth functions of~$(\ri,\ri_x)$
and~$Y$ is the potential associated with the conserved current
$\big(\eta^1+\eta^2, V^1\eta^1+V^2\eta^2\big)$ of the system~\eqref{eq:IDFMLinearizedSystem},
which is the linearized counterpart of the conserved current $\big(\ri^1+\ri^2,\frac12(\ri^1+\ri^2)^2+\ri^1-\ri^2\big)$
of the system~\eqref{eq:IDFMDiagonalizedSystem}.
In other words, the potential~$Y$ is defined by the system
\begin{gather}\label{eq:IDFMPseudoDifferentialPart}
\mathrm D_tY=-V^1\eta^1-V^2\eta^2,\quad \mathrm D_xY=\eta^1+\eta^2.
\end{gather}
By definition, the operator~$\mathfrak R_4$ is a recursion operator of the system~\eqref{eq:IDFMDiagonalizedSystem}
if for an arbitrary solution~$\eta$ of the system~\eqref{eq:IDFMLinearizedSystem},
$\mathfrak R_4\eta$ is a solution of the same system.
We successively substitute the ansatz~\eqref{eq:IDFMRecursionOperator} for~$\mathfrak R_4\eta$,
the expressions~\eqref{eq:IDFMPseudoDifferentialPart} for~$\mathrm D_tY$ and~$\mathrm D_xY$
and the expressions for $\mathrm D_t\eta$ in view of the system~\eqref{eq:IDFMLinearizedSystem}
into the system~\eqref{eq:IDFMLinearizedSystem} for~$\mathfrak R_4\eta$,
which leads to the system
\begin{gather*}{}
(A^{kj}_{\ri^l}\mathrm D_x\eta^j+B^{kj}_{\ri^l}\eta^j+C^k_{\ri^l}Y)(V^k-V^l)\ri^l_x
+(B^{kj}_{\ri^l_x}\eta^j+C^k_{\ri^l_x}Y)\big((V^k-V^l)\ri^l_{xx}-(\ri^1_x+\ri^2_x)\ri^l_x\big)
\\\qquad{}
-A^{kj}\mathrm D_x\big(V^j\mathrm D_x\eta^j+(\eta^1+\eta^2)\ri^j_x\big)
-B^{kj}\big(V^j\mathrm D_x\eta^j+(\eta^1+\eta^2)\ri^j_x\big)
-C^k(V^1\eta^1+V^2\eta^2)
\\\qquad{}
+V^k\big(A^{kj}\mathrm D^2_x\eta^j+B^{kj}\mathrm D_x\eta^j+C^k(\eta^1+\eta^2)\big)
\\\qquad{}
+\ri^k_x\big(A^{1j}\mathrm D_x\eta^j+B^{1j}\eta^j+C^1Y+A^{2j}\mathrm D_x\eta^j+B^{2j}\eta^j+C^2Y\big)=0.
\end{gather*}
Here and in what follows the indices~$j$, $k$, $k'$ and~$l$ run from 1 to~3,
we assume summation with respect to the repeated indices~$j$ and~$l$, and there is no summation over~$k$ and~$k'$.
The splitting of the obtained system with respect to~$\mathrm D^2_x\eta^{k'}$, $\mathrm D_x\eta^{k'}$, $\eta^{k'}$ and~$Y$
yields the system of determining equations for entries of~$A$, $B$ and~$C$,
\begin{gather*}
A^{kk'}(V^k-V^{k'})=0,
\\[1ex]
A^{kk'}_{\ri^l}(V^k-V^l)\ri^l_x-A^{kk'}(\ri^1_x+\ri^2_x)
-(\delta^1_{k'}+\delta^2_{k'})A^{kj}\ri^j_x+(V^k-V^{k'})B^{kk'}\!+\ri^k_x(A^{1k'}\!{+}A^{2k'})=0,
\\[1ex]
B^{kk'}_{\ri^l}(V^k-V^l)\ri^l_x
+B^{kk'}_{\ri^l_x}\big((V^k-V^l)\ri^l_{xx}-(\ri^1_x+\ri^2_x)\ri^l_x\big)
+\ri^k_x(B^{1k'}\!+B^{2k'})
\\ \qquad{}
-(\delta^1_{k'}+\delta^2_{k'})\big(A^{kj}\ri^j_{xx}+B^{kj}\ri^j_x-(V^k-V^{k'})C^k\big)=0,
\\[1ex]
C^k_{\ri^l}(V^k-V^l)\ri^l_x+C^k_{\ri^l_x}\big((V^k-V^l)\ri^l_{xx}-(\ri^1_x+\ri^2_x)\ri^l_x\big)+\ri^k_x(C^1+C^2)=0,
\end{gather*}
solving which, we prove the following proposition.

\begin{proposition}
The system~\eqref{eq:IDFMDiagonalizedSystem} admits the formally pseudo-differential recursion
operator~$\mathfrak R_4$ acting on a symmetry characteristic~$\eta$ as
\begin{gather*}
\mathfrak R_4\eta=B\eta+CY,\quad
\text{where}\quad
B=\begin{pmatrix}
2 &  0 & 0\\
0 & -2 & 0\\
0 &  0 & 0
\end{pmatrix},\quad
C=\begin{pmatrix}
\ri^1_x\\ \ri^2_x\\ \ri^3_x
\end{pmatrix},
\end{gather*}
and $Y$ is the potential of the system~\eqref{eq:IDFMLinearizedSystem}
that is defined by~\eqref{eq:IDFMPseudoDifferentialPart}.
\end{proposition}

\section{Conclusion}\label{sec:IDFMConclusions}

To study the diagonalized form~\eqref{eq:IDFMDiagonalizedSystem} of the system~$\mathcal S$, we heavily rely on its two primary features.
The first feature is the degeneracy of~$\mathcal S$ in the sense
that this system is not genuinely nonlinear with respect to~$\ri^3$
and, moreover, it is partially decoupled since the first two equations of~$\mathcal S$ do not involve~$\ri^3$.
To take into account the degeneracy efficiently, we introduce the modified coordinates on~$\mathcal S^{(\infty)}$,
where derivatives of~$\ri^3$ are replaced by~$\omega$'s constituting a functional basis of the kernel of the operator $\mathscr B$.
This operator is nothing else but the differential operator in the total derivatives that is associated with the equation on~$\ri^3$.
From another perspective, the infinite tuple of~$\omega$'s,
$\omega^0:=\ri^3$, $\omega^{\kappa+1}:=\mathscr A\omega^\kappa$, $\kappa\in\mathbb N_0$,
can be seen to be generated
by the differential operator~$\mathscr A:={\rm e}^{\ri^2-\ri^1}\mathscr D_x$,
commuting with~$\mathscr B$, $[\mathscr A,\mathscr B]=0$, cf.~\cite{Doyle1994}.
The introduction of the modified coordinates essentially simplifies
computations of all kinds of symmetry-like objects for the system~$\mathcal S$.
Due to the partial decoupling of the system~$\mathcal S$, we recognize its essential subsystem~$\mathcal S_0$
constituted by the equations~\eqref{eq:IDFMDEq1}, \eqref{eq:IDFMDEq2}.
The second primary feature of~$\mathcal S$ is the linearization of~$\mathcal S_0$
to the (1+1)-dimensional Klein--Gordon equation,
which was thoroughly studied from the point of view of generalized symmetries and conservation laws
in~\cite{OpanasenkoBihloPopovychSergyeyev2020}.

In turn, these features allow us to describe symmetry-like objects for the system~$\mathcal S$
by working within the following general approach.
For a given kind of symmetry-like objects for~$\mathcal S$, we show
that the chosen space~$U$ of canonical representatives of equivalence classes of such objects
is the sum of three subspaces, $U=U_1+U_2+U_3$.
One of them, say, $U_1$, stems from the degeneracy of~$\mathcal S$,
and thus its elements are parameterized by an arbitrary function of a finite but unspecified number of~$\omega$'s.
The other two subspaces, $U_2$ and~$U_3$, are related to the linearization of~$\mathcal S_0$
to the (1+1)-dimensional Klein--Gordon equation~\eqref{eq:IDFMSupersystemA}.
Singling out these two subspaces is induced by decomposing the objects of the same kind for
the Klein--Gordon equation as sums of those underlaid by linear superposition of solutions of~\eqref{eq:IDFMSupersystemA}
and those associated with linear generalized symmetries of~\eqref{eq:IDFMSupersystemA}.
This is why the elements of the subspaces~$U_2$ and~$U_3$
are parameterized by an arbitrary solution of the (1+1)-dimensional Klein--Gordon equation
and by characteristics of reduced linear generalized symmetries of this equation, respectively.
Although $(U_1+U_2)\cap U_3=\{0\}$, the sum $U_1+U_2+U_3$ is not direct
since the subspaces~$U_1$ and~$U_2$ are not disjoint, and their intersection is one-dimensional.

The first kind of objects we exhaustively describe for the system~$\mathcal S$ is given by generalized symmetries.
Not all generalized symmetries of the Klein--Gordon equation~\eqref{eq:IDFMSupersystemA}
have counterparts among generalized symmetries of the system~$\mathcal S$,
which was also noted in~\cite{OpanasenkoBihloPopovychSergyeyev2020} for first-order generalized symmetries.
The most difficult problem here, which is solved in Lemma~\ref{lem:IDFM:GenSymsOfSupersystem2B},
is to single out the subalgebra~$\mathfrak A$ of canonical representatives of generalized symmetries
of the Klein--Gordon equation~\eqref{eq:IDFMSupersystemA} that have such counterparts.
A complementary subalgebra to~$\mathfrak A$ is
$\bar{\mathfrak A}=\langle\,(\mathscr J^\kappa q)\p_q,\,\kappa\in\mathbb N\,\rangle$.
We conjecture that elements of~$\bar{\mathfrak A}$ have counterparts
among nonlocal, or specifically potential, symmetries of the system~$\mathcal S$.
To show this, we plan to study certain Abelian coverings and potential symmetries
of the system~$\mathcal S$ and of the Klein--Gordon equation~\eqref{eq:IDFMSupersystemA}.
We expect that the main role in this consideration will be played
by the conservation laws of the Klein--Gordon equation~\eqref{eq:IDFMSupersystemA}
with characteristics of the form $\mathscr J^\kappa {\rm e}^{y+z}$, $\kappa\in\mathbb N_0$,
and by their counterparts for the system~$\mathcal S$.

Considering cosymmetries and local conservation laws,
we do not need to make the selection among those for the Klein--Gordon equation~\eqref{eq:IDFMSupersystemA}
since all of them have counterparts for the system~$\mathcal S$.
For conservation laws, this follows directly from the general assertion proved in~\cite[Theorem~1]{KunzingerPopovych2008}.
Amongst cosymmetries, local conservation laws and their characteristics,
the complete description of the space of cosymmetries for the system~$\mathcal S$ is the most complicated 
since it requires utilizing a couple of nontrivial tricks within the framework of our general approach.

To construct the space of local conservation laws of~$\mathcal S$,
we have to make use of the direct method \cite{PopovychIvanova2005b,Wolf2002}
whose essence is the direct construction of conserved currents canonically representing conservation laws
using the definitions of conserved currents and of their equivalence.
The standard approach~\cite{Bocharov1999} based on singling out conservation-law characteristics among cosymmetries
is not effective for the system~$\mathcal S$
since its application to~$\mathcal S$ leads to too cumbersome computations.
At the same time, we still need to know conservation-law characteristics for the system~$\mathcal S$, in particular,
to look for special-feature conservation laws, like low-order and translation-invariant ones.
The known formula~\cite[Proposition 7.41]{Tsujishita1982} relating characteristics of conservation laws
of systems in the extended Kovalevskaya form~\cite[Definition~4]{PopovychBihlo2020}
to densities of these conservation laws gives suitable expressions only
for characteristics of conservation laws from the second family of Theorem~\ref{thm:IDFM:CLs},
which are of order zero.
The other two families should be tackled differently.
For the first family, we in fact derive an analogue of the above formula in terms of the operator~$\mathcal A$
using the formal integration by parts.
Characteristics of conservation laws from the third family are constructed from their counterparts
being variational symmetries of the Klein--Gordon equation~\eqref{eq:IDFMSupersystemA}.
We also prove that under the action of generalized symmetries of the system~$\mathcal S$
on its space of conservation laws,
a generating set of conservation laws of this system
is constituted by two zeroth-order conservation laws.
One of them belongs to and generates the first subspace of conservation laws,
which is related to the degeneracy of~$\mathcal S$.
The other is the counterpart of a single generating conservation law
of the Klein--Gordon equation~\eqref{eq:IDFMSupersystemA}.
It belongs to the third subspace of conservation laws of~$\mathcal S$
but generates the second subspace as well.
The claim on generation of the entire third subspace
is unexpected since only a proper part of linear generalized symmetries
of the Klein--Gordon equation~\eqref{eq:IDFMSupersystemA}
are naturally mapped to generalized symmetries of~$\mathcal S$
but the amount of the images still suffices for generating all required conservation laws.
%\looseness=1

Interrelating generalized symmetries and cosymmetries,
we construct a family of compatible Hamiltonian operators for the system~$\mathcal S$
parameterized by an arbitrary function of~$\ri^3$,
and a Hamiltonian operator from this family is degenerate if the corresponding value of the parameter function vanishes at some point.
This fundamentally differs from the case of genuinely nonlinear hydrodynamic-type systems,
for which the number of local Hamiltonian operators of hydrodynamic type is known not to exceed~$n+1$,
where $n$ is the number of dependent variables, see~\cite{FerapontovPavlov1991}.
Note that the conjecture from~\cite{KerstenKrasilshchikVerbovetsky2006}
that skew-symmetric Noether operators for non-scalar systems of first-order evolution equations are Hamiltonian ones
holds for Noether operators of~$\mathcal S$  with entries of the form~\eqref{eq:AnsatzForNoetherOps}.
%\looseness=1

Finally, having the comprehensive description of the algebra of generalized symmetries of the system~$\mathcal S$ at our disposal,
we find \textit{ad hoc} broad families of local recursion operators, which are presented in Theorem~\ref{thm:IDFMrecursionOps}.
The system~$\mathcal S$ admits the canonical Teshukov's recursion operator~$\mathfrak R_{\rm T}$
but this operator is equivalent to a linear combination of the three simplest local recursion operators from Theorem~\ref{thm:IDFMrecursionOps}.
We also construct a nonlocal recursion operator of~$\mathcal S$.
It is clear that one can construct many such operators, in particular,
using the relation of the system~$\mathcal S$ to the Klein--Gordon equation~\eqref{eq:IDFMSupersystemA},
which will be a subject of our further studies.

We should like to emphasize that the local description of the solution set of the system~$\mathcal S$
in Theorem~\ref{thm:IDFMCompleteSolutioN} is implicit and
involves the general solution of the (1+1)-dimensional Klein--Gordon equation.
This is why it is difficult to further use this description,
and thus it is still worthwhile to comprehensively study the system~$\mathcal S$
within the framework of symmetry analysis of differential equations.

\looseness=-1
As the essential subsystem~$\mathcal S_0$ coincides with the diagonalized form of the system
describing one-dimensional isentropic gas flows with constant sound speed~\cite[Section~2.2.7, Eq.~(16)]{RozhdestvenskiiJanenko1983},
symmetry-like objects of~$\mathcal S_0$ deserve a separate consideration
but in fact they are implicitly described in the present paper.
In contrast to the system~$\mathcal S$,
all the quotient spaces of symmetry-like objects of the subsystem~$\mathcal S_0$ are isomorphic to their counterparts
for the system~\eqref{eq:IDFMSupersystemA},~\eqref{eq:IDFMSupersystemC} and thus to their counterparts
for the Klein--Gordon equation~\eqref{eq:IDFMSupersystemA}.
Therefore, to construct an algebra of canonical representatives of generalized symmetries for the subsystem~$\mathcal S_0$,
we take the respective algebra for the equation~\eqref{eq:IDFMSupersystemA}
and follow the procedure given in the first paragraph of the proof of Theorem~\ref{thm:IDFMGenSyms},
just ignoring the $\ri^3$-components in the point transformation~\eqref{eq:IDFMInverseToTransReducingToKGEq}
and in the vector field~$\tilde Q$.
As a result, we obtain that the quotient algebra of generalized symmetries of the subsystem~$\mathcal S_0$
is naturally isomorphic to the algebra spanned by the generalized vector fields
\begin{gather*}
\big(x-(\ri^1+\ri^2+1)t\big)\ri^1_x\p_{\ri^1}+\big(x-(\ri^1+\ri^2-1)t\big)\ri^2_x\p_{\ri^2},
\quad
{\rm e}^{(\ri^2-\ri^1)/2}\big(\Gamma\ri^1_x\p_{\ri^1}+\tilde{\mathscr D}_z\Gamma\ri^2_x\p_{\ri^2}\big),
\\
{\rm e}^{(\ri^2-\ri^1)/2}\left(
(\Phi+2\Phi_{\ri^1})\ri^1_x\p_{\ri^1}
+(\Phi-2\Phi_{\ri^2})\ri^2_x\p_{\ri^2}\right),
\end{gather*}
where the parameter function~$\Phi=\Phi(\ri^1,\ri^2)$ runs through the solution set of
the Klein--Gordon equation $\Phi_{\ri^1\ri^2}=-\Phi/4$,
$\Gamma$ runs through the set
$
\{\tilde{\mathscr J}^\kappa\tilde q,\,
\tilde{\mathscr D}_y^\iota\tilde{\mathscr J}^\kappa\tilde q,\,
\tilde{\mathscr D}_z^\iota\tilde{\mathscr J}^\kappa\tilde q,\,
\kappa\in\mathbb N_0,\,\iota\in\mathbb N\}
$
with
\begin{gather*}
\tilde{\mathscr D}_y:=-\frac1{\ri^1_x}\big(\mathscr D_t+(\ri^1+\ri^2-1)\mathscr D_x\big),\quad
\tilde{\mathscr D}_z:=-\frac1{\ri^2_x}\big(\mathscr D_t+(\ri^1+\ri^2+1)\mathscr D_x\big),
\\
\tilde{\mathscr J}:=\frac{\ri^1}2\tilde{\mathscr D}_y+\frac{\ri^2}2\tilde{\mathscr D}_z,\quad
\tilde q:={\rm e}^{(\ri^1-\ri^2)/2}\big(x-(\ri^1+\ri^2+1)t\big),
\end{gather*}
and instead of the complete operators~$\mathscr D_t$ and~$\mathscr D_x$ defined in Section~\ref{sec:IDFMPreliminaries},
one should use their restrictions to $(\ri^1,\ri^2)$,
\[
\mathscr D_x:=\p_x+\sum_{\kappa=0}^\infty\sum_{i=1}^2\ri^i_{\kappa+1}\p_{\ri^i_\kappa},\quad
\mathscr D_t:=\p_t-\sum_{\kappa=0}^\infty\sum_{i=1}^2\mathscr D_x^\kappa(V^i\ri^i_1)\p_{\ri^i_\kappa}.
\]
The descriptions of cosymmetries and conservation laws of~$\mathcal S_0$
are derived from those for the system~$\mathcal S$ by
excluding the first families of cosymmetries and conservation laws,
which are related to the degeneracy of~$\mathcal S$,
in Theorems~\ref{thm:IDFM:Cosyms} and~\ref{thm:IDFM:CLs}.

\section*{Acknowledgments}
%\bigskip\par\noindent{\bf Acknowledgments.}
The authors thank Galyna Popovych for helpful discussions and interesting comments.
The research of AB and SO was undertaken, in part,
thanks to funding from the Canada Research Chairs program, the NSERC Discovery Grant program and the InnovateNL LeverageR{\&}D program.
The research of SO was supported by the NSERC's Alexander Graham Bell Canada Graduate Scholarship~-- Doctoral Fellowship.
The research of ROP was supported by the Austrian Science Fund (FWF),
projects P25064 and P29177, and by ``Project for fostering collaboration in
science, research and education'' funded by the Moravian-Silesian
Region, Czech Republic. The research of AS was supported in part by the Grant Agency of the Czech Republic (GA \v{C}R)
under grant P201/12/G028.

\footnotesize
%\bibliography{../../opanasenko}
%\end{document}

\end{document}